\documentclass{article}

\usepackage{multirow}
\usepackage{layout}

\usepackage{amsmath}
\usepackage{amsthm}
\usepackage{amssymb}
\usepackage{amsfonts}
\usepackage{mathtools}
\usepackage{siunitx}

\usepackage{enumitem}

\usepackage[flushleft]{threeparttable}
\usepackage{colortbl}
\usepackage{array,booktabs,makecell}

\usepackage{bbm}

\usepackage{url}
\usepackage{color}
\usepackage{graphicx}
\usepackage{verbatim}
\usepackage{subcaption}

\usepackage{natbib}

\usepackage{apptools}
\usepackage{thmtools}
\usepackage{algorithm, algpseudocode}

\usepackage[nice]{nicefrac}

\newcommand{\norm}[1]{\left\| #1 \right\|}

\newcommand{\inp}[2]{\left\langle#1,#2\right\rangle} % inner product

\newcommand{\flr}[1]{\left\lfloor #1\right\rfloor} 
\newcommand{\ceil}[1]{\left\lceil #1\right\rceil} 
\newcommand{\abs}[1]{\left| #1 \right|}

\newcommand{\parens}[1]{\left( #1 \right)}
\newcommand{\brac}[1]{\left\{ #1 \right\}}

\newcommand{\cA}{\mathcal{A}}

\newcommand{\cF}{\mathcal{F}}

\newcommand{\cI}{\mathcal{I}}
\newcommand{\cM}{\mathcal{M}}
\newcommand{\cN}{\mathcal{N}}
\newcommand{\cO}{\mathcal{O}}

\newcommand{\cS}{\mathcal{S}}
\newcommand{\cU}{\mathcal{U}}

\newcommand{\mA}{\mathbf{A}}

\newcommand{\mI}{\mathbf{I}}

\newcommand{\mS}{\mathbf{S}}

\newcommand{\del}[1]{}

\newcommand{\R}{\mathbb{R}} % reals
\newcommand{\N}{\mathbb{N}} % reals
\newcommand{\eqdef}{:=} %\vcentcolon

\newcommand{\Prob}[1]{\mathbb{P}\left(#1\right)} % probability

\newcommand{\Exp}[1]{{\mathbb{E}}\left[#1\right]}

\usepackage{tcolorbox}
\usepackage{pifont}
\definecolor{mydarkgreen}{RGB}{39,130,67}
\definecolor{mydarkred}{RGB}{192,47,25}
\definecolor{mydarkorange}{RGB}{39,130,67}
\definecolor{bgcolor}{rgb}{0.8,1,1}
\newcommand{\green}{\color{mydarkgreen}}
\newcommand{\red}{\color{mydarkred}}

\usepackage[scaled=0.86]{helvet}
\newcommand{\algname}[1]{{\sf #1}}

\newcommand{\algnametiny}[1]{{\sf #1}}
\newcommand{\algnamelarge}[1]{{\sf #1}} % for subsections
\newcommand{\Teq}{t^*}

\makeatletter
\newenvironment{protocol}[1][htb]{%
    \renewcommand{\ALG@name}{Protocol}% Update algorithm name
   \begin{algorithm}[#1]%
  }{\end{algorithm}}
\makeatother

\newtheorem{assumption}{Assumption}
\newtheorem{lemma}{Lemma}

\newtheorem{theorem}{Theorem}

\theoremstyle{plain}

\newtheorem{remark}[theorem]{Remark}

\theoremstyle{definition}

\newtheorem{definition}[theorem]{Definition}

\usepackage{listings}
\usepackage{xcolor}

\definecolor{codegreen}{rgb}{0,0.6,0}
\definecolor{codegray}{rgb}{0.5,0.5,0.5}
\definecolor{codepurple}{rgb}{0.58,0,0.82}
\definecolor{backcolour}{rgb}{0.95,0.95,0.92}

\lstdefinestyle{mystyle}{
    backgroundcolor=\color{backcolour},   
    commentstyle=\color{codegreen},
    keywordstyle=\color{magenta},
    numberstyle=\tiny\color{codegray},
    stringstyle=\color{codepurple},
    basicstyle=\ttfamily\footnotesize,
    breakatwhitespace=false,         
    breaklines=true,                 
    captionpos=b,                    
    keepspaces=true,                 
    numbers=left,                    
    numbersep=5pt,                  
    showspaces=false,                
    showstringspaces=false,
    showtabs=false,                  
    tabsize=2
}

\usepackage[final]{neurips_2024}

\usepackage[utf8]{inputenc} % allow utf-8 input
\usepackage[T1]{fontenc}    % use 8-bit T1 fonts
\usepackage{hyperref}       % hyperlinks
\usepackage{url}            % simple URL typesetting
\usepackage{booktabs}       % professional-quality tables
\usepackage{amsfonts}       % blackboard math symbols
\usepackage{nicefrac}       % compact symbols for 1/2, etc.
\usepackage{microtype}      % microtypography
\usepackage{xcolor}         % colors
\usepackage[export]{adjustbox}

\hypersetup{
  colorlinks   = true, %Colours links instead of ugly boxes
  urlcolor     = blue, %Colour for external hyperlinks
  linkcolor    = {blue!75!black}, %Colour of internal links
  citecolor   = {blue!50!black} %Colour of citations
}

\newcommand\mybar{\kern1pt\rule[-\dp\strutbox]{1pt}{\baselineskip}\kern1pt}

\allowdisplaybreaks

\title{Freya PAGE: First Optimal Time Complexity for Large-Scale Nonconvex Finite-Sum Optimization with Heterogeneous Asynchronous Computations}

\author{%
  Alexander Tyurin \\
  KAUST\thanks{King Abdullah University of Science and Technology, Thuwal, Saudi Arabia},\,\,AIRI\thanks{AIRI, Moscow, Russia},\,\,Skoltech\thanks{Skolkovo Institute of Science and Technology, Moscow, Russia} \\
  \And
  Kaja Gruntkowska \\
  KAUST\footnotemark[1] \\
  \And
  Peter Richt\'{a}rik \\
  KAUST\footnotemark[1] \\
}

\allowdisplaybreaks

\usepackage{todonotes}

\begin{document}

\maketitle

\begin{abstract}

    In practical distributed systems, workers are typically not homogeneous, and due to differences in hardware configurations and network conditions, can have highly varying processing times. We consider smooth nonconvex finite-sum (empirical risk minimization) problems in this setup and introduce a new parallel method, \algname{Freya PAGE}, designed to handle arbitrarily heterogeneous and asynchronous computations. By being robust to ``stragglers'' and adaptively ignoring slow computations, \algname{Freya PAGE} offers significantly improved \emph{time complexity} guarantees compared to all previous methods, including \algname{Asynchronous SGD}, \algname{Rennala SGD}, \algname{SPIDER}, and \algname{PAGE}, while requiring weaker assumptions. The algorithm relies on novel generic stochastic gradient collection strategies with theoretical guarantees that can be of interest on their own, and may be used in the design of future optimization methods. Furthermore, we establish a lower bound for smooth nonconvex finite-sum problems in the asynchronous setup, providing a fundamental time complexity limit. This lower bound is tight and demonstrates the optimality of \algname{Freya PAGE} in the large-scale regime, i.e., when $\sqrt{m} \geq n,$ where $n$ is \# of workers, and $m$ is \# of data samples.

\end{abstract}

\section{Introduction}\label{sec:intro}

In real-world distributed systems used for large-scale machine learning tasks, it is common to encounter device heterogeneity and variations in processing times among different computational units. These can stem from GPU computation delays, disparities in hardware configurations, network conditions, and other factors, resulting in different computational capabilities and speeds across devices \citep{chen2016revisiting, tyurin2023optimal}. As a result, some clients may execute computations faster, while others experience delays or even fail to participate in the training altogether.

Due to the above reasons, we aim to address the challenges posed by device heterogeneity in the context of solving finite-sum nonconvex optimization problems of the form
\begin{align}\label{eq:problem}
    \textstyle \min\limits_{x\in \R^d} \brac{f(x) \eqdef \frac{1}{m} \sum\limits_{i=1}^m f_i(x)},
\end{align}
where $f_i: \R^d \rightarrow \R$ can be viewed as the loss of a machine learning model $x$ on the $i$\textsuperscript{th} example in a training dataset with $m$ samples. Our goal is to find an $\varepsilon$-stationary point, i.e., a (possibly random) point $\hat{x}$ such that $\mathbb{E}[\|\nabla f(\hat{x})\|^2] \leq \varepsilon$.
We focus on the homogeneous distributed setup:
\begin{itemize}
    \item there are $n$ \emph{workers}/\emph{clients}/\emph{devices} able to work in parallel,
    \item each worker has access to stochastic gradients $\nabla f_j$, $j \in [m]$,
    \item worker $i$ calculates $\nabla f_j(\cdot)$ in less or equal to $\tau_i \in [0, \infty]$ seconds for all $i \in [n], j \in [m].$
\end{itemize}
Without loss of generality, we assume that $\tau_1\leq\ldots\leq\tau_n$. One can think of $\tau_i \in [0, \infty]$ as an upper bound on the computation time rather than a fixed deterministic time. 
Looking ahead, iteration complexity can be established even if $\tau_i = \infty$ for all $i \in [n]$ (Theorem~\ref{thm:page_rate}). We also provide results where the bounds $\{\tau_i^k\}$ are dynamic and change with every iteration $k$ (Section~\ref{sec:time_var_times}). For simplicity of presentation, however, we assume that $\tau_i^k = \tau_i$ for $i \in [n], k \geq 0$, unless explicitly stated otherwise.

\subsection{Assumptions}\label{sec:assumptions}

We adopt two weak assumptions, which are standard for the problem \eqref{eq:problem} \citep{fang2018spider}.
\begin{assumption}\label{as:smooth_lower_bdd}
    The function $f$ is $L_-$-smooth and lower-bounded by $f^* \in \R$.
\end{assumption}
\begin{assumption}\label{as:L+}
    $\exists \, L_+ \geq 0$ such that $\frac{1}{m} \sum\limits_{i=1}^{m} \norm{\nabla f_i(x) - \nabla f_i(y)}^2 \leq L_+^2 \norm{x - y}^2\, \forall x,y\in\R^d.$
\end{assumption}
We also consider Assumption~\ref{as:L_pm}. Note that this assumption does not restrict the class of considered functions $\{f_i\}.$
Indeed, if Assumption~\ref{as:L+} holds with $L_{+},$ then Assumption~\ref{as:L_pm} holds with some $L_{\pm} \leq L_{+}$. If one only wants to rely on Assumptions~\ref{as:smooth_lower_bdd} and \ref{as:L+}, it is sufficient to take $L_{\pm} = L_{+}$. However, Assumption~\ref{as:L_pm} enables us to derive sharper rates, since $L_{\pm}$ can be small or even $0$, even if $L_-$ and $L_{+}$ are large \citep{szlendak2021permutation,tyurin2022sharper,kovalev2022optimal}.
\begin{assumption}[Hessian variance \citep{szlendak2021permutation}]\label{as:L_pm}
    There exists $L_{\pm} \geq 0$ such that
    \begin{align*}
        \textstyle \frac{1}{m} \sum\limits_{i=1}^m \norm{\nabla f_i(x) - \nabla f_i(y) - \left(\nabla f(x) - \nabla f(y)\right)}^2 \leq L^2_{\pm} \norm{x-y}^2
        \qquad \forall x, y \in \R^d.
    \end{align*}
\end{assumption}

\subsection{Gradient oracle complexities}\label{sec:sync_vr}

Iterative algorithms are traditionally evaluated based on their \emph{gradient complexity}. Let us present a brief overview of existing theory.
The classical result of Gradient Descent (\algname{GD}) says that in the smooth nonconvex regime, the number of oracle calls needed to solve problem~\eqref{eq:problem} is $\cO(m\varepsilon^{-1})$ because \algname{GD} converges in $\cO(\varepsilon^{-1})$ iterations, and calculates the full gradient $\nabla f = \nicefrac{1}{m} \sum_{i=1}^m \nabla f_i$ in each iteration.
This was improved to $\cO(m + m^{\nicefrac{2}{3}} \varepsilon^{-1})$ by several \emph{variance-reduced} methods, including \algname{SVRG} and \algname{SCSG} \citep{allen2016variance, reddi2016stochastic, lei2017non, horvath2019nonconvex}.
Since then, various other algorithms, such as \algname{SNVRG}, \algname{SARAH}, \algname{SPIDER}, \algname{SpiderBoost}, \algname{PAGE} and their variants, have been developed \citep{fang2018spider, wang2019spiderboost, nguyen2017sarah, li2021page, zhou2020stochastic, horvath2022adaptivity}. These methods achieve a gradient complexity of $\cO(m + \sqrt{m} \varepsilon^{-1})$, matching the lower bounds \citep{fang2018spider,li2021page}.

That said, in practical scenarios, what often truly matters is the \emph{time complexity} rather than the \emph{gradient complexity} \citep{tyurin2023optimal}. Although the latter metric serves as a natural benchmark for sequential methods, it seems ill-suited in the context of parallel methods. %Indeed, differences in processing speeds across devices can significantly impact the performance of parallel algorithms.

\subsection{Some previous time complexities}
\label{sec:previous_compl}
Let us consider some examples to provide intuition about time complexities for problem \eqref{eq:problem}.

\textbf{\algname{GD} with $1$ worker (\algname{Hero GD}).} In principle, each worker can solve the problem on their own. Hence, one approach would be to select the fastest client (assuming it is known) and delegate the task to them exclusively.
A well-known result says that for $L_-$-smooth objective function $f$ (Assumption~\ref{as:smooth_lower_bdd}), \algname{GD} converges in $\delta^0L_-\varepsilon^{-1}$ iterations, where $\delta^0 \eqdef f(x^0)-f^*,$ and $x^0$ is the starting point. Since at each iteration the method computes $m$ gradients $\nabla f_i(\cdot)$, $i\in[m]$, the time required to find an $\varepsilon$-stationary point is $\delta^0L_-\varepsilon^{-1} \times m \tau_1$ seconds.

\textbf{\algname{GD} with $n$ workers and equal data allocation (\algname{Soviet GD}).} The above strategy leaves the remaining $n-1$ workers idle, and thus potentially useful computing resources are wasted. A common approach is to instead divide the data into $n$ equal parts and assign one such part to each worker, so that each has to compute $\nicefrac{m}{n}$ gradients (assuming for simplicity that $m$ is divisible by $n$). Since at each iteration the strategy needs to wait for the slowest worker, the total time is $\delta^0L_-\varepsilon^{-1} \times \nicefrac{m \tau_n}{n}$. Depending on the relationship between $\tau_1$ and $\nicefrac{\tau_n}{n}$, this could be more efficient or less efficient compared to \algname{Hero~GD}. This shows that the presence of stragglers can eliminate the potential speedup expected from parallelizing the training \citep{dutta2018slow}.

\textbf{\algname{SPIDER}/\algname{PAGE} with $1$ worker or $n$ workers and equal data allocation (\algname{Hero PAGE} and \algname{Soviet PAGE}).}
As mentioned in Section~\ref{sec:sync_vr}, \algname{SPIDER}/\algname{PAGE} can have better \emph{gradient complexity} guarantees than \algname{GD}. Using the result of \citet{li2021page}, the equal data allocation strategy with $n$ workers leads to the time complexity of
\begin{align}
    \label{eq:page_time}
    \textstyle T_{\textnormal{Soviet PAGE}} \eqdef \Theta\left(\tau_n \max\left\{\frac{m}{n}, 1\right\} + \tau_n \frac{\delta^0 \max\{L_-, L_{\pm}\}}{\varepsilon}\max\left\{\frac{\sqrt{m}}{n}, 1\right\} \right)
\end{align}
seconds. We refer to this method as \algname{Soviet PAGE}. In practical regimes, when $\varepsilon$ is small and $L_- \approx L_{\pm}$, this complexity can be $\sqrt{m}$ better than that of \algname{GD}. Running \algname{PAGE} on the fastest worker (which we will call \algname{Hero PAGE}), we instead get the time complexity $T_{\textnormal{Hero PAGE}}~\eqdef~\Theta\left(\tau_1 m + \tau_1 \nicefrac{\delta^0}{\varepsilon} \sqrt{m}\right).$

Given these examples, the following question remains unanswered: what is the best possible time complexity in our setting? This paper aims to answer this question.

\begin{table}
\caption{\small Comparison of the \emph{worst-case time complexity} guarantees of methods that work with asynchronous computations in the setup from Section~\ref{sec:intro} (up to smoothness constants). We assume that $\tau_{i} \in [0, \infty]$ is the bound on the times required to calculate one stochastic gradient $\nabla f_j$ by worker $i,$ $\tau_1\leq\ldots\leq\tau_n,$ and $m \geq n \log n.$ Abbr: $\delta^0 \eqdef f(x^0) - f^*,$ $m = $ \# of data samples, $n = $ \# of workers, $\varepsilon = $ error tolerance.}
\label{tbl:main}
\tiny
\centering 
\begin{adjustbox}{width=\columnwidth,center}
\begin{threeparttable}
\begin{tabular}[t]{ccc}
    \toprule
    \bf  Method & \bf Worst-Case Time Complexity & \bf Comment \\
    \midrule
    \makecell{\algnametiny{Hero GD} \, \big(\algnametiny{Soviet GD}\big)} & \makecell{$\tau_1 m \frac{\delta^0}{\varepsilon}$ \quad \big($\tau_n \frac{m}{n} \frac{\delta^0}{\varepsilon}$\big)} & Suboptimal \\
    \midrule
    \makecell{\algnametiny{Hero PAGE} \, \big(\algnametiny{Soviet PAGE}\big) \\ \citep{li2021page}} & \makecell{$\tau_1 m + \tau_1 \frac{\delta^0}{\varepsilon}\sqrt{m}$ \quad \big($\tau_n \frac{m}{n} + \tau_n \frac{\delta^0}{\varepsilon} \frac{\sqrt{m}}{n}$\big)} & Suboptimal \\
    \midrule
    \makecell{\algnametiny{SYNTHESIS} \\ \citep{liu2022synthesis}} & \makecell{---} &  \makecell{Limitations: \\ bounded gradient assumption, \\ calculates the full gradients\textsuperscript{{\color{blue}(a)}}, \\ suboptimal.\textsuperscript{{\color{blue}(b)}}} \\
    \midrule
    \makecell{\algnametiny{Asynchronous SGD} \\ \citep{koloskova2022sharper} \\ \citep{mishchenko2022asynchronous}} & \makecell{$\frac{\delta^0}{\varepsilon} \parens{\parens{\sum\limits_{i=1}^n \frac{1}{\tau_i}}^{-1} \left(\frac{\sigma^2}{\varepsilon}+n\right)}$} &  \makecell{Limitations: \\ $\sigma^2$--bounded variance assumption, \\ suboptimal when $\varepsilon$ is small.} \\
    \midrule
    \makecell{\algnametiny{Rennala SGD} \\ \citep{tyurin2023optimal}} & \makecell{$\frac{\delta^0}{\varepsilon} \min\limits_{j\in[n]} \parens{\parens{\sum\limits_{i=1}^j \frac{1}{\tau_i}}^{-1} \left(\frac{\sigma^2}{\varepsilon}+j\right)}$} &  \makecell{Limitations: \\ $\sigma^2$--bounded variance assumption, \\ suboptimal when $\varepsilon$ is small.} \\
    \midrule
    \cellcolor{bgcolor} \begin{tabular}{c}\algnametiny{Freya PAGE} \\ (Theorems~\ref{cor:nice_upper_bound} and \ref{cor:nice_upper_bound_better})\end{tabular} & \cellcolor{bgcolor} \begin{tabular}{c} $\min\limits_{j\in[n]} \parens{\parens{\sum\limits_{i=1}^j \frac{1}{\tau_i}}^{-1} (m+j)}$ \\ $\qquad\qquad\qquad + \frac{\delta^0}{\varepsilon} \min\limits_{j\in[n]} \parens{\parens{\sum\limits_{i=1}^j \frac{1}{\tau_i}}^{-1} (\sqrt{m}+j)}$\textsuperscript{{\color{blue}(c)}} \end{tabular} & \cellcolor{bgcolor} \begin{tabular}{c}Optimal in the large-scale regime, \\ i.e., $\sqrt{m} \geq n$ (see Section~\ref{sec:lower_bound_simple}) \end{tabular}\\
    \midrule
    \midrule
    \cellcolor{bgcolor} \begin{tabular}{c} Lower bound \\ (Theorem \ref{thm:lower_bound_simple}) \end{tabular} & \cellcolor{bgcolor} \begin{tabular}{c} $\min\limits_{j\in[n]} \parens{\parens{\sum\limits_{i=1}^j \frac{1}{\tau_i}}^{-1} (m+j)} $ \\ $\qquad\qquad\qquad + \frac{\delta^0}{\sqrt{m} \varepsilon} \min\limits_{j\in[n]} \parens{\parens{\sum\limits_{i=1}^j \frac{1}{\tau_i}}^{-1} (m+j)}$ \end{tabular} & \cellcolor{bgcolor} \begin{tabular}{c}--- \end{tabular}\\
    \midrule
    \multicolumn{3}{c}{\makecell{\algnametiny{Freya PAGE} has \emph{universally} better guarantees than all previous methods: the dependence on $\varepsilon$ is $\cO\left(\nicefrac{1}{\varepsilon}\right)$ (unlike \algnametiny{Rennala SGD} and \algnametiny{Asynchronous SGD}), \\ the dependence on $\{\tau_i\}$ is harmonic-like and robust to slow workers (robust to $\tau_n \to \infty$) (unlike \algnametiny{Soviet PAGE} and \algnametiny{SYNTHESIS}), \\
    the assumptions are weak, and the time complexity of \algnametiny{Freya PAGE} is optimal when $\sqrt{m} \geq n$.}}\\
    \bottomrule
    \end{tabular}
    \begin{tablenotes}
    \scriptsize
    \item [{\color{blue}(a)}] In Line $3$ of their Algorithm $3$, they calculate the full gradient, assuming that it can be done for free and not explaining how.
    \item [{\color{blue}(b)}] Their convergence rates in Theorems $1$ and $3$ depend on a bound on the delays $\Delta,$ which in turn depends on the performance of the slowest worker. Our method does not depend on the slowest worker if it is too slow (see Section~\ref{sec:discussion}), which is required for optimality.
    \item [{\color{blue}(c)}] We prove better time complexity in Theorem~\ref{thm:opt_S_main}, but this result requires the knowledge of $\{\tau_i\}$ in advance,  unlike Theorems~\ref{cor:nice_upper_bound} and \ref{cor:nice_upper_bound_better}.
\end{tablenotes}
\end{threeparttable}
\end{adjustbox}
\end{table}

\section{Contributions}

We consider the finite-sum optimization problem \eqref{eq:problem} under weak assumptions and develop a new method, \algname{Freya PAGE}. 
The method works with arbitrarily heterogeneous and asynchronous computations on the clients without making \emph{any} assumptions about the bounds on the processing times~${\tau_i}$.
We show that the \emph{time complexity} of \algname{Freya PAGE} is provably better than that of all previously proposed synchronous/asynchronous methods (Table~\ref{tbl:main}). Moreover, we prove a lower bound that guarantees optimality of \algname{Freya PAGE} in the large-scale regime ($\sqrt{m} \geq n$). 
The algorithm leverages new computation strategies, \algname{ComputeGradient} (Alg.~\ref{algorithm:ga_asynch_main_new}) and \algname{ComputeBatchDifference} (Alg.~\ref{algorithm:ga_asynch_main_diff}), which are generic and can be used in any other asynchronous method. These strategies enable the development of our new \algname{SGD} method (\algname{Freya SGD}); see Sections~\ref{sec:using_strat} and \ref{sec:freya_sgd}.
Experiments from Section~\ref{sec:experiments} on synthetic optimization problems and practical logistic regression tasks support our theoretical results.

\section{The Design of the New Algorithm}
\label{sec:design}

It is clear that to address the challenges arising in the setup under consideration and achieve optimality, a distributed algorithm has to adapt to and effectively utilize the heterogeneous nature of the underlying computational infrastructure. With this in mind, we now present a new algorithm, \algname{Freya PAGE}, that can efficiently coordinate and synchronize computations across the $n$ devices, accommodating arbitrarily varying processing speeds, while mitigating the impact of slow devices or processing delays on the overall performance of the system.

\begin{algorithm}[t]
    \caption{\algname{Freya PAGE}}
    \label{algorithm:rennala_page_gen}
    \begin{algorithmic}[1]
    \State \textbf{Parameters:} starting point $x^0 \in \R^d$, learning rate $\gamma > 0$, minibatch size $S\in\N$, probability $p \in (0, 1]$, initialization $g^0 = \nabla f(x^0)$ using ${\green \textnormal{\algname{ComputeGradient}}(x^{0})}$ \quad (Alg.~\ref{algorithm:ga_asynch_main_new})
    \For{$k = 0, 1, \ldots, K-1$}
        \State $x^{k+1} = x^k - \gamma g^k$
        \State Sample $c^k \sim \text{Bernoulli}(p)$
        \If{$c^k=1$} \hfill{\color{gray} (with probability $p$)}
            \State {\green $\nabla f(x^{k+1}) = \textnormal{\algname{ComputeGradient}}(x^{k+1})$} \hfill (Alg.~\ref{algorithm:ga_asynch_main_new}) \label{line:grad}
            \State $g^{k+1} = \nabla f(x^{k+1})$
        \Else \hfill{\color{gray} (with probability $1-p$)}
            \State {\green $\frac{1}{S} \sum\limits_{i \in \cS^k}\left(\nabla f_i(x^{k+1}) - \nabla f_i(x^{k})\right) = \textnormal{\algname{ComputeBatchDifference}}(S, x^{k+1}, x^k)$} \hfill (Alg.~\ref{algorithm:ga_asynch_main_diff}) \label{line:batch}
            \State $g^{k+1} = g^k + \frac{1}{S} \sum\limits_{i \in \cS^k}\left(\nabla f_i(x^{k+1}) - \nabla f_i(x^{k})\right)$ 
        \EndIf
    \EndFor
    \end{algorithmic}
    {\small \hspace*{0.5cm} (note): $\cS^k$ is a set of i.i.d. indices that are sampled from $[m]$, \emph{uniformly with replacement}, $\abs{\cS^k} = S$}
\end{algorithm}

\algname{Freya PAGE} is formalized in Algorithm~\ref{algorithm:rennala_page_gen}. 
The update rule is just the regular \algname{PAGE} \citep{li2021page} update: at each iteration, with some (typically small) probability $p$, the algorithm computes the full gradient $\nabla f(x^{k+1})$, and otherwise, it samples a minibatch $\cS^k$ of size $S$ and reuses the gradient estimator $g^k$ from the previous iteration, updated by the cheaper-to-compute adjustment $\frac{1}{S} \sum_{i\in \cS^k} \parens{\nabla f_i(x^{k+1}) - \nabla f_i(x^k)}$. 

Within Algorithm~\ref{algorithm:rennala_page_gen}, at each iteration we call one of two subroutines: \algname{ComputeGradient} (Alg.~\ref{algorithm:ga_asynch_main_new}, performing the low-probability step), and \algname{ComputeBatchDifference} (Alg.~\ref{algorithm:ga_asynch_main_diff}, performing the high-probability step). Let us focus on \algname{ComputeGradient}, designed to collect the full gradient: it takes a point $x$ as input and returns $\nabla f(x) = \frac{1}{m} \sum_{i=1}^{m} \nabla f_i(x).$
There exist many strategies for implementing this calculation, some of which were outlined in Section~\ref{sec:previous_compl}. The most naive one is to assign the task of calculating the whole gradient $\nabla f$ to a single worker $i$, resulting in a worst-case running time of $m \tau_i$ seconds for \algname{ComputeGradient}. Another possible strategy is to distribute the functions $\{f_i\}$ evenly among the workers; in this case, calculating $\nabla f$ takes $\tau_n \max\{\nicefrac{m}{n}, 1\}$ seconds in the worst case.

Clearly, we could do better if we \emph{knew $\{\tau_i\}$ in advance}. Indeed, let us allocate to each worker $j$ a number of functions $\{f_i\}$ inversely proportional to $\tau_j$. This strategy is reasonable -- the faster the worker, the more gradients it can compute. We can show that such a strategy finds $\nabla f$ in 
\begin{align}
    \label{eq:unieLpdtWZUm}
    \textstyle \Theta\left(\min\limits_{j\in[n]} \parens{\parens{\sum\limits_{i=1}^j \frac{1}{\tau_i}}^{-1} (m + j)}\right)
\end{align}
seconds in the worst case (see the proof of Theorem \ref{thm:bath_diff}). This complexity is better than $m \tau_1$ and $\tau_n \max\{\nicefrac{m}{n}, 1\}$ (Theorem~\ref{eq:simple_upper_bounds_for_eq_time}).
However, this approach comes with two major limitations: i) it requires knowledge of the upper bounds $\{\tau_i\}$, ii)~even if we have access to $\{\tau_i\}$, the computation environment can be adversarial: theoretically and practically, it is possible that at the beginning the first worker is the fastest and the last worker is the slowest, but after some time, their performances swap. Consequently, the first worker might end up being assigned the largest batch, despite now having the lowest performance. Thus, this strategy is not robust to time-varying speeds.

\begin{figure}[t]
\begin{minipage}[t]{0.5\textwidth}
\begin{algorithm}[H]
    \centering
    \caption{\algname{ComputeGradient}($x$)}
    \label{algorithm:ga_asynch_main_new}
    \begin{algorithmic}[1]
    \State \textbf{Input:} point $x\in\R^d$
    \State Init $g=0\in\R^d$, set $\cM = \emptyset$
    \State Broadcast $x$ to all workers
    \State For each worker $i\in[n]$, sample $j$ from $[m]$ uniformly and ask it to calculate $\nabla f_j(x)$
        \While{$\cM \neq [m]$}
            \State Wait for $\nabla f_p(x)$ from a worker
            \If{$p \in [m] \backslash \cM$}
            \State $g \gets g + \frac{1}{m} \nabla f_p(x)$
            \State Update $\cM \gets \cM \cup \{p\}$
            \EndIf
            \State Sample $j$ from $[m] \backslash \cM$ uniformly
            and ask \hspace*{0.45cm}  this worker to calculate $\nabla f_{j}(x)$ \label{line:sample_two_main}
        \EndWhile
        \State Return $g = \frac{1}{m} \sum\limits_{i=1}^{m} \nabla f_i(x)$
    \end{algorithmic}
\end{algorithm}
\end{minipage}
\hfill
\begin{minipage}[t]{0.5\textwidth}
\begin{algorithm}[H]
    \centering
    \caption{\algname{ComputeBatchDifference}($S$, $x$, $y$)}
    \label{algorithm:ga_asynch_main_diff}
    \begin{algorithmic}[1]
    \State \textbf{Input:} batch size $S\in\N$, points $x,y\in\R^d$
    \State Init $g=0\in\R^d$
    \State Broadcast $x, y$ to all workers
    \State For each worker, sample $j$ from $[m]$ uniformly and ask it to calculate $\nabla f_j(x) - \nabla f_j(y)$
        \For{$i = 1, 2, \ldots, S$}
            \State Wait for $\nabla f_p(x) - \nabla f_p(y)$ from a worker
            \State $g \gets g + \frac{1}{S} (\nabla f_p(x) - \nabla f_p(y))$
            \State Sample $j$ from $[m]$ uniformly and ask \hspace*{0.45cm}  this worker to calculate $\nabla f_j(x) - \nabla f_j(y)$
        \EndFor
        \State Return $g$
    \end{algorithmic}
\end{algorithm}
\vspace*{-0.5cm}
{\scriptsize Notes: i) the workers can aggregate $\nabla f_p$ locally, and the algorithm can call \algname{AllReduce} once to collect all calculated gradients. ii) By splitting $[m]$ into blocks, instead of one $\nabla f_p,$ we can ask the workers to calculate the sum of one block in Alg.~\ref{algorithm:ga_asynch_main_new} (and use a similar idea in Alg.~\ref{algorithm:ga_asynch_main_diff}).}
\end{minipage}
\end{figure}

\paragraph{New gradient computation strategy.} 
The key innovation of this work lies in the introduction of new computation strategies: Algorithms~\ref{algorithm:ga_asynch_main_new} and \ref{algorithm:ga_asynch_main_diff}. We start by examining Algorithm~\ref{algorithm:ga_asynch_main_new}. It first broadcasts the input $x\in\R^d$ to all workers. Then, for each worker, it samples $j$ \emph{uniformly} from~$[m]$ and asks it to calculate $\nabla f_j(x)$ (with a non-zero probability, two workers can be assigned the same computation). Then, the algorithm enters the loop and waits for any worker to finish their calculations. Once this happens, it asks this worker to compute a stochastic gradient with a new index sampled \emph{uniformly} from the set $[m] \backslash \cM$ of indices that have not yet been processed (again, it is possible to resample an index previously assigned to another worker). This continues until all indices $i\in[m]$ have been processed and the full gradient $\frac{1}{m} \sum_{i=1}^m \nabla f_i$ has been collected.
Unlike the previous strategies, our Algorithm~\ref{algorithm:ga_asynch_main_new} does not use $\{\tau_i\}$, thus being \emph{robust and adaptive to the changing compute times}. 
Furthermore, we can prove that its time complexity (almost) equals \eqref{eq:unieLpdtWZUm}:
\begin{restatable}{theorem}{THMFULLGRADTIME}\label{thm:full_grad_time}
    The expected time needed by Algorithm~\ref{algorithm:ga_asynch_main_new} to calculate $g = \frac{1}{m} \sum\limits_{i=1}^m \nabla f_i$ is at most
    \begin{align}
        \label{eq:EfEPFSGXiRSAUaS}
        \textstyle 12 \min\limits_{j\in[n]} \parens{\parens{\sum\limits_{i=1}^j \frac{1}{\tau_i}}^{-1} (m + {\min\{m,n\} \log \left(\min\{m,n\}\right)} + j)}
    \end{align}
    seconds.
\end{restatable}

The result \eqref{eq:EfEPFSGXiRSAUaS} (the proof of which can be found in Section~\ref{sec:time_compl_sampling}) is slightly worse than \eqref{eq:unieLpdtWZUm} due to the extra ${\min\{m,n\} \log \left(\min\{m,n\}\right)}$ term. This term arises because a worker may be assigned a gradient $\nabla f_j(x)$ that was previously assigned to another worker (in \begin{NoHyper}Line~\ref{line:sample_two_main}\end{NoHyper} of Algorithm~\ref{algorithm:ga_asynch_main_new}). Hence, with a \emph{small} (but non-zero) probability, two workers can perform the same calculations. However, typically the number of samples $m$ is much larger than the number of workers $n.$ If we assume that $m \geq n \log n,$ which is satisfied in many practical scenarios, then the time complexity \eqref{eq:EfEPFSGXiRSAUaS} equals 
\begin{align*}
    \textstyle \Theta\left(\min\limits_{j\in[n]} \parens{\parens{\sum\limits_{i=1}^j \frac{1}{\tau_i}}^{-1} (m + j)}\right)
\end{align*}
and the term $\min\{m,n\} \log \left(\min\{m,n\}\right)$ never dominates. Since this complexity is not worse than~\eqref{eq:unieLpdtWZUm}, our strategy behaves as if it knew $\{\tau_i\}$ in advance! To simplify formulas and avoid the logarithmic term, we use the following assumption throughout the main part of this paper:
\begin{assumption}
    \label{ass:m_is_large}
    $m \geq n \log n,$ where $m$ is the \# of data samples and $n$ is the \# of workers.
\end{assumption}

We now proceed to discuss \algname{ComputeBatchDifference} (Algorithm~\ref{algorithm:ga_asynch_main_diff}), designed to compute a minibatch of stochastic gradient differences. Both Algorithms~\ref{algorithm:ga_asynch_main_new} and~\ref{algorithm:ga_asynch_main_diff} calculate sums. However the latter only waits until there are $S$ samples in the sum, where some indices in the batch may not be unique. On the other hand, Algorithm~\ref{algorithm:ga_asynch_main_new} must ensure the collection of a full batch of $m$ \emph{unique} stochastic gradients. As a result, Algorithm~\ref{algorithm:ga_asynch_main_diff} offers better complexity results and, unlike \algname{ComputeGradient}, does not suffer from suboptimal guarantees and logarithmic terms, as demonstrated in the theorem below.

\begin{restatable}{theorem}{THMBATHDIFF}\label{thm:bath_diff}
    The time needed by Algorithm~\ref{algorithm:ga_asynch_main_diff} to calculate $g$ is at most
    \begin{align}
        \label{eq:bath_diff_time}
        \textstyle 4 \min\limits_{j\in[n]} \parens{\parens{\sum\limits_{i=1}^j \frac{1}{\tau_i}}^{-1} (S + j)}
    \end{align}
    seconds.
\end{restatable}

Algorithm~\ref{algorithm:rennala_page_gen} uses \emph{uniform sampling with replacement}, implemented in Algorithm~\ref{algorithm:ga_asynch_main_diff}. However, one can modify the two algorithms slightly to support virtually \emph{any unbiased sampling} (see Section~\ref{sec:proofs_rennala_page_other}).

\paragraph{Note on asynchronicity.}

It is clear that to eliminate the need of waiting for very slow machines, some level of asynchronicity has to be injected into an algorithm for it to be efficient. \algname{Asynchronous SGD} \citep{recht2011hogwild,nguyen2018sgd,koloskova2022sharper,mishchenko2022asynchronous} takes this concept to the extreme by never synchronizing and continually overwriting the updates. Consequently, the algorithm's time complexity is suboptimal.
Conversely, imposing limitations on asynchronicity leads to optimal methods, both in our context (in the large-scale regime) and in the scenario examined by \citet{tyurin2023optimal}. \algname{Freya PAGE} seamlessly combines synchrony and asynchrony, getting the best out of the two worlds.

\section{Time Complexities and Convergence Rates}
\label{sec:theory}
Formulas \eqref{eq:unieLpdtWZUm} and \eqref{eq:EfEPFSGXiRSAUaS} will be used frequently throughout the paper. To lighten up the heavy notation, let us define the following mapping.
\begin{definition}[Equilibrium time]
\label{def:time_budget}
A mapping $\Teq \,:\, \R_{\geq 0} \times \R^n_{\geq 0} \rightarrow \R_{\geq 0}$ defined by
\begin{equation}
\begin{aligned}
    \label{eq:equil_time}
    \textstyle \Teq(S, [\bar{\tau}_i]_{i=1}^{n}) \eqdef \min\limits_{j\in[n]} \parens{\parens{\sum\limits_{i=1}^j \frac{1}{\bar{\tau}_i}}^{-1} (S+j)} \in [0, \infty]
\end{aligned}
\end{equation}
is called the \emph{equilibrium time}.
We let $\Teq(S) \equiv \Teq(S, [\tau_i]_{i=1}^{n})$ when considering $\{\tau_i\}$ from Section~\ref{sec:intro}.
\end{definition}

Returning to the algorithm, we guarantee the following iteration complexity.
\begin{restatable}[Iteration complexity]{theorem}{THMPAGERENNALAINDEPITER}
    \label{thm:page_rate}
    Let Assumptions~\ref{as:smooth_lower_bdd}, \ref{as:L+} and \ref{as:L_pm} hold. 
    Consider any minibatch size $S \in \N \eqdef \{1,2,\ldots\}$, any probability $p \in (0, 1],$ and let the stepsize be $\gamma = \left(L_- + L_{\pm} \sqrt{\frac{1-p}{p S}}\right)^{-1}$.
    Then, after
    \begin{align}\label{eq:iter_compl_indep}
        \textstyle K \geq K_{\algname{PAGE}} \eqdef \frac{2 \delta^0}{\varepsilon} \left(L_- + L_{\pm} \sqrt{\frac{1-p}{p S}}\right)
    \end{align}
    iterations of Algorithm~\ref{algorithm:rennala_page_gen}, we have $\Exp{\norm{\nabla f (\hat{x}^K)}^2} \leq \varepsilon$, where $\hat{x}^K$ is sampled uniformly at random from the iterates $\{x^0,\ldots,x^{K-1}\}$.
\end{restatable}

Theorem~\ref{thm:page_rate} states that the iteration complexity is the same as in the optimal \algname{PAGE} method \citep{li2021page}. Note that we can guarantee convergence even if the upper bounds $\{\tau_i\}$ are unknown or infinite (as long as there exists some worker that can complete computations within a finite time).

We now derive time complexity guarantees. With probability $p$, the workers need to supply to the algorithm $m$ stochastic gradients at each of the $m$ data samples, which by Theorem \ref{thm:full_grad_time} can be done in at most $\Teq(m)$ seconds (up to a log factor). Otherwise, they compute $S$ differences of stochastic gradients, which by Theorem \ref{thm:bath_diff} takes at most $\Teq(S)$ seconds (up to a constant factor).
The resulting time complexity is given in the theorem below.
\begin{restatable}[Time complexity with free parameters $p$ and $S$]{theorem}{THMPAGERENNALATIME}
    \label{thm:page_time_compl}
    Consider the assumptions and the parameters from Theorem~\ref{thm:page_rate}, plus Assumption~\ref{ass:m_is_large}. The expected time complexity of Algorithm~\ref{algorithm:rennala_page_gen} is at most
    \begin{equation}
    \begin{aligned}\label{eq:compl_p_S_indep}
        &\textstyle T(p,S, [\tau_i]_{i=1}^n) \eqdef 12 \cdot \Teq(m, [\tau_i]_{i=1}^n) \\
        &\textstyle \quad + \frac{48 \delta^0}{\varepsilon} \parens{L_- + L_{\pm} \sqrt{\frac{1-p}{p S}}} \times 
        \brac{p \cdot \Teq(m, [\tau_i]_{i=1}^n) + (1-p) \cdot \Teq(S, [\tau_i]_{i=1}^n)}.
    \end{aligned}
    \end{equation}
\end{restatable}
The first term comes from the preprocessing step, where the full gradient is calculated to obtain $g^0 = \nabla f(x^0)$. Here, we use Assumption~\ref{ass:m_is_large} that $m \geq n \log n.$ The result \eqref{eq:compl_p_S_indep} is valid even without this assumption, but at the cost of extra logarithmic factors.

\subsection{Optimal parameters $S$ and $p$}\label{sec:opt_params}

The time complexity \eqref{eq:compl_p_S_indep} depends on two free parameters, $S \in \N$ and $p \in (0, 1].$ The result below (following from Theorems~\ref{thm:opt_p_indep} and \ref{thm:opt_S}) determines their optimal choice.

\begin{restatable}[Main result]{theorem}{THMOPTPSMAIN}
    \label{thm:opt_S_main}
    Consider the assumptions and parameters from Theorem~\ref{thm:page_rate}, plus Assumption~\ref{ass:m_is_large}. Up to a constant factor, the time complexity \eqref{eq:compl_p_S_indep} is at least $T([\tau_i]_{i=1}^n) \eqdef \Teq(m, [\tau_i]_{i=1}^n)$
    \begin{align}
        \label{eq:best_time_comp}
        \hspace{-0.5cm}\textstyle +\frac{\delta^0}{\varepsilon}  \min\Bigg\{L_- \Teq(m, [\tau_i]_{i=1}^n), \min\limits_{S \in \N} \underbrace{\textstyle \left[L_- \Teq(S, [\tau_i]_{i=1}^n) + L_{\pm} \frac{\sqrt{\Teq(m, [\tau_i]_{i=1}^n) \Teq(S, [\tau_i]_{i=1}^n)}}{\sqrt{S}}\right]}_{F(S) \eqdef } \Bigg\},
    \end{align}
    and this lower bound is achieved with $S^* = \arg\min\limits_{S \in \N} F(S)$ and $p^* = p^*(S^*),$ where
    \begin{align*}
        \textstyle p^*(S) = \begin{cases}
            1, & \textstyle \textnormal{if } \textstyle L_- \Teq(m, [\tau_i]_{i=1}^n) \leq L_- \Teq(S, [\tau_i]_{i=1}^n) + L_{\pm} \frac{\sqrt{\Teq(m, [\tau_i]_{i=1}^n)\Teq(S, [\tau_i]_{i=1}^n)}}{\sqrt{S}}, \\
            \textstyle \frac{\Teq(S, [\tau_i]_{i=1}^n)}{\Teq(m, [\tau_i]_{i=1}^n)}, & \textnormal{otherwise}.
        \end{cases}
    \end{align*}
\end{restatable}

Result \eqref{eq:best_time_comp} is the best possible time complexity that can be achieved with the \algname{Freya PAGE} method.
Unfortunately, the final time complexity has non-trivial structure, and the optimal parameters depend on $\{\tau_i\}$ in the general case. If we have access to all parameters and times $\{\tau_i\},$ then \eqref{eq:best_time_comp}, $S^*,$ and $p^*$ can be computed efficiently. Indeed, the main problem is to find $\min_{S \in \N} F(S),$ which can be solved, for instance, by using the bisection method, because $L_- \Teq(S, [\tau_i]_{i=1}^n)$ is non-decreasing and $L_{\pm} \sqrt{\Teq(m, [\tau_i]_{i=1}^n) \Teq(S, [\tau_i]_{i=1}^n)} / \sqrt{S}$ is non-increasing in $S$.

\subsection{Optimal parameters $S$ and $p$ in the large-scale regime }\label{sec:opt_params_large_data}

Surprisingly, we can significantly simplify the choice of the optimal parameters $S$ and $p$ in the large-scale regime, when $\sqrt{m} \geq n$. This is a weak assumption, since typically the number of data samples is much larger than the number of workers.

\begin{restatable}[Main result in the large-scale regime]{theorem}{THEOREMSIMPLETIMECOMPLEXITYCHOICE}\label{cor:nice_upper_bound}
    Consider the assumptions and parameters from Theorem~\ref{thm:page_rate}, plus Assumption~\ref{ass:m_is_large}. Up to a constant factor and smoothness constants, if $\sqrt{m} \geq n,$ then the optimal choice of parameters in~\eqref{eq:compl_p_S_indep} is $S^* = \ceil{\sqrt{m}}$ and $p^* = \nicefrac{1}{\sqrt{m}}$. For this choice, the expected time complexity of Algorithm~\ref{algorithm:rennala_page_gen} is at most
    \begin{align}
        \label{eq:time_comp}
        \textstyle T(\nicefrac{1}{\sqrt{m}},\sqrt{m}, [\tau_i]_{i=1}^n) = 12 \Teq(m, [\tau_i]_{i=1}^n) + \frac{192 \delta^0 \max\{L_-, L_{\pm}\}}{\varepsilon} \Teq(\sqrt{m}, [\tau_i]_{i=1}^n)
    \end{align}
    seconds. The iteration complexity with $S = \ceil{\sqrt{m}}$ and $p = \nicefrac{1}{\sqrt{m}}$ is $K_{\algname{PAGE}} \leq \nicefrac{4 \delta^0 \max\{L_-, L_{\pm}\}}{\varepsilon}$.
\end{restatable}

We cannot guarantee that $S = \ceil{\sqrt{m}}$ and $p = \nicefrac{1}{\sqrt{m}}$ is the \emph{optimal} pair when $\sqrt{m} < n,$ but it is a valid choice for all $m \geq 1.$ Note that \eqref{eq:time_comp} is true if $m \geq n \log n$, and it is true up to a log factor if $m < n \log n$. 
In light of Theorem~\ref{thm:rpage_psopt_indep}, we can further refine Theorem~\ref{cor:nice_upper_bound} if the ratio $\nicefrac{L_{\pm}}{L}$ is known:

\begin{theorem}[Main result in the large-scale regime using the ratio $\nicefrac{L_{\pm}}{L}$]
    \label{cor:nice_upper_bound_better}
    Consider the assumptions and parameters from Theorem~\ref{thm:page_rate}, plus Assumption~\ref{ass:m_is_large}. The expected time complexity of Algorithm~\ref{algorithm:rennala_page_gen} is at most $\Theta(\Teq(m, [\tau_i]_{i=1}^n) + \nicefrac{\delta^0 L_-}{\varepsilon} \times \Teq(\min\{\max\{\ceil{\nicefrac{L_{\pm} \sqrt{m}}{L_-}}, 1\}, m\}, [\tau_i]_{i=1}^n))$ seconds, where $S = \min\{\max\{\ceil{\nicefrac{L_{\pm} \sqrt{m}}{L_-}}, 1\}, m\}$ and $p = \nicefrac{S}{m}.$
\end{theorem}

For brevity reasons, we will continue working with the result from Theorem~\ref{cor:nice_upper_bound} in the main part of this paper. Note that the optimal parameters do not depend on $\{\tau_i\}$, and can be easily calculated since the number of functions $m$ is known in advance. Hence, our method is \emph{fully adaptive to changing and heterogeneous compute times of the workers}.

Even if the bounds are unknown and $\tau_i = \infty$ for all $i \in [n],$ our method converges after $4 \delta^0 \max\{L_-, L_{\pm}\} / \varepsilon$ iterations, and calculates the optimal number of stochastic gradients equal to $\cO(m +\sqrt{m} \delta^0 \max\{L_-, L_{\pm}\}/\varepsilon)$.

\subsection{Discussion of the time complexity}
\label{sec:discussion}
Let us use Definition~\ref{def:time_budget} and unpack the second term in the complexity \eqref{eq:time_comp}, namely,
\begin{align*}
    \textstyle \frac{\delta^0 \max\{L_-, L_{\pm}\}}{\varepsilon} \min\limits_{j\in[n]} \parens{\parens{\sum\limits_{i=1}^j \frac{1}{\tau_i}}^{-1} (\sqrt{m} + j)}.
\end{align*}
A somewhat similar expression involving $\min_{j \in [n]}$ and harmonic means was obtained by \citet{tyurin2023optimal,tyurin2024shadowheart} for minimizing expectation under the bounded variance assumption. The term $\delta^0 \max\{L_-, L_{\pm}\} / \varepsilon$ is standard in optimization \citep{nesterov2018lectures,lan2020first} and describes the difficulty of the problem \eqref{eq:problem}. The term $\min_{j\in[n]} ((\sum_{i=1}^j \nicefrac{1}{\tau_i})^{-1} (\sqrt{m} + j))$ represents the average time of one iteration and has some nice properties. For instance, if the last worker is slow and $\tau_n \approx \infty,$ then 
    $\min_{j\in[n]} (\cdots) = \min_{j\in[{\red n - 1}]} (\cdots),$
so the complexity effectively ignores it. Moreover, if $j^*$ is an index that minimizes $\min_{j \in [n]}(\cdots),$ then $\min_{j\in[n]} ((\sum_{i=1}^j \nicefrac{1}{\tau_i})^{-1} (\sqrt{m} + j)) = ((\sum_{i=1}^{j^*} \nicefrac{1}{\tau_i})^{-1} (\sqrt{m} + j^*)).$ The last formula, again, does not depend on the slowest workers $\{j^* + 1, \dots, n\}$, which are automatically excluded from the time complexity expression. The same reasoning applies to the term $\Teq(m, [\tau_i]_{i=1}^n).$ Let us now consider some extreme examples which are meant to shed some light on our time complexity result \eqref{eq:time_comp}:

\begin{restatable}{example}{EXAMPLESAMEMAIN}[Equally Fast Workers]
    \label{example:samespeed_main}
    Suppose that the upper bounds on the processing times are equal, i.e., $\tau_j = \tau$ for all $j \in [n]$. Then 
    \begin{align*}
        \textstyle T(\nicefrac{1}{\sqrt{m}},\sqrt{m}, [\tau_i]_{i=1}^n) = \Theta\left(\tau \max\left\{\frac{m}{n}, 1\right\} + \tau \frac{\delta^0 \max\{L_-, L_{\pm}\}}{\varepsilon}\max\left\{\frac{\sqrt{m}}{n}, 1\right\} \right).
    \end{align*}
\end{restatable}
The complexity in Example~\ref{example:samespeed_main} matches that in \eqref{eq:page_time}, which makes sense since \algname{Soviet PAGE} is a reasonable method when $\{\tau_i\}$ are equal. Note that the reduction happens without prior knowledge of $\{\tau_i\}$.
\begin{restatable}{example}{EXAMPLEINFFASTMAIN}[Infinitely Fast Worker]
    \label{example:inffast_main}
    If $\tau_1=0$, then $T(\nicefrac{1}{\sqrt{m}},\sqrt{m}, [\tau_i]_{i=1}^n) = 0.$
\end{restatable}
\begin{restatable}{example}{EXAMPLEINFSLOWMAIN}[Infinitely Slow Workers]
    \label{example:infslow_main}
    If $\tau_j=\infty \, \forall j \in[n]$, then $T(\nicefrac{1}{\sqrt{m}},\sqrt{m}, [\tau_i]_{i=1}^n) = \infty.$
\end{restatable}
\begin{restatable}{example}{EXAMPLESLOWMAIN}[Extremely Slow Workers]
    \label{example:slow_main}
    Suppose that the times $\tau_j < \infty$ are fixed $\forall j \leq j_{B}$ and $\tau_j\geq B$ $\forall j > j_{B}$ for some $B$ large enough. Then $T(\nicefrac{1}{\sqrt{m}},\sqrt{m}, [\tau_i]_{i=1}^n) = T(\nicefrac{1}{\sqrt{m}},\sqrt{m}, [\tau_i]_{i=1}^{j_{B}}).$
\end{restatable}

Example \ref{example:slow_main} says that the workers whose processing time is too large are \emph{ignored}, 
which supports the discussion preceding the examples.

\subsection{Dynamic bounds}\label{sec:time_var_times}

It turns out that the guarantees from Theorem~\ref{cor:nice_upper_bound} can be generalized to the situation where the compute times $\{\tau_i\}$ are allowed to dynamically change throughout the iterations. Consider the $k$\textsuperscript{th} iteration of Algorithm~\ref{algorithm:rennala_page_gen} and assume that worker $i$ calculates one $\nabla f_j$ in at most $\tau_i^k \in [0, \infty]$ seconds $\forall i \in [n], j \in [m].$ 
Clearly, $\max_{k \geq -1}\tau_i^k \leq \tau_i$ $\forall i \in [n]$ (where $\{\tau_i^{-1}\}$ are upper bounds from the preprocessing step), but $\tau_i^k$ can be arbitrarily smaller than $\tau_i$ (possibly $\tau_i^k = 0$ and $\tau_i = \infty$).

\begin{theorem}
    Consider the assumptions and parameters from Theorem~\ref{thm:page_rate}, plus Assumption~\ref{ass:m_is_large}. Up to a constant factor, the expected time complexity of Algorithm~\ref{algorithm:rennala_page_gen} with $S = \ceil{\sqrt{m}}$ and $p = \nicefrac{1}{\sqrt{m}}$ is at most
\begin{align}
    \label{eq:QpDKSb}
    \textstyle \Teq(m, [\tau_{\pi_{-1,i}}^{-1}]_{i=1}^n) + \sum\limits_{k=0}^{\ceil{4 \delta^0 \max\{L_-, L_{\pm}\} / \varepsilon}} \Teq(\sqrt{m}, [\tau_{\pi_{k,i}}^k]_{i=1}^n),
\end{align}
where $\pi_{k,\cdot}$ is a permutation such that $\tau_{\pi_{k,1}}^k \leq \dots \leq \tau_{\pi_{k,n}}^k$ for all $k \geq -1.$
(This theorem follows from Theorem~\ref{thm:page_time_compl_var} with the chosen parameters).
\end{theorem}

Hence, our algorithm is \emph{adaptive to the dynamic compute times} $\{\tau_i^k\}.$ Let us consider an example with $2$ workers. Assume that the first worker is stable: $\tau_1^k = \tau$ for all $k \geq 0$, and the second worker is unstable: $\tau_2^0 = \tau$ is small in the first iteration, and $\tau_2^1 \approx \infty$ in the second iteration. For explanation purposes, we ignore the preprocessing term $\Teq(m, \cdot),$ which is not a factor if $\varepsilon$ is small.
Then, 
\begin{align*}
    \eqref{eq:QpDKSb} 
    = \Teq(\sqrt{m}, [\tau_{\pi_{0,1}}^0, \tau_{\pi_{0,2}}^0]) + \Teq(\sqrt{m}, [\tau_{\pi_{1,1}}^1, \tau_{\pi_{1,2}}^1]) + ... 
    = \Teq(\sqrt{m}, [\tau, \tau]) + \Teq(\sqrt{m}, [\tau]) + ...
\end{align*}
because $\Teq(\sqrt{m}, [\tau_{\pi_{1,1}}^1, \tau_{\pi_{1,2}}^1]) = \Teq(\sqrt{m}, [\tau])$ when $\tau_2^1 \approx \infty.$ The time complexity in the second iteration depends on the first (stable) worker only, which is reasonable since $\tau_2^1 \approx \infty,$ and this happens \emph{automatically}. At the same time, the first term $\Teq(\sqrt{m}, [\tau, \tau])$ depends on both workers, and this iteration will be faster because $\Teq(\sqrt{m}, [\tau, \tau]) \leq \Teq(\sqrt{m}, [\tau]).$

\subsection{Comparison with previous strategies from Section~\ref{sec:previous_compl}}
Our time complexities \eqref{eq:best_time_comp} and \eqref{eq:time_comp} are better than all known previous guarantees if $m \geq n \log n.$ In particular,
$T(\nicefrac{1}{\sqrt{m}},\sqrt{m}, [\tau_i]_{i=1}^n) \leq T_{\textnormal{Soviet PAGE}}$ (from \eqref{eq:page_time}), because $\Teq(\sqrt{m}, [\tau_i]_{i=1}^n) \lesssim \nicefrac{\sqrt{m} \tau_n}{n}$ (Theorem~\ref{eq:simple_upper_bounds_for_eq_time}). In fact, since $\lim_{\tau_n \to \infty} \Teq(\sqrt{m}, [\tau_i]_{i=1}^n) = \Teq(\sqrt{m}, [\tau_i]_{i=1}^{n-1}) < \infty$ and $\lim_{\tau_n \to \infty} = \nicefrac{\sqrt{m} \tau_n}{n} = \infty$, $T_{\textnormal{Soviet PAGE}}$ can be arbitrarily larger. We also improve on \algname{Hero PAGE} (see Remark~\ref{remark:better}).

\subsection{Comparison with existing asynchronous variance reduced methods}

Several studies have explored asynchronous variance reduced algorithms. Essentially all of them are variants of the existing synchronous methods discussed in Section \ref{sec:sync_vr} and depend on the slowest worker in every iteration. 
There have been several attempts to combine variance reduction techniques with asynchronous computations.
Perhaps the most relevant baseline is \algname{SYNTHESIS}, an asynchronous variant of \algname{SPIDER} \citep{fang2018spider} introduced by \citet{liu2022synthesis}. The obtained \emph{gradient complexity} matches that of \algname{SPIDER} in terms of dependence on $m$, but scales linearly with the bound on the time performance of the slowest worker, making it non-adaptive to slow computations. Moreover, in Line~$3$ of their Algorithm~$3$, the full gradient is calculated, assuming that it can be done for free.
Lastly, the analysis assumes the gradients to be bounded.

\section{Lower Bound}
\label{sec:lower_bound_simple}

In previous sections, we showed that \algname{Freya PAGE} converges in at most \eqref{eq:best_time_comp} or \eqref{eq:time_comp} seconds, providing better time complexity guarantees compared to all previous methods. The natural question is: how good are these time complexities, and can they be improved? In Section~\ref{sec:lower_bound_appendix}, we formalize our setup and prove Theorems~\ref{theorem:lower_bound} and \ref{theorem:second_lower_bound}, which collectively yield the following lower bound.

\begin{theorem}[Less formal version of Theorems~\ref{theorem:lower_bound} and \ref{theorem:second_lower_bound}]
    \label{thm:lower_bound_simple}
    Assume that $0 < \tau_1 \leq \dots \leq \tau_n$ and take any $L_{+}, \delta^0, \varepsilon > 0$ and $m \in \N$ such that $\varepsilon < c_1 L_{+} \delta^0.$ Then, for any (zero-respecting) algorithm, there exists a function $f$ that satisfies $f(0) - f^* \leq \delta^0$ and Assumption~\ref{as:L+}, such that it is impossible to find an $\varepsilon$--stationary point faster than in
    \begin{align}
        \label{eq:lower_bound_main}
        \textstyle \Omega\left(\Teq(m, [\tau_i]_{i=1}^n) + \frac{\delta^0 L_+}{\sqrt{m} \varepsilon} \Teq(m, [\tau_i]_{i=1}^n)\right)
    \end{align}
    seconds using uniform sampling with replacement.
\end{theorem}

Comparing \eqref{eq:time_comp} and \eqref{eq:lower_bound_main}, we see that \algname{Freya PAGE} is \emph{optimal} under Assumptions~\ref{as:smooth_lower_bdd} and \ref{as:L+} in the large-scale regime ($\sqrt{m} \geq n$). Indeed, without Assumption~\ref{as:L_pm}, we have $\max\{L_-, L_{\pm}\} = \max\{L_-, L_{+}\} = L_{+}.$ Up to constant factor, \eqref{eq:time_comp} is less or equal to \eqref{eq:lower_bound_main} since
\begin{align*}
    \Teq(\sqrt{m}, [\tau_i]_{i=1}^n) 
    &= \textstyle \min\limits_{j\in[n]} \parens{\parens{\sum\limits_{i=1}^j \frac{1}{\tau_i}}^{-1} (\sqrt{m} + j)} = \frac{1}{\sqrt{m}}\min\limits_{j\in[n]} \parens{\parens{\sum\limits_{i=1}^j \frac{1}{\tau_i}}^{-1} (m + \sqrt{m} j)} \\
    &\textstyle \overset{\sqrt{m} \geq n}{\leq} \frac{2}{\sqrt{m}}\min\limits_{j\in[n]} \parens{\parens{\sum\limits_{i=1}^j \frac{1}{\tau_i}}^{-1} (m + j)} = \frac{2}{\sqrt{m}} \Teq(m, [\tau_i]_{i=1}^n).
\end{align*}
This is \emph{the first optimal method for the problem we consider}, and Theorem \ref{thm:lower_bound_simple} gives \emph{the first lower bound}.

\section{Using the Developed Strategies in Other Methods}
\label{sec:using_strat}
\algname{ComputeBatchDifference} (Algorithm~\ref{algorithm:ga_asynch_main_diff}) is a generic subroutine and can be used in other methods. 
In Section~\ref{sec:freya_sgd}, we introduce \algname{Freya SGD}, a simple algorithm with update rule $x^{k+1} = x^{k} - \gamma \times \algname{ComputeBatch}(S, x^{k})$, where $S$ is a batch size and \algname{ComputeBatch} (Algorithm~\ref{algorithm:ga_asynch_batch}) is a minor modification of \algname{ComputeBatchDifference}. Theorem~\ref{thm:rennala_sgd_indep} establishes that \algname{Freya SGD} converges in $\cO\left(\nicefrac{1}{\varepsilon} \times \Teq\parens{\nicefrac{1}{\varepsilon}, [\tau_i]_{i=1}^n}\right)$
seconds (where we only keep the dependence on $\varepsilon$ and $\{\tau_i\}$). For small~$\varepsilon$, this complexity is worse than \eqref{eq:time_comp}, but it can be better, for instance, in the interpolation regime \citep{schmidt2013fast,ma2018power}. \algname{Freya SGD} resembles \algname{Rennala SGD} \citep{tyurin2023optimal}, but unlike the latter, it is specialized to work with finite-sum problems \eqref{eq:problem} and does not require the $\sigma^2$--bounded variance assumption on stochastic gradients (which is not satisfied in our setting).

\begin{ack}
The research reported in this publication was supported by funding from King Abdullah University of Science and Technology (KAUST): i) KAUST Baseline Research Scheme, ii) Center of Excellence for Generative AI, under award number 5940, iii) SDAIA-KAUST Center of Excellence in Artificial Intelligence and Data Science. The work of A.T. was partially supported by the Analytical center under the RF Government (subsidy agreement 000000D730321P5Q0002, Grant No. 70-2021-00145 02.11.2021).
\end{ack}

\bibliography{bibliography_main}

\begin{thebibliography}{32}
\providecommand{\natexlab}[1]{#1}
\providecommand{\url}[1]{\texttt{#1}}
\expandafter\ifx\csname urlstyle\endcsname\relax
  \providecommand{\doi}[1]{doi: #1}\else
  \providecommand{\doi}{doi: \begingroup \urlstyle{rm}\Url}\fi

\bibitem[Allen-Zhu and Hazan(2016)]{allen2016variance}
Z.~Allen-Zhu and E.~Hazan.
\newblock Variance reduction for faster non-convex optimization.
\newblock In \emph{International Conference on Machine Learning}, pages 699--707. PMLR, 2016.

\bibitem[Arjevani et~al.(2022)Arjevani, Carmon, Duchi, Foster, Srebro, and Woodworth]{arjevani2022lower}
Y.~Arjevani, Y.~Carmon, J.~C. Duchi, D.~J. Foster, N.~Srebro, and B.~Woodworth.
\newblock Lower bounds for non-convex stochastic optimization.
\newblock \emph{Mathematical Programming}, pages 1--50, 2022.

\bibitem[Carmon et~al.(2020)Carmon, Duchi, Hinder, and Sidford]{carmon2020lower}
Y.~Carmon, J.~C. Duchi, O.~Hinder, and A.~Sidford.
\newblock Lower bounds for finding stationary points {I}.
\newblock \emph{Mathematical Programming}, 184\penalty0 (1):\penalty0 71--120, 2020.

\bibitem[Chen et~al.(2016)Chen, Pan, Monga, Bengio, and Jozefowicz]{chen2016revisiting}
J.~Chen, X.~Pan, R.~Monga, S.~Bengio, and R.~Jozefowicz.
\newblock Revisiting distributed synchronous {SGD}.
\newblock \emph{arXiv preprint arXiv:1604.00981}, 2016.

\bibitem[Dutta et~al.(2018)Dutta, Joshi, Ghosh, Dube, and Nagpurkar]{dutta2018slow}
S.~Dutta, G.~Joshi, S.~Ghosh, P.~Dube, and P.~Nagpurkar.
\newblock Slow and stale gradients can win the race: Error-runtime trade-offs in distributed {SGD}.
\newblock In \emph{International Conference on Artificial Intelligence and Statistics}, pages 803--812. PMLR, 2018.

\bibitem[Fang et~al.(2018)Fang, Li, Lin, and Zhang]{fang2018spider}
C.~Fang, C.~J. Li, Z.~Lin, and T.~Zhang.
\newblock {SPIDER}: Near-optimal non-convex optimization via stochastic path-integrated differential estimator.
\newblock \emph{Advances in Neural Information Processing Systems}, 31, 2018.

\bibitem[Horv{\'a}th and Richt{\'a}rik(2019)]{horvath2019nonconvex}
S.~Horv{\'a}th and P.~Richt{\'a}rik.
\newblock Nonconvex variance reduced optimization with arbitrary sampling.
\newblock In \emph{International Conference on Machine Learning}, pages 2781--2789. PMLR, 2019.

\bibitem[Horv{\'a}th et~al.(2022)Horv{\'a}th, Lei, Richt{\'a}rik, and Jordan]{horvath2022adaptivity}
S.~Horv{\'a}th, L.~Lei, P.~Richt{\'a}rik, and M.~I. Jordan.
\newblock Adaptivity of stochastic gradient methods for nonconvex optimization.
\newblock \emph{SIAM Journal on Mathematics of Data Science}, 4\penalty0 (2):\penalty0 634--648, 2022.

\bibitem[Khaled and Richt{\'a}rik(2022)]{khaled2020better}
A.~Khaled and P.~Richt{\'a}rik.
\newblock Better theory for {SGD} in the nonconvex world.
\newblock \emph{Transactions on Machine Learning Research}, 2022.

\bibitem[Kingma and Ba(2015)]{kingma2014adam}
D.~P. Kingma and J.~Ba.
\newblock Adam: A method for stochastic optimization.
\newblock \emph{International Conference on Learning Representations}, 2015.

\bibitem[Koloskova et~al.(2022)Koloskova, Stich, and Jaggi]{koloskova2022sharper}
A.~Koloskova, S.~U. Stich, and M.~Jaggi.
\newblock Sharper convergence guarantees for asynchronous {SGD} for distributed and federated learning.
\newblock \emph{Advances in Neural Information Processing Systems}, 2022.

\bibitem[Kovalev et~al.(2022)Kovalev, Beznosikov, Borodich, Gasnikov, and Scutari]{kovalev2022optimal}
D.~Kovalev, A.~Beznosikov, E.~Borodich, A.~Gasnikov, and G.~Scutari.
\newblock Optimal gradient sliding and its application to optimal distributed optimization under similarity.
\newblock \emph{Advances in Neural Information Processing Systems}, 35:\penalty0 33494--33507, 2022.

\bibitem[Lan(2020)]{lan2020first}
G.~Lan.
\newblock \emph{First-order and stochastic optimization methods for machine learning}.
\newblock Springer, 2020.

\bibitem[LeCun et~al.(2010)LeCun, Cortes, and Burges]{lecun2010mnist}
Y.~LeCun, C.~Cortes, and C.~Burges.
\newblock Mnist handwritten digit database.
\newblock \emph{ATT Labs [Online]. Available: http://yann.lecun.com/exdb/mnist}, 2, 2010.

\bibitem[Lei et~al.(2017)Lei, Ju, Chen, and Jordan]{lei2017non}
L.~Lei, C.~Ju, J.~Chen, and M.~I. Jordan.
\newblock Non-convex finite-sum optimization via {SCSG} methods.
\newblock \emph{Advances in Neural Information Processing Systems}, 30, 2017.

\bibitem[Li et~al.(2021)Li, Bao, Zhang, and Richt{\'a}rik]{li2021page}
Z.~Li, H.~Bao, X.~Zhang, and P.~Richt{\'a}rik.
\newblock {PAGE}: A simple and optimal probabilistic gradient estimator for nonconvex optimization.
\newblock In \emph{International Conference on Machine Learning}, pages 6286--6295. PMLR, 2021.

\bibitem[Liu et~al.(2022)Liu, Zhang, and Liu]{liu2022synthesis}
Z.~Liu, X.~Zhang, and J.~Liu.
\newblock Synthesis: A semi-asynchronous path-integrated stochastic gradient method for distributed learning in computing clusters.
\newblock In \emph{Proceedings of the Twenty-Third International Symposium on Theory, Algorithmic Foundations, and Protocol Design for Mobile Networks and Mobile Computing}, pages 151--160, 2022.

\bibitem[Ma et~al.(2018)Ma, Bassily, and Belkin]{ma2018power}
S.~Ma, R.~Bassily, and M.~Belkin.
\newblock The power of interpolation: Understanding the effectiveness of sgd in modern over-parametrized learning.
\newblock In \emph{International Conference on Machine Learning}, pages 3325--3334. PMLR, 2018.

\bibitem[Mishchenko et~al.(2022)Mishchenko, Bach, Even, and Woodworth]{mishchenko2022asynchronous}
K.~Mishchenko, F.~Bach, M.~Even, and B.~Woodworth.
\newblock Asynchronous {SGD} beats minibatch {SGD} under arbitrary delays.
\newblock \emph{Advances in Neural Information Processing Systems}, 2022.

\bibitem[Nemirovskij and Yudin(1983)]{nemirovskij1983problem}
A.~S. Nemirovskij and D.~B. Yudin.
\newblock \emph{Problem complexity and method efficiency in optimization}.
\newblock Wiley-Interscience, 1983.

\bibitem[Nesterov(2018)]{nesterov2018lectures}
Y.~Nesterov.
\newblock \emph{Lectures on convex optimization}, volume 137.
\newblock Springer, 2018.

\bibitem[Nguyen et~al.(2018)Nguyen, Nguyen, Dijk, Richt{\'a}rik, Scheinberg, and Tak{\'a}c]{nguyen2018sgd}
L.~Nguyen, P.~H. Nguyen, M.~Dijk, P.~Richt{\'a}rik, K.~Scheinberg, and M.~Tak{\'a}c.
\newblock {SGD} and hogwild! convergence without the bounded gradients assumption.
\newblock In \emph{International Conference on Machine Learning}, pages 3750--3758. PMLR, 2018.

\bibitem[Nguyen et~al.(2017)Nguyen, Liu, Scheinberg, and Tak{\'a}{\v{c}}]{nguyen2017sarah}
L.~M. Nguyen, J.~Liu, K.~Scheinberg, and M.~Tak{\'a}{\v{c}}.
\newblock {SARAH}: A novel method for machine learning problems using stochastic recursive gradient.
\newblock In \emph{International Conference on Machine Learning}, pages 2613--2621. PMLR, 2017.

\bibitem[Recht et~al.(2011)Recht, Re, Wright, and Niu]{recht2011hogwild}
B.~Recht, C.~Re, S.~Wright, and F.~Niu.
\newblock Hogwild!: A lock-free approach to parallelizing stochastic gradient descent.
\newblock \emph{Advances in Neural Information Processing Systems}, 24, 2011.

\bibitem[Reddi et~al.(2016)Reddi, Hefny, Sra, Poczos, and Smola]{reddi2016stochastic}
S.~J. Reddi, A.~Hefny, S.~Sra, B.~Poczos, and A.~Smola.
\newblock Stochastic variance reduction for nonconvex optimization.
\newblock In \emph{International Conference on Machine Learning}, pages 314--323. PMLR, 2016.

\bibitem[Schmidt and Roux(2013)]{schmidt2013fast}
M.~Schmidt and N.~L. Roux.
\newblock Fast convergence of stochastic gradient descent under a strong growth condition.
\newblock \emph{arXiv preprint arXiv:1308.6370}, 2013.

\bibitem[Szlendak et~al.(2021)Szlendak, Tyurin, and Richt{\'a}rik]{szlendak2021permutation}
R.~Szlendak, A.~Tyurin, and P.~Richt{\'a}rik.
\newblock Permutation compressors for provably faster distributed nonconvex optimization.
\newblock In \emph{International Conference on Learning Representations}, 2021.

\bibitem[Tyurin and Richt{\'a}rik(2023)]{tyurin2023optimal}
A.~Tyurin and P.~Richt{\'a}rik.
\newblock Optimal time complexities of parallel stochastic optimization methods under a fixed computation model.
\newblock \emph{Advances in Neural Information Processing Systems}, 2023.

\bibitem[Tyurin et~al.(2023)Tyurin, Sun, Burlachenko, and Richt{\'a}rik]{tyurin2022sharper}
A.~Tyurin, L.~Sun, K.~Burlachenko, and P.~Richt{\'a}rik.
\newblock Sharper rates and flexible framework for nonconvex sgd with client and data sampling.
\newblock \emph{Transactions on Machine Learning Research}, 2023.

\bibitem[Tyurin et~al.(2024)Tyurin, Pozzi, Ilin, and Richt{\'a}rik]{tyurin2024shadowheart}
A.~Tyurin, M.~Pozzi, I.~Ilin, and P.~Richt{\'a}rik.
\newblock Shadowheart {SGD}: Distributed asynchronous {SGD} with optimal time complexity under arbitrary computation and communication heterogeneity.
\newblock \emph{arXiv preprint arXiv:2402.04785}, 2024.

\bibitem[Wang et~al.(2019)Wang, Ji, Zhou, Liang, and Tarokh]{wang2019spiderboost}
Z.~Wang, K.~Ji, Y.~Zhou, Y.~Liang, and V.~Tarokh.
\newblock {S}pider{B}oost and momentum: Faster variance reduction algorithms.
\newblock \emph{Advances in Neural Information Processing Systems}, 32, 2019.

\bibitem[Zhou et~al.(2020)Zhou, Xu, and Gu]{zhou2020stochastic}
D.~Zhou, P.~Xu, and Q.~Gu.
\newblock Stochastic nested variance reduction for nonconvex optimization.
\newblock \emph{Journal of Machine Learning Research}, 21\penalty0 (103):\penalty0 1--63, 2020.

\end{thebibliography}
\bibliographystyle{abbrvnat}

\newpage

\appendix

\section*{APPENDIX}
\tableofcontents
\newpage

\section{Experiments}\label{sec:experiments}

We compare \algname{Freya PAGE} with \algname{Rennala SGD}, \algname{Asynchronous SGD}, and \algname{Soviet PAGE} on nonconvex quadratic optimization tasks and practical machine learning problems. The experiments were conducted in Python 3.8 with Intel(R) Xeon(R) Gold 6248 CPU @ 2.50GHz.
We developed a library that emulates the working behavior of thousands of nodes.

\subsection{Experiments with nonconvex quadratic optimization}
\label{sec:exp_quad}

\begin{figure}[H]
    \centering
    \begin{subfigure}[t]{0.4\columnwidth}
      \centering
      \includegraphics[width=\columnwidth]{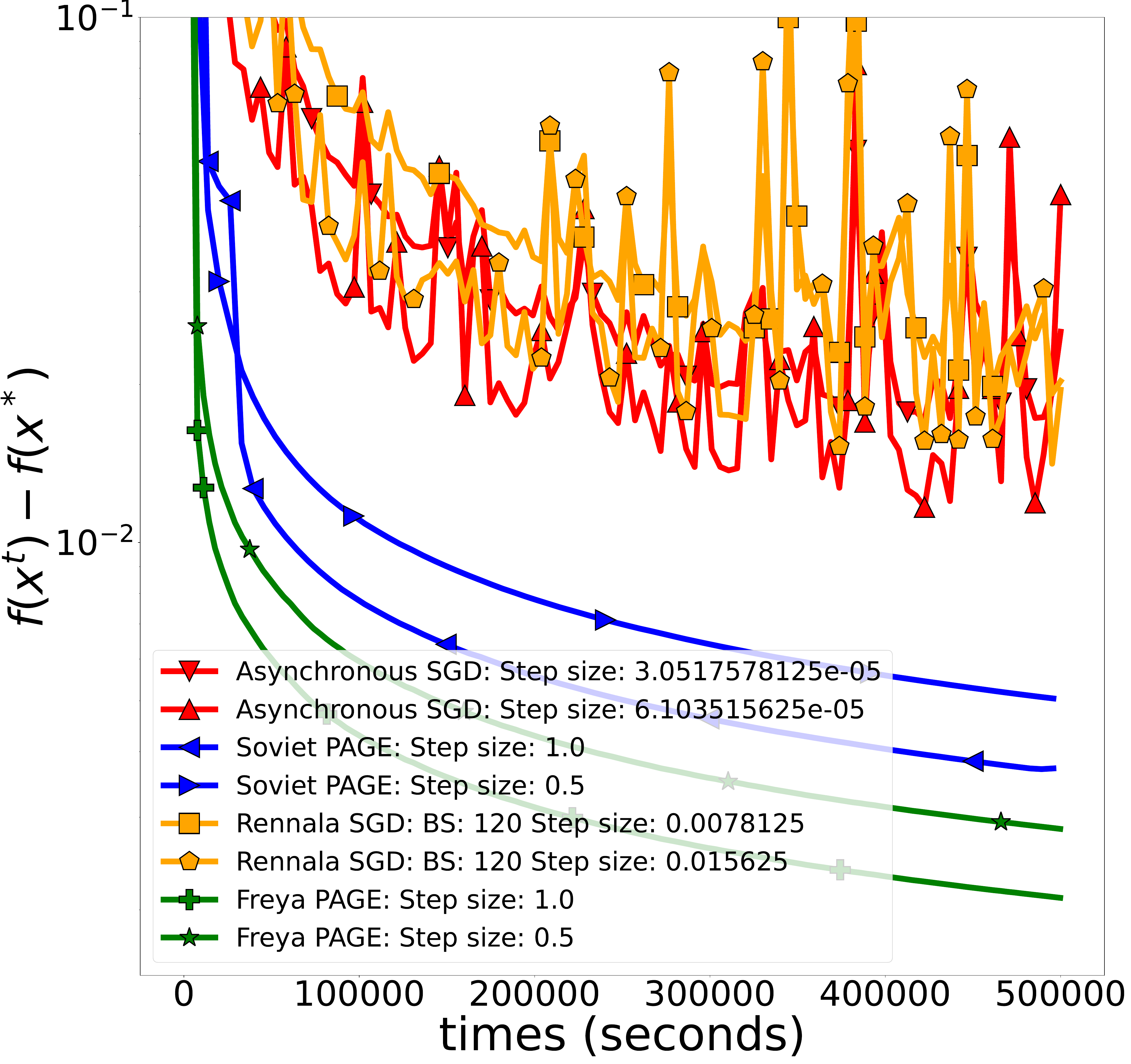}
      \caption{$n = 1000$}
    \end{subfigure}%
    \hspace{2cm}
    \begin{subfigure}[t]{0.4\columnwidth}
      \centering
      \includegraphics[width=\columnwidth]{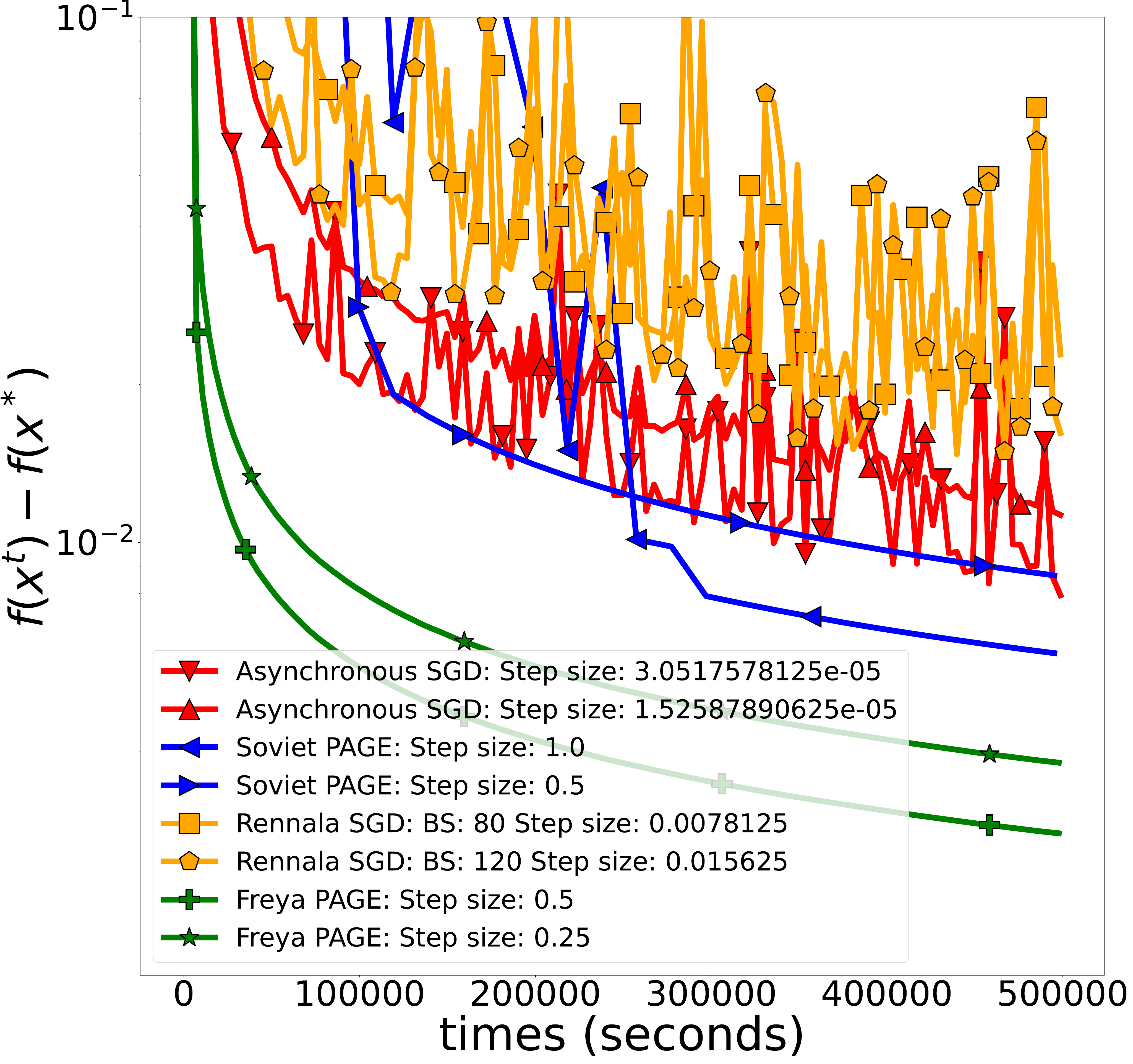}
      \caption{$n = 10000$}
    \end{subfigure}
    \caption{Experiments with nonconvex quadratic optimization tasks. We plot function suboptimality against elapsed time.}
    \label{fig:quad}
\end{figure}

In the first set of experiments, we compare the algorithms on a synthetic quadratic optimization task generated using the procedure from Section~\ref{sec:exp:setup}. To ensure robust and fair comparison, we fix the performance of each worker and emulate our setup by assuming that the $i$\textsuperscript{th} worker requires $\sqrt{i}$ seconds to calculate a stochastic gradient. For each algorithm, we fine-tune the step size from the set $\{2^i\,|\,i \in [-20, 20]\}$. Uniform sampling with replacement is used across all methods. In \algname{Freya PAGE}, we set $S = \ceil{\sqrt{m}}$ according to Theorem~\ref{cor:nice_upper_bound}. We consider $n\in\{1000, 10000\}$ and in each case plot the best run of each method. 

The results are presented in Figure~\ref{fig:quad}. It is evident that our new method, \algname{Freya PAGE}, has the best convergence performance among all algorithms considered. The convergence behavior of \algname{Rennala SGD} and \algname{Asynchronous SGD} is very noisy, and both achieve lower accuracy than \algname{Freya PAGE}. Furthermore, the gap between \algname{Freya PAGE} and \algname{Soviet PAGE} widens with increasing $n$ because \algname{Soviet PAGE} is not robust to the presence of slow workers.

\subsection{Experiments with logistic regression}

\begin{figure}[h]
    \centering
    \begin{subfigure}[t]{.4\columnwidth}
      \centering
      \includegraphics[width=\columnwidth]{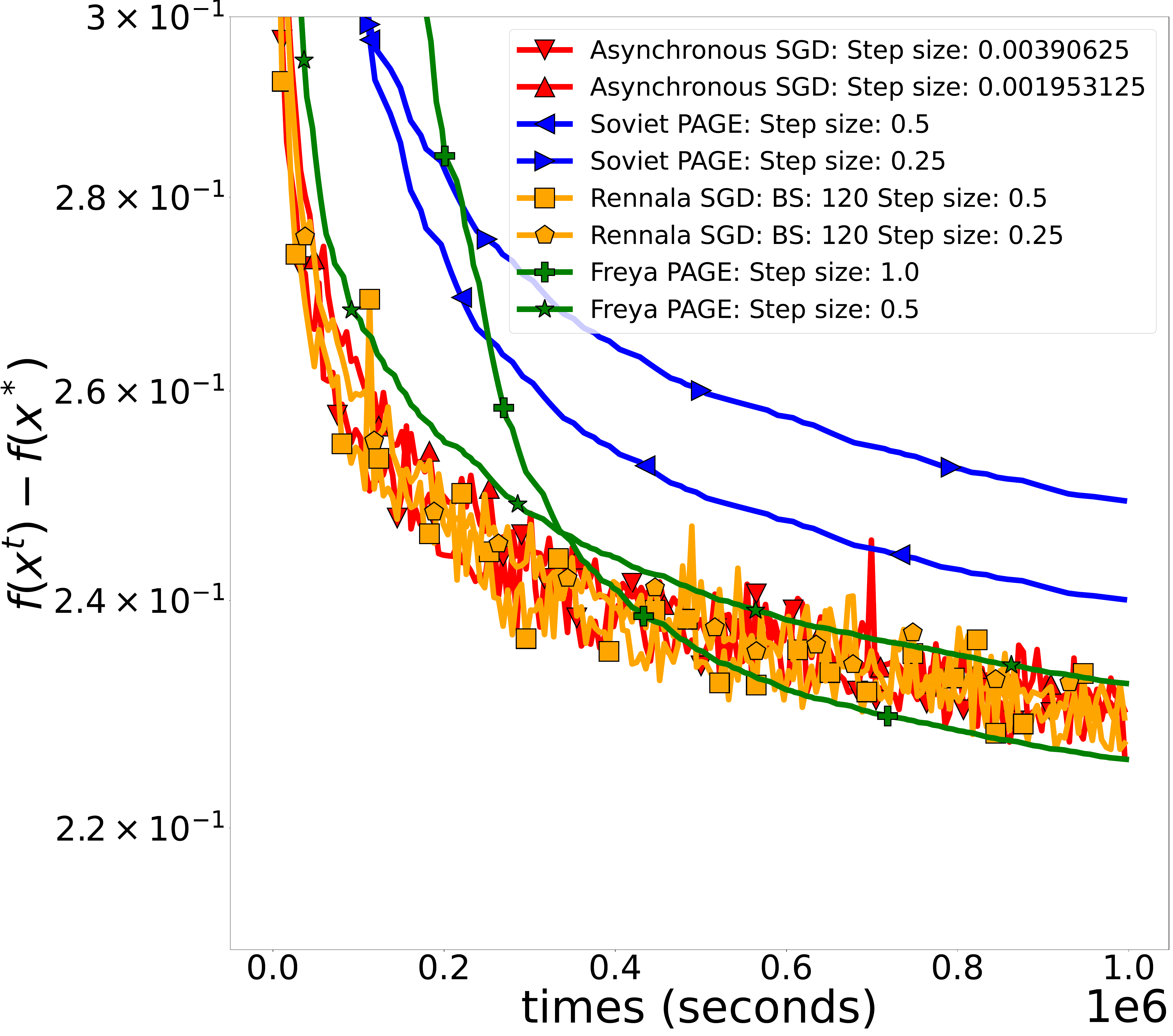}
      \caption{$n = 100$}
    \end{subfigure}%
    \hspace{2cm}
    \begin{subfigure}[t]{.4\columnwidth}
      \centering
      \includegraphics[width=\columnwidth]{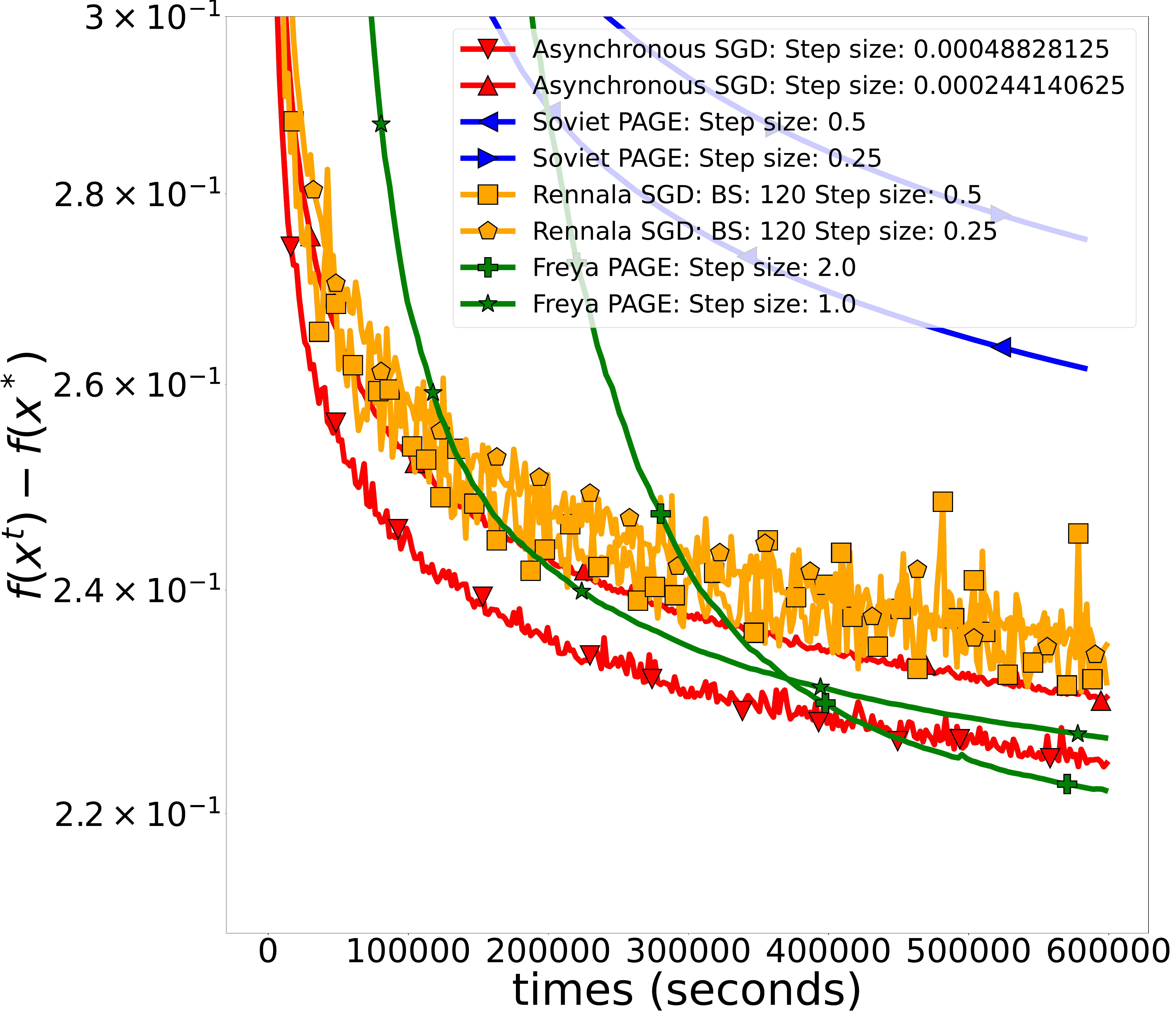}
      \caption{$n = 10000$}
      \label{fig:log_2}
    \end{subfigure}
    \caption{Experiments with the logistic regression problem on the \texttt{MNIST} dataset.}
    \label{fig:log}
\end{figure}

% \begin{table}[H]
% \centering
% \caption{Test accuracies and their variances for methods. \label{tbl:acc}}
% \begin{tabular}{|c|c|c|}
% {Method} & {Accuracy} & {Variance of Accuracy} \\
% \hline
% \algname{Asynchronous SGD} & 92.60 & 5.85e-07 \\
% \algname{Soviet PAGE} & 92.31 & 1.62e-07 \\
% \algname{Rennala SGD} & 92.37 & 3.12e-06 \\
% \algname{Freya PAGE} & \textbf{92.66} & \textbf{1.01e-07} \\
% \end{tabular}
% \end{table}

\begin{table}[H]
\caption{Mean and variance of algorithm accuracies on the \texttt{MNIST} test set during the final 100K seconds of the experiments from Figure~\ref{fig:log_2}.}
\label{tbl:acc}
\centering 
\begin{threeparttable}
\begin{tabular}[t]{ccc}
    \toprule
    \bf  Method & \bf Accuracy & \bf Variance of Accuracy \\
    \midrule
    \makecell{\algname{Asynchronous SGD} \\ \small\citep{koloskova2022sharper} \\ \small\citep{mishchenko2022asynchronous}} & 92.60 & 5.85e-07 \\
    \midrule
    \makecell{\algname{Soviet PAGE} \\ \small\citep{li2021page}} & 92.31 & 1.62e-07 \\
    \midrule
    \makecell{\algname{Rennala SGD} \\ \small\citep{tyurin2023optimal}} & 92.37 & 3.12e-06 \\
    \midrule
    \begin{tabular}{c} \algname{\textbf{Freya PAGE}} \end{tabular} & \begin{tabular}{c} \textbf{92.66} \end{tabular} & \begin{tabular}{c} \textbf{1.01e-07} \end{tabular}\\
    \bottomrule
    \end{tabular}
\end{threeparttable}
\end{table}

We now consider the logistic regression problem on the \texttt{MNIST} dataset \citep{lecun2010mnist}, where each algorithm samples one data point at a time. The results of the experiments are presented in Figure~\ref{fig:log}. The difference between \algname{Freya PAGE} and \algname{Rennala SGD}/\algname{Asynchronous SGD} is not as pronounced as in Section~\ref{sec:exp_quad}: the methods have almost the same performance for this particular problem. However, our method still outperforms its competitors in the low accuracy regime, and is significantly better than \algname{Soviet PAGE}.

A critical disadvantage of \algname{Rennala SGD} and \algname{Asynchronous SGD} is their noisy behavior, evident in both the plots and reflected in higher variance of the accuracy (see Table~\ref{tbl:acc}). In contrast, the iterations of \algname{Freya PAGE} in Figure~\ref{fig:log} are smooth, and its accuracy exhibits the lowest variance, as shown in Table~\ref{tbl:acc}. This stability can be attributed to the variance-reduction nature of \algname{Freya PAGE}.

\newpage

\section{The Time Complexity Guarantees of Algorithms~\ref{algorithm:ga_asynch_main_diff} and \ref{algorithm:ga_asynch_batch}}

In addition to \algname{ComputeBatchDifference} (Algorithm~\ref{algorithm:ga_asynch_main_diff}) introduced in the main part, we also analyze \algname{ComputeBatch} (Algorithm~\ref{algorithm:ga_asynch_batch}) that is similar to \algname{ComputeBatchDifference}, but calculates a minibatch of stochastic gradients $\nabla f_i(x)$ instead of stochastic gradient differences $\nabla f_i(x) - \nabla f_i(y).$

\begin{algorithm}[h]
    \centering
    \caption{\algname{ComputeBatch}($S$, $x$)}
    \label{algorithm:ga_asynch_batch}
    \begin{algorithmic}[1]
    \State \textbf{Input:} batch size $S\in\N$, point $x\in\R^d$
    \State Init $g=0\in\R^d$
    \State Broadcast $x$ to all workers
    \State For each worker, sample $j$ from $[m]$ (uniformly) and ask it to calculate $\nabla f_j(x)$
        \For{$i = 1, 2, \ldots, S$}
            \State Wait for $\nabla f_p(x)$ from a worker
            \State $g \gets g + \frac{1}{S} \nabla f_p(x)$
            \State Sample $j$ from $[m]$ (uniformly)
            and ask this worker to calculate $\nabla f_j(x)$
        \EndFor
        \State Return $g$
    \end{algorithmic}
\end{algorithm}

\THMBATHDIFF*

\begin{proof}
    Let
    \begin{align*}
        t = 4 \min_{j\in[n]} \parens{\parens{\sum_{i=1}^j \frac{1}{\tau_i}}^{-1} (S + j)}.
    \end{align*}
    As soon as some worker finishes calculating the stochastic gradient difference, it immediately starts computing the difference of the next pair. Hence, by the time $t,$ all workers will have processed at least
    \begin{align*}
        \sum_{i=1}^{n} \flr{\frac{t}{2 \tau_i}}
    \end{align*}
    pairs. Assume that
    \begin{align*}
        j^* = \arg\min_{j\in[n]} \parens{\parens{\sum_{i=1}^j \frac{1}{\tau_i}}^{-1} (S + j)}.
    \end{align*}
    Using Lemma~\ref{lemma:techn1}, we have $t \geq 4 \tau_i$ for all $i \leq j^*.$ Thus, we get $\frac{t}{2 \tau_i} \geq 1$ for all $i \leq j^*$ and 
    \begin{align*}
        &\sum_{i=1}^{n} \flr{\frac{t}{2 \tau_i}} \geq \sum_{i=1}^{j^*} \flr{\frac{t}{2 \tau_i}} \geq \sum_{i=1}^{j^*} \frac{t}{4 \tau_i} \\
        &= \frac{1}{4} \left(\sum_{i=1}^{j^*} \frac{1}{\tau_i}\right) \left(4 \parens{\sum_{i=1}^{j^*} \frac{1}{\tau_i}}^{-1} (S + j^*)\right) = S + j^* \geq S.
    \end{align*}
    We can conclude that by the time \eqref{eq:bath_diff_time}, the algorithm will have calculated $S$ pairs of stochastic gradients and exited the loop.
\end{proof}

\begin{theorem}\label{thm:batch}
    The time needed by Algorithm~\ref{algorithm:ga_asynch_batch} to calculate $g$ is at most
    \begin{align}
        2 \min_{j\in[n]} \parens{\parens{\sum_{i=1}^j \frac{1}{\tau_i}}^{-1} (S + j)}
    \end{align}
    seconds.
\end{theorem}

\begin{proof}
    The proof of this theorem is essentially the same as the proof of Theorem~\ref{thm:bath_diff}. The only difference is that Algorithm~\ref{algorithm:ga_asynch_batch} calculates $\nabla f_i(x)$ instead of $\nabla f_i(x) - \nabla f_i(y).$
\end{proof}

\section{The Time Complexity Guarantees of Algorithms~\ref{algorithm:ga_asynch_main_new} and \ref{algorithm:ga_asynch_main_new_full}}
\label{sec:time_compl_sampling}

Instead of Algorithm~\ref{algorithm:ga_asynch_main_new}, we analyze a more general Algorithm~\ref{algorithm:ga_asynch_main_new_full} that reduces to Algorithm~\ref{algorithm:ga_asynch_main_new} when $\cS = [m].$

\begin{algorithm}[H]
    \centering
    \caption{\algname{ComputeBatchAnySampling}($\cS$, $x$)}
    \label{algorithm:ga_asynch_main_new_full}
    \begin{algorithmic}[1]
    \State \textbf{Input:} multiset $\cS$, point $x\in\R^d$
    \State Init $g=0\in\R^d$, multiset $\cM = \emptyset$
    \State Broadcast $x$ to all workers
    \State For each worker, sample $j$ from $\cS$ (uniformly) and ask it to calculate $\nabla f_j(x)$ \label{line:sample_one}
        \While{$\cM \neq \cS$}
            \State Wait for $\nabla f_p(x)$ from a worker
            \If{$p \in \cS \backslash \cM$}
            \State $g \gets g + \frac{1}{\abs{\cS}} \nabla f_p(x)$
            \State Update $\cM \gets \cM \cup \{p\}$
            \EndIf
            \State Sample $j$ from $\cS \backslash \cM$ (uniformly) \label{line:sample_two}
            and ask this worker to calculate $\nabla f_{j}(x)$
        \EndWhile
        \State Return $g = \frac{1}{\abs{\cS}} \sum\limits_{i \in \cS} \nabla f_i(x)$
    \end{algorithmic}
\end{algorithm}

\begin{restatable}{theorem}{THMFULLGRADTIMEFULL}
    \label{thm:compute_batch}
    The expected time needed by Algorithm~\ref{algorithm:ga_asynch_main_new_full} to calculate $g = \frac{1}{\abs{\cS}} \sum\limits_{i \in \cS} \nabla f_i$ is at most
    \begin{align}
        \label{eq:COVSRDFTGtaGOfFWmaO}
        12 \min_{j\in[n]} \parens{\parens{\sum_{i=1}^j \frac{1}{\tau_i}}^{-1} (\abs{\cS} + \min\{\abs{\cS},n\} \log \left(\min\{\abs{\cS},n\}\right) + j)}
    \end{align}
    seconds. 
\end{restatable}

\textit{Proof Sketch:} While the following proof is technical, the intuition and idea behind it and the algorithm are relatively simple. For simplicity, assume that $n \geq \abs{\cS}.$ The set $\cS \backslash \cM$ includes all indices that have not yet been calculated. Each worker is assigned a new random index from $\cS \backslash \cM$ and starts the calculation of the gradient. At the beginning of the algorithm, when the set $\cS \backslash \cM$ is large, the probability that two workers are assigned the same index is very small. Hence, using the same idea as in the proof of Theorem~\ref{thm:bath_diff}, the workers will calculate $\approx \abs{\cS} - n$ stochastic gradients after 
\begin{align*}
    \approx \min_{j\in[n]} \parens{\parens{\sum_{i=1}^j \frac{1}{\tau_i}}^{-1} (\abs{\cS} - n + j)}
\end{align*}
seconds. However, once the size of $\cS \backslash \cM$ becomes roughly equal to $n$, the probability that two workers sample the same index increases. In the final steps of the algorithm, we encounter the same issue as in the famous \emph{coupon collector's problem}, resulting in an additional factor of $n \log n$ because some stochastic gradients will be calculated multiple times.

\begin{proof}
    Let us define $S = \abs{\cS}$ and take any $k \in [n].$ We refer to the workers with the upper bounds $\tau_i$ such that $\tau_i \leq \tau_{k}$ as ``fast'', and the others will be termed ``slow''.

    Consider the moment when the algorithm samples $j$ from the set $\cS \backslash \cM$ to allocate it to one of the ``fast'' workers (\begin{NoHyper}Line~\ref{line:sample_one} or \ref{line:sample_two}\end{NoHyper}).
    The probability of sampling $j$ such that $\nabla f_j(x)$ is currently being calculated by another ``fast'' worker is
    \begin{align*}
        &\frac{\abs{\{\textnormal{indices from $\cS \backslash \cM$ taken by ``fast'' workers}\}}}{\abs{\{\textnormal{indices from $\cS \backslash \cM$ taken by ``fast'' workers}\}} + \abs{\{\textnormal{indices from $\cS \backslash \cM$ not taken by ``fast'' workers}\}}} \\
        &\leq \frac{\min\{k, S\}}{\min\{k, S\} + \abs{\{\textnormal{indices from $\cS \backslash \cM$ not taken by ``fast'' workers}\}}}
    \end{align*}
    because there are at most $k$ ``fast'' workers and at most $S$ distinct stochastic gradients. Let us define the set
    \begin{align*}
        \cU \eqdef \{\textnormal{indices from $\cS \backslash \cM$ not taken by ``fast'' workers}\}.
    \end{align*}
    A ``fast'' worker can be ``unlucky'' and start calculating a stochastic gradient that is being computed by another ``fast'' worker.
    However, with probability at least
    \begin{align*}
        \geq \frac{\abs{\cU}}{\min\{k, S\} + \abs{\cU}},
    \end{align*}
    it will take a new index $j$ that was not previously taken by another ``fast'' worker.

    Thus, the while loop of the algorithm defines a Markov process that begins with some $\cU \subseteq \cS.$ The size of $\cU$ decreases by one with probability at least $\frac{\abs{\cU}}{\min\{k, S\} + \abs{\cU}}$ in iterations where the algorithm samples $j$ from $\cU$ and asks a ``fast'' worker to calculate the stochastic gradient. Additionally, the size of $\cU$ can decrease by one when a ``slow'' worker finishes calculating a stochastic gradient from $\cU$.

    Let $\bar{t}$ be the time required for the Markov process to reach the state $\cU = \emptyset.$ Then, the while loop in Algorithm~\ref{algorithm:ga_asynch_main_new} will finish after at most 
    \begin{align}
    \label{eq:ItxNEbvoIPCNKKusLH}
    \bar{T} \eqdef \bar{t} + \tau_{k}
    \end{align} 
    seconds because once $\cU = \emptyset,$ all non-processed indices from $\cS \backslash \cM$ are assigned to the ``fast'' workers, so calculating the remaining stochastic gradients will take at most $\tau_{k}$ seconds.

    It remains to estimate $\bar{t}.$ Let $\eta_p$ be the number of iterations of the while loop where the algorithm samples $j$ from $\cS \backslash \cM$ and asks a ``fast'' worker to calculate the stochastic gradient when $\abs{\cU} = p.$ By the definition of the Markov chain, we have
    \begin{align}
        \label{eq:hzHdbQumZOGv}
        \Exp{\eta_p} \leq \left(\frac{p}{\min\{k, S\} + p}\right)^{-1} = 1 + \frac{\min\{k, S\}}{p}
    \end{align}
    because with probability at least $\frac{p}{\min\{k, S\} + p},$ one of the (``lucky'') ``fast'' workers receives $j$ from $\cU$ and decreases the size of $\cU$ by $1$ ($\eta_p$ has a geometric-like distribution).

    Since $\abs{\cU} \leq S$ at the beginning of the while loop, it is sufficient for the ``fast'' workers to calculate at most 
    \begin{align*}
        \sum_{p=1}^{S} \left(\eta_p + 1\right)
    \end{align*}
    stochastic gradients to ensure that $\cU = \emptyset$ (it is possible that some stochastic gradients will be calculated many times). Indeed, if $\abs{\cU} = p$ for the first moment, then after $\eta_p + 1$ calculations of stochastic gradients by the ``fast'' workers, the size of set will be at most $p - 1.$ The last ``plus one'' calculation can only happen when $\abs{\cU} = p - 1.$

    The time required for the ``fast'' workers to process this number of stochastic gradients is at most
    \begin{align*}
        t' = 2 \parens{\sum_{i=1}^{k} \frac{1}{\tau_i}}^{-1} \left(\sum_{p=1}^{S} (\eta_p + 1)\right) + \tau_{k},
    \end{align*}
    because for this choice of $t'$, we have
    \begin{align*}
        \sum_{i=1}^{k} \flr{\frac{t'}{\tau_i}} \geq \frac{1}{2} \sum_{i=1}^{k} \frac{t'}{\tau_i} \geq \sum_{p=1}^{S} (\eta_p + 1),
    \end{align*}
    where $\flr{\frac{t'}{\tau_i}}$ is the number of stochastic gradients that worker $i$ can calculate in $t'$ seconds.
    Taking expectation gives
    \begin{align*}
        \Exp{t'}
        &= 2 \parens{\sum_{i=1}^{k} \frac{1}{\tau_i}}^{-1} \left(\sum_{p=1}^{S} \Exp{\eta_p + 1}\right) + \tau_{k} \\
        &\overset{\eqref{eq:hzHdbQumZOGv}}{\leq} 2 \parens{\sum_{i=1}^{k} \frac{1}{\tau_i}}^{-1} \left(2 S + \sum_{p=1}^{S} \frac{\min\{k, S\}}{p} \right) + \tau_{k} \\
        &\leq 2 \parens{\sum_{i=1}^{k} \frac{1}{\tau_i}}^{-1} \left(2 S + \sum_{p=1}^{S} \frac{\min\{n, S\}}{p} \right) + \tau_{k} \\
        &= 2 \parens{\sum_{i=1}^{k} \frac{1}{\tau_i}}^{-1} \left(2 S + \sum_{p=\min\{n, S\} + 1}^{S} \frac{\min\{n, S\}}{p} + \min\{n, S\} \sum_{p=1}^{\min\{n, S\}} \frac{1}{p} \right) + \tau_{k} \\
        &\leq 2 \parens{\sum_{i=1}^{k} \frac{1}{\tau_i}}^{-1} \left(3 S + \min\{n, S\} \sum_{p=1}^{\min\{n, S\}} \frac{1}{p}\right) + \tau_{k} \\
        &\leq 2 \parens{\sum_{i=1}^{k} \frac{1}{\tau_i}}^{-1} \left(3 S + \min\{n, S\} \left(2 + \log\left(\min\{n, S\}\right)\right)\right) + \tau_{k} \\
        &\leq 10 \parens{\sum_{i=1}^{k} \frac{1}{\tau_i}}^{-1} \left(S + \min\{n, S\} \log\left(\min\{n, S\}\right)\right) + \tau_{k},
    \end{align*}
    where we use the standard bound on the harmonic series. Thus, the expectation of the total time \eqref{eq:ItxNEbvoIPCNKKusLH} can be bounded by
    \begin{align*}
        \Exp{\bar{T}} 
        &\leq 10 \parens{\sum_{i=1}^{k} \frac{1}{\tau_i}}^{-1} \left(S + \min\{n, S\} \log\left(\min\{n, S\}\right)\right) + 2 \tau_{k} \\
        &\leq 10 \parens{\sum_{i=1}^{k} \frac{1}{\tau_i}}^{-1} \left(S + \min\{n, S\} \log\left(\min\{n, S\}\right) + k\right) + 2 \tau_{k},
    \end{align*}
    where in the last line we add $k \geq 0.$
    Recall that $k$ is a parameter we can choose. Let us take
    \begin{align*}
        k = \arg\min_{j\in[n]} \parens{\sum_{i=1}^{j} \frac{1}{\tau_i}}^{-1} \left(S + \min\{n, S\} \log\left(\min\{n, S\}\right) + j\right).
    \end{align*}
    Using Lemma~\ref{lemma:techn1}, we have
    \begin{align*}
        \tau_{k} \leq \min_{j \in [n]}\parens{\sum_{i=1}^{j} \frac{1}{\tau_i}}^{-1} \left(S + \min\{n, S\} \log\left(\min\{n, S\}\right) + j\right)
    \end{align*}
    and hence
    \begin{align*}
        \Exp{\bar{T}} 
        \leq 12 \min_{j \in [n]} \left(\parens{\sum_{i=1}^{j} \frac{1}{\tau_i}}^{-1} \left(S + \min\{n, S\} \log\left(\min\{n, S\}\right) + j\right)\right).
    \end{align*}

\end{proof}

\newpage

\section{Proofs for Algorithm \ref{algorithm:rennala_page_gen} (\algnamelarge{Freya PAGE})}\label{sec:proofs_rennala_page}

The proofs use the simplified notation $\Teq(S) \eqdef \Teq(S, [\tau_i]_{i=1}^n)$ from Definition~\ref{def:time_budget}.

Since the update rule of \algname{PAGE} coincides with that of \algname{Freya PAGE}, one can directly apply the iteration complexity results established in \citet{tyurin2022sharper}.

\THMPAGERENNALAINDEPITER*

\begin{proof}
    The result follows from Theorem 6 of \citet{tyurin2022sharper}, using the parameters from the ``Uniform With Replacement'' line in Table 1 of the same work.
\end{proof}

\THMPAGERENNALATIME*
\begin{proof}
    The result established in Theorem \ref{thm:page_rate} says that the iteration complexity of the algorithm is
    \begin{align*}
        K_{\algname{PAGE}} \eqdef \frac{2 \delta^0}{\varepsilon} \parens{L_- + L_{\pm} \sqrt{\frac{1-p}{pS}}}.
    \end{align*}
    At each iteration, with probability $1-p$, the workers compute $S$ differences of stochastic gradients, which by Theorem~\ref{thm:bath_diff} takes
    \begin{align*}
        4 \min_{j\in[n]} \parens{\parens{\sum_{i=1}^j \frac{1}{\tau_i}}^{-1} (S+j)}
    \end{align*}
    seconds. Otherwise, they collect the full gradient, which can be done (Theorem \ref{thm:full_grad_time}) in 
    \begin{align*}
        &12 \min_{j\in[n]} \parens{\parens{\sum_{i=1}^j \frac{1}{\tau_i}}^{-1} (m + \min\{m,n\} \log \left(\min\{m,n\}\right) + j)} \\
        &\leq 24 \min_{j\in[n]} \parens{\parens{\sum_{i=1}^j \frac{1}{\tau_i}}^{-1} (m+j)}
    \end{align*}
    seconds, where the inequality uses Assumption~\ref{ass:m_is_large}.
    Hence, recalling the notation
    \begin{align*}
        \Teq(S) \eqdef \min_{j\in[n]} \parens{\parens{\sum_{i=1}^j \frac{1}{\tau_i}}^{-1} (S+j)},
    \end{align*}
    the (expected) time complexity of the method is
    \begin{align*}
        T(p,S, [\tau_i]_{i=1}^n) &= 24 \Teq(m) + 24 K_{\algname{PAGE}} \times \parens{p \Teq(m) + (1-p) \Teq(S)},
    \end{align*}
    where the term $\Teq(m)$ corresponds to the preprocessing step, when the algorithm needs to calculate $g^0 = \nabla f(x^0) = \nicefrac{1}{m} \sum_{i=1}^m \nabla f_i(x^0)$.
\end{proof}

\begin{theorem}
    \label{thm:opt_p_indep}
    Up to a constant factor, the time complexity $T(p,S, [\tau_i]_{i=1}^n)$ from \eqref{eq:compl_p_S_indep} is at least 
    \begin{align}
        \label{eq:T_S_indep}
        \Teq(m) + \frac{\delta^0}{\varepsilon}  \min\left\{L_- \Teq(m), L_- \Teq(S) + L_{\pm} \sqrt{\frac{\Teq(m)\Teq(S)}{S}} \right\},
    \end{align}
    and attains this value with 
    \begin{align}\label{eq:optimal_p_indep}
        p^*(S) = \begin{cases}
            1, & L_- \Teq(m) \leq L_- \Teq(S) + L_{\pm} \sqrt{\frac{\Teq(m)\Teq(S)}{S}} \\
            \frac{\Teq(S)}{\Teq(m)}, & \textnormal{otherwise}.
        \end{cases}
    \end{align}
\end{theorem}

\begin{proof}
    Up to a constant factor, by Theorem \ref{thm:page_time_compl}, the time complexity of \algname{Freya PAGE} is
    \begin{align*}
        T(p,S, [\tau_i]_{i=1}^n) \eqdef \Teq(m) + \frac{\delta^0}{\varepsilon} \parens{L_- + L_{\pm} \sqrt{\frac{1-p}{p S}}} 
        \parens{p \Teq(m) + (1-p) \Teq(S)}.
    \end{align*}
    Let us denote the second term in the above equation as $T_{p,S}$. Then for all $p \geq \frac{1}{2},$ we have
    \begin{align}\label{eq:tpropto_indep}
        T_{p,S} 
        &\propto \frac{\delta^0}{\varepsilon} \parens{L_- + L_{\pm} \sqrt{\frac{1-p}{pS}}} \parens{p \Teq(m) + (1-p) \Teq(S)}
        \geq \frac{\delta^0}{2\varepsilon} L_- \Teq(m),
    \end{align}
    and for all $S \geq \frac{m}{2}$
    \begin{align*}
        T_{p,S} 
        &\geq \frac{\delta^0}{\varepsilon} L_- \brac{p \Teq(m)
        + (1-p) \min_{j\in[n]} \parens{\parens{\sum_{i=1}^j \frac{1}{\tau_i}}^{-1} \parens{\frac{m}{2}+j}}}
        \geq \frac{\delta^0}{2 \varepsilon} L_- \Teq(m).
    \end{align*}
    Otherwise, when $p < \frac{1}{2}$ and $S < \frac{m}{2}$, we have
    \begin{align*}
        T_{p,S} 
        &\geq \frac{\delta^0}{2 \varepsilon} \parens{L_- + L_{\pm} \sqrt{\frac{1}{p S}}} \parens{p \Teq(m) + \Teq(S)} \\
        &\geq \frac{\delta^0}{2 \varepsilon} L_- \Teq(S) + \frac{\delta^0}{2 \varepsilon} \parens{L_{\pm} \sqrt{\frac{1}{p S}}} \parens{p \Teq(m) + \Teq(S)} \\
        &\geq \frac{\delta^0}{2 \varepsilon} L_- \Teq(S) + \frac{\delta^0}{\varepsilon} \parens{L_{\pm} \sqrt{\frac{1}{p S}}} \sqrt{p \Teq(m)\Teq(S)} \\
        &= \frac{\delta^0}{2 \varepsilon} L_- \Teq(S) + \frac{\delta^0}{\varepsilon} L_{\pm} \sqrt{\frac{\Teq(m)\Teq(S)}{S}}.
    \end{align*}
    Hence, up to a constant factor,
    \begin{align*}
        T(p,S, [\tau_i]_{i=1}^n)
        = \Teq(m) + T_{p,S}
        \geq \Teq(m) + \frac{\delta^0}{\varepsilon}  \min\left\{L_- \Teq(m), L_- \Teq(S) + L_{\pm} \frac{\sqrt{\Teq(m)\Teq(S)}}{\sqrt{S}} \right\}.
    \end{align*}
    It can be easily verified that this bound can be attained (up to a constant factor) using the parameter $p$ as defined in~\eqref{eq:optimal_p_indep}.
\end{proof}

\begin{theorem}\label{thm:opt_S}
    Up to a constant factor, the minimum of the time complexities \eqref{eq:compl_p_S_indep} and \eqref{eq:T_S_indep} is
    \begin{align*}
        \Teq(m) + \frac{\delta^0}{\varepsilon}  \min\left\{L_- \Teq(m), \min_{S \in [m]} \left[L_- \Teq(S) + L_{\pm} \frac{\sqrt{\Teq(m)\Teq(S)}}{\sqrt{S}}\right] \right\}
    \end{align*}
    and is achieved for 
    \begin{align*}
        S^* = \arg\min_{S \in [m]} \min\left\{L_- \Teq(m), L_- \Teq(S) + L_{\pm} \frac{\sqrt{\Teq(m)\Teq(S)}}{\sqrt{S}} \right\}
    \end{align*}
    and $p^* = p^*(S^*),$ where $p^*(S)$ is defined in \eqref{eq:optimal_p_indep}.
\end{theorem}
\begin{proof}
    This is a simple corollary of Theorem~\ref{thm:opt_p_indep}: we take $S$ that minimizes \eqref{eq:T_S_indep}.
\end{proof}

In certain scenarios, we can derive the optimal parameter values explicitly.

\begin{theorem}\label{thm:rpage_psopt_indep}
    \leavevmode
    \begin{enumerate}
        \item If $n \leq \frac{L_{\pm} \sqrt{m}}{L_-} \leq m$, then (up to constants) $S^* = \frac{L_{\pm} \sqrt{m}}{L_-}$ and $p^* = \frac{L_{\pm}}{L_- \sqrt{m}}$ are optimal parameters and
        \begin{align*}
            T(p^*,S^*,[\tau_i]_{i=1}^n) = \Teq(m) + \frac{\delta^0 L_-}{\varepsilon}
            \Teq\parens{\frac{L_{\pm} \sqrt{m}}{L_-}}.
        \end{align*}
        \item If $\frac{L_{\pm}\sqrt{m}}{L_-} \leq 1$, then (up to constants) $S^* = 1$ and $p^* = \nicefrac{1}{m}$ are optimal parameters and
        \begin{align*}
            T(p^*,S^*,[\tau_i]_{i=1}^n) = \Teq(m) + \frac{\delta^0 L_-}{\varepsilon} \Teq(1).
        \end{align*}
        \item If $\frac{L_{\pm}\sqrt{m}}{L_-} \geq m$, then (up to constants) $p^* = 1$ is an optimal parameter and
        \begin{align*}
            T(p^*,S,[\tau_i]_{i=1}^n) = \frac{\delta^0 L_-}{\varepsilon} \Teq(m)
        \end{align*}
        for any $S\in[m]$.
    \end{enumerate}
\end{theorem}

\begin{proof}
    We fist consider the case when $n \leq \frac{L_{\pm} \sqrt{m}}{L_-} \leq m$.
    Since $j \leq n \leq \frac{L_{\pm} \sqrt{m}}{L_-}$, we have
    \begin{align}\label{eq:ineq1tm_indep}
        \frac{L_{\pm} \sqrt{m}}{L_-} + \frac{L_{\pm}}{L_- \sqrt{m}} j \geq \frac{1}{2} \left(\frac{L_{\pm} \sqrt{m}}{L_-} + j\right),
    \end{align}
    and from the assumption that $L_- \geq \frac{L_{\pm}}{\sqrt{m}}$ it follows that
    \begin{align}\label{eq:ineq2tm_indep}
        \frac{L_{\pm} \sqrt{m}}{L_-} + j \geq \frac{L_{\pm} \sqrt{m}}{L_-} + \frac{L_{\pm}}{L_- \sqrt{m}} j
    \end{align}
    for all $j \in [n]$. Thus,
    \begin{align*}
        \frac{L_{\pm}}{L_-\sqrt{m}} \Teq(m)
        = \min_{j\in[n]} \parens{\parens{\sum_{i=1}^j \frac{1}{\tau_i}}^{-1} \parens{\frac{L_{\pm} \sqrt{m}}{L_-} + \frac{L_{\pm}}{L_-\sqrt{m}} j}}
        \overset{\eqref{eq:ineq1tm_indep}}{\geq} \frac{1}{2} \Teq\parens{\frac{L_{\pm} \sqrt{m}}{L_-}},
    \end{align*}
    and
    \begin{align*}
        \frac{L_{\pm}}{L_-\sqrt{m}} \Teq(m)
        = \min_{j\in[n]} \parens{\parens{\sum_{i=1}^j \frac{1}{\tau_i}}^{-1} \parens{\frac{L_{\pm} \sqrt{m}}{L_-} + \frac{L_{\pm}}{L_-\sqrt{m}} j}}
        \overset{\eqref{eq:ineq2tm_indep}}{\leq} \Teq\parens{\frac{L_{\pm} \sqrt{m}}{L_-}}.
    \end{align*}
    It follows that
    \begin{align}
        \label{eq:ineq3tmt_indep}
        \frac{1}{2} \Teq\parens{\frac{L_{\pm} \sqrt{m}}{L_-}} \leq \frac{L_{\pm}}{L_-\sqrt{m}} \Teq(m) \leq \Teq\parens{\frac{L_{\pm} \sqrt{m}}{L_-}}.
    \end{align}
    Since $\frac{L_{\pm} \sqrt{m}}{L_-} \leq m$,  we have
    \begin{align}\label{eq:toptin_indep}
        \min_{S \in [m]} \left\{L_- \Teq(S) + L_{\pm} \frac{\sqrt{\Teq(m)\Teq(S)}}{\sqrt{S}}\right\}
        &\leq L_- \min_{S \in [m]} \left\{\Teq(S) + \sqrt{m} \frac{\sqrt{\Teq(m)\Teq(S)}}{\sqrt{S}}\right\} \nonumber \\
        &\leq 2 L_- \Teq(m),
    \end{align}
    and thus, according to the result from Theorem \ref{thm:opt_S}, it is sufficient to minimize
    \begin{align*}
        t'(S) \eqdef L_- \Teq(S) + L_{\pm} \frac{\sqrt{\Teq(m)\Teq(S)}}{\sqrt{S}}.
    \end{align*}
    First, let us note that
    \begin{align*}
        t'\parens{\frac{L_{\pm} \sqrt{m}}{L_-}} %&= L_- \Teq\parens{\frac{L_{\pm} \sqrt{m}}{L_-}} + L_{\pm} \sqrt{\frac{\Teq(m) \Teq\parens{\frac{L_{\pm} \sqrt{m}}{L_-}}}{\frac{L_{\pm} \sqrt{m}}{L_-}}} \\
        &= L_- \Teq\parens{\frac{L_{\pm} \sqrt{m}}{L_-}}
        + L_{\pm} \sqrt{\frac{L_-}{L_{\pm} \sqrt{m}} \Teq(m) \Teq\parens{\frac{L_{\pm} \sqrt{m}}{L_-}}} \\
        &\overset{\eqref{eq:ineq3tmt_indep}}{\leq} L_- \Teq\parens{\frac{L_{\pm} \sqrt{m}}{L_-}}
        + L_{\pm} \sqrt{\frac{L_-}{L_{\pm} \sqrt{m}} \frac{L_-\sqrt{m}}{L_{\pm}} \Teq\parens{\frac{L_{\pm} \sqrt{m}}{L_-}} \Teq\parens{\frac{L_{\pm} \sqrt{m}}{L_-}}} \\
        &= 2 L_- \Teq\parens{\frac{L_{\pm} \sqrt{m}}{L_-}}.
    \end{align*}
    If $S \geq \frac{L_{\pm} \sqrt{m}}{L_-}$, then
    \begin{align*}
        t'(S) = L_- \Teq(S) + L_{\pm} \sqrt{\frac{\Teq(m)\Teq(S)}{S}} \geq L_- \Teq\parens{\frac{L_{\pm} \sqrt{m}}{L_-}}.
    \end{align*}
    Otherwise, if $S < \frac{L_{\pm} \sqrt{m}}{L_-}$, then
    \begin{align*}
        t'(S) &= L_- \Teq(S) + L_{\pm} \sqrt{\Teq(m)} \sqrt{\min_{j\in[n]} \parens{\parens{\sum_{i=1}^j \frac{1}{\tau_i}}^{-1} \parens{1+ \frac{j}{S}}}} \\
        &\geq L_{\pm} \sqrt{\Teq(m)} \sqrt{\min_{j\in[n]} \parens{\parens{\sum_{i=1}^j \frac{1}{\tau_i}}^{-1} \parens{1+ \frac{L_- j}{L_{\pm} \sqrt{m}}}}} \\
        &= L_{\pm} \sqrt{\Teq(m)} \sqrt{\frac{L_-}{L_{\pm} \sqrt{m}} \min_{j\in[n]} \parens{\parens{\sum_{i=1}^j \frac{1}{\tau_i}}^{-1} \parens{\frac{L_{\pm} \sqrt{m}}{L_-} + j}}} \\
        &= L_{\pm} \sqrt{\Teq(m) \frac{L_-}{L_{\pm} \sqrt{m}} \Teq\parens{\frac{L_{\pm} \sqrt{m}}{L_-}}} \\
        &\overset{\eqref{eq:ineq3tmt_indep}}{\geq} L_{\pm} \sqrt{\frac{1}{2} \Teq\parens{\frac{L_{\pm} \sqrt{m}}{L_-}} \frac{L_-\sqrt{m}}{L_{\pm}} \frac{L_-}{L_{\pm} \sqrt{m}} \Teq\parens{\frac{L_{\pm} \sqrt{m}}{L_-}}} \\
        &= \frac{L_-}{\sqrt{2}} \Teq\parens{\frac{L_{\pm} \sqrt{m}}{L_-}}.
    \end{align*}
    Therefore, the optimal choice is $S^* = \frac{L_{\pm} \sqrt{m}}{L_-}$, and by Theorem \ref{thm:opt_p_indep} and inequality \eqref{eq:toptin_indep}, $p$ should chosen to be
    \begin{align*}
        \frac{\Teq\parens{\frac{L_{\pm} \sqrt{m}}{L_-}}}{\Teq(m)}.
    \end{align*}
    Using \eqref{eq:ineq3tmt_indep}, we can conclude that $p^* = \frac{L_{\pm}}{L_- \sqrt{m}}$ is optimal, proving the first part of the Theorem.

    Next, consider the case when $\frac{L_{\pm}\sqrt{m}}{L_-} \leq 1$.
    By the reasoning above, it is sufficient to minimize
    \begin{align*}
        t'(S) \eqdef L_- \Teq(S) + L_{\pm} \frac{\sqrt{\Teq(m)\Teq(S)}}{\sqrt{S}}.
    \end{align*}
    First, let us note that
    \begin{align*}
        t'(1) &= L_- \Teq(1) + L_{\pm} \sqrt{\Teq(m)\Teq(1)}
        \leq L_- \Teq(1) + L_- \sqrt{\frac{\Teq(m)\Teq(1)}{m}}
        \leq 2 L_- \Teq(1),
    \end{align*}
    where the last inequality follows from the fact that for any $S\in[m]$, $\nicefrac{\Teq(S)}{S} \leq \Teq(1)$.
    On the other hand, if $S \geq 1$, then
    \begin{align*}
        t'(S) = L_- \Teq(S) + L_{\pm} \frac{\sqrt{\Teq(m)\Teq(S)}}{\sqrt{S}}
        \geq L_- \Teq(S)
        \geq L_- \Teq(1).
    \end{align*}
    Therefore, the optimal choice is $S^* = 1$. Then
    \begin{align*}
        &\frac{\delta^0}{\varepsilon} \parens{L_- + L_{\pm} \sqrt{m-1}} \parens{\frac{\Teq(m)}{m} + \parens{1-\frac{1}{m}} \Teq(1)} \\
        &\leq \frac{2\delta^0}{\varepsilon} \parens{L_- + \frac{L_-}{\sqrt{m}} \sqrt{m-1}} \Teq(1) \\
        &\leq \frac{4\delta^0 L_-}{\varepsilon} \Teq(1),
    \end{align*}
    while for any $p$
    \begin{align*}
        &\frac{\delta^0}{\varepsilon} \parens{L_- + L_{\pm} \sqrt{\frac{1-p}{p}}} \parens{p \Teq(m) + (1-p) \Teq(1)} \\
        &\geq \frac{\delta^0 L_-}{\varepsilon} \parens{p \Teq(m) + (1-p) \Teq(1)} \\
        &\geq \frac{\delta^0 L_-}{\varepsilon} \parens{p \Teq(1) + (1-p) \Teq(1)} \\
        &= \frac{\delta^0 L_-}{\varepsilon} \Teq(1).
    \end{align*}
    Hence $p^*=\nicefrac{1}{m}$.

    It remains to prove the third result. Suppose that $L_- < \frac{L_{\pm}}{\sqrt{m}}$. Then, the last part of the theorem follows from the fact that
    \begin{align*}
        \min_{S \in [m]} \left\{L_- \Teq(S) + L_{\pm} \frac{\sqrt{\Teq(m)\Teq(S)}}{\sqrt{S}}\right\} &\geq L_{\pm} \min_{S \in [m]} \left\{\frac{\sqrt{\Teq(m)\Teq(S)}}{\sqrt{S}}\right\} \\
        &= L_{\pm} \left\{\frac{\sqrt{\Teq(m)\Teq(m)}}{\sqrt{m}}\right\} \\
        &\geq L_- \Teq(m).
    \end{align*}
\end{proof}

In practice, the values of smoothness constants are often unknown. However, the algorithm can still be run with close to optimal parameters. 

\THEOREMSIMPLETIMECOMPLEXITYCHOICE*

\begin{proof}
    The proof is the same as in Theorem~\ref{thm:rpage_psopt_indep}. Indeed, 
    up to a constant factor, the time complexity~\eqref{eq:compl_p_S_indep} can be bounded as
    \begin{align*}
        T(p,S, [\tau_i]_{i=1}^n) &=  \Teq(m) + \frac{\delta^0}{\varepsilon} \parens{L_- + L_{\pm} \sqrt{\frac{1-p}{pS}}} \parens{p \Teq(m) + (1-p) \Teq(S)} \\
        &\leq  \Teq(m) + \frac{2 \delta^0 \max\{L_-,L_{\pm}\}}{\varepsilon} \parens{1 + \sqrt{\frac{1-p}{pS}}} \parens{p \Teq(m) + (1-p) \Teq(S)}.
    \end{align*}
    Therefore, by setting $L_{\pm} = L_-$ in Theorem \ref{thm:rpage_psopt_indep}, one can easily derive the parameters $p$ and $S$ that are optimal up to the smoothness constants. The time complexity \eqref{eq:time_comp} can be obtained by applying $S = \ceil{\sqrt{m}}$ and $p = \nicefrac{1}{\sqrt{m}}$ to \eqref{eq:compl_p_S_indep}.
\end{proof}

\newpage

\section{\algnamelarge{Freya PAGE} with Other Samplings}\label{sec:proofs_rennala_page_other}

Algorithm~\ref{algorithm:rennala_page_gen} can be adapted to accommodate other sampling methods. This brings us to the introduction of Algorithm~\ref{algorithm:rennala_page_gen_any}, which supports virtually any sampling strategy, formalized by the following mapping:

\begin{definition}[Sampling]\label{def:sampling}
    A \emph{sampling} is a random mapping $\mS_S$, which takes as an input a set of indices $\cI \eqdef \{a_1,\ldots,a_m\}$ and returns a (multi)set $\{a_{i_1}, \ldots, a_{i_S}\}$, where $a_{i_j} \in \cI$ for all $j\in[S]$.
\end{definition}

\begin{algorithm}[H]
    \caption{\algname{Freya PAGE} (with virtually any sampling)}
    \label{algorithm:rennala_page_gen_any}
    \begin{algorithmic}[1]
    \State \textbf{Parameters:} starting point $x^0 \in \R^d$, learning rate $\gamma > 0$, minibatch size $S$, sampling $\mS_S$, probability $p \in (0, 1]$, initialization $g^0 = \nabla f(x^0)$ using ${\green \textnormal{\algname{ComputeGradient}}([m], x^{0})}$ \quad (Alg.~\ref{algorithm:ga_asynch_main_new})
    \For{$k = 0, 1, \ldots, K-1$}
        \State $x^{k+1} = x^k - \gamma g^k$
        \State Sample $c^k \sim \text{Bernoulli}(p)$
        \If{$c^k=1$}
            \State {\green $\nabla f(x^{k+1}) = \textnormal{\algname{ComputeGradient}}(x^{k+1})$} \hfill (Alg.~\ref{algorithm:ga_asynch_main_new})
            \State $g^{k+1} = \nabla f(x^{k+1})$
        \Else
            \State Sample indices $\cS^k = \mS_S([m])$
            \State {\green $\frac{1}{S} \sum\limits_{i \in \cS^k}\left(\nabla f_i(x^{k+1}) - \nabla f_i(x^{k})\right)$} \newline \hspace*{1cm} {\green $= \textnormal{\algname{ComputeBatchDifferenceAnySampling}}(\cS^k, x^{k+1}, x^k)$} \hfill (Alg.~\ref{algorithm:ga_asynch_main_new_full_diff})
            \State $g^{k+1} = g^k + \frac{1}{S} \sum\limits_{i \in \cS^k}\left(\nabla f_i(x^{k+1}) - \nabla f_i(x^{k})\right)$ 
        \EndIf
    \EndFor
    \end{algorithmic}
\end{algorithm}

\begin{algorithm}[H]
    \centering
    \caption{\algname{ComputeBatchDifferenceAnySampling}($\cS$, $x$, $y$)}
    \label{algorithm:ga_asynch_main_new_full_diff}
    \begin{algorithmic}[1]
    \State \textbf{Input:} multiset $\cS$, points $x,y\in \R^d$
    \State Init $g=0\in \R^d$, multiset $\cM = \emptyset$
    \State Broadcast $x$ to all workers
    \State For each worker, sample $j$ from $\cS$ (uniformly) and ask it to calculate $\nabla f_j(x) - \nabla f_j(y)$
        \While{$\cM \neq \cS$}
            \State Wait for $\nabla f_p(x) - \nabla f_p(y)$ from a worker
            \If{$p \in \cS \backslash \cM$}
            \State $g \gets g + \frac{1}{\abs{\cS}} \left(\nabla f_p(x) - \nabla f_p(y)\right)$
            \State Update $\cM \gets \cM \cup \{p\}$
            \EndIf
            \State Sample $j$ from $\cS \backslash \cM$ (uniformly)
            and ask this worker to calculate $\nabla f_{j}(x) - \nabla f_{j}(y)$
        \EndWhile
        \State Return $g = \frac{1}{\abs{\cS}} \sum\limits_{i \in \cS} \left(\nabla f_i(x) - \nabla f_i(y)\right)$
    \end{algorithmic}
\end{algorithm}

The only difference is that instead of \algname{ComputeBatchDifference} (Algorithm~\ref{algorithm:ga_asynch_main_new_full}), Algorithm~\ref{algorithm:rennala_page_gen_any} uses a new subroutine, called \algname{ComputeBatchDifferenceAnySampling} (Algorithm~\ref{algorithm:ga_asynch_main_diff}). 

For this algorithm, we can prove the following time complexity guarantees.

\begin{restatable}{theorem}{THMFULLGRADTIMEFULLGENERIC}
    \label{thm:calculate_gradient_diff}
    The expected time needed by Algorithm~\ref{algorithm:ga_asynch_main_new_full_diff} to calculate $g = \frac{1}{\abs{\cS}} \sum\limits_{i \in \cS} (\nabla f_i(x) - \nabla f_i(y))$ is at most
    \begin{align}
        24 \min_{j\in[n]} \parens{\parens{\sum_{i=1}^j \frac{1}{\tau_i}}^{-1} (\abs{\cS} + \min\{\abs{\cS},n\} \log \left(\min\{\abs{\cS},n\}\right) + j)}
    \end{align}
    seconds. 
\end{restatable}

\begin{proof}
    The proof of this theorem is the same as the proof of Theorem~\ref{thm:compute_batch}. We only have to multiply~\eqref{eq:COVSRDFTGtaGOfFWmaO} by $2$ because Algorithm~\ref{algorithm:ga_asynch_main_diff} calculates $\nabla f_i(x) - \nabla f_i(y)$ instead of $\nabla f_i(x).$
\end{proof}

While changing the sampling strategy might affect the \emph{iteration complexity} of the method, for a fixed minibatch size $S$, the \emph{time complexity of a single iteration} remains unchanged. Thus, having established the expected time needed by the algorithm to perform a single iteration (i.e., to collect a minibatch of stochastic gradients of the required size), one can simply multiply it by the iteration complexity of the method determined for any supported sampling technique to obtain the resulting time complexity.

With this in mind, we now analyse the time complexity of Algorithm~\ref{algorithm:rennala_page_gen_any} with $2$ different sampling techniques: nice sampling and importance sampling. However, it can be analyzed with virtually any other unbiased sampling \citep{tyurin2022sharper}.

\subsection{Nice sampling}
\label{sec:nice}

\emph{Nice sampling} returns a random subset of fixed cardinality $S$ chosen uniformly from $[m]$. Unlike \emph{uniform sampling with replacement} used in Algorithm~\ref{algorithm:ga_asynch_main_new_full}, which returns a random multiset (that can include repetitions), the samples obtained by nice sampling are distinct. The iteration complexity of Algorithm~\ref{algorithm:rennala_page_gen_any} with nice sampling is given by the following theorem.

\begin{theorem}[\citet{tyurin2022sharper}, Section $3$]\label{thm:page_iter_nice}
    Let Assumptions \ref{as:smooth_lower_bdd}, \ref{as:L+} and \ref{as:L_pm} hold. Choose a minibatch size $S\in[m]$, a probability $p \in (0, 1]$ and the stepsize
    \begin{align*}
        \gamma = \frac{1}{L_- + L_{\pm} \sqrt{\frac{(1-p)(m-S)}{p(m-1)S}}}.
    \end{align*}
    Then, the number of iteration needed by Algorithm~\ref{algorithm:rennala_page_gen_any} with \emph{nice sampling} to reach an $\varepsilon$-stationary point is
    \begin{align}\label{eq:page_iter_nice}
        K = \frac{2 \delta^0}{\varepsilon} \left(L_- + L_{\pm} \sqrt{\frac{(1-p)(m-S)}{p(m-1)S}}\right).
    \end{align}
\end{theorem}
\begin{proof}
    The result follows from Theorem $6$ and Table 1 from \citep{tyurin2022sharper}.
\end{proof}

\begin{theorem}\label{thm:page_time_compl_nice}
    Consider the assumptions and parameters from Theorem~\ref{thm:page_iter_nice} and Assumption~\ref{ass:m_is_large}. Up to a constant factor, the time complexity of Algorithm~\ref{algorithm:rennala_page_gen_any} is
    \begin{align}\label{eq:compl_p_S_nice}
        &T(p,S, [\tau_i]_{i=1}^n) = \Teq(m, [\tau_i]_{i=1}^n)
        + \frac{\delta^0}{\varepsilon} \parens{L_- + L_{\pm} \sqrt{\frac{(1-p)(m-S)}{p(m-1)S}}} \times \nonumber \\
        &\times \brac{p \times \Teq(m, [\tau_i]_{i=1}^n) + (1-p)  \times \Teq(S + \min\{S,n\} \log \left(\min\{S,n\}\right), [\tau_i]_{i=1}^n)},
    \end{align}
    where $t^*$ is defined from Definition~\ref{def:time_budget}.
\end{theorem}

\begin{remark}
    Compared to Theorem~\ref{thm:page_time_compl}, which uses \emph{uniform sampling with replacement}, the guarantees for \emph{nice sampling} are slightly worse: the term $\Teq(S, [\tau_i]_{i=1}^n)$ from Theorem~\ref{thm:page_time_compl} here is replaced with $\Teq(S + \min\{S,n\} \log \left(\min\{S,n\}\right), [\tau_i]_{i=1}^n).$ Ignoring the logarithmic term $\log \left(\min\{S,n\}\right)$ ($\leq \log \left(\min\{m,n\}\right)$), the result from Theorem~\ref{thm:page_time_compl_nice} is equivalent to that in Theorem~\ref{thm:page_time_compl}. Thus, Theorems~\ref{thm:opt_S_main} and \ref{cor:nice_upper_bound} hold also for the nice sampling (up to logarithmic factors).
\end{remark}

\begin{proof}
    We use the same reasoning as in the proof of Theorem \ref{thm:page_time_compl}. With probability $p,$ the algorithm calculates the full gradients, which by Theorem~\ref{thm:full_grad_time} requires
    \begin{align*}
        &\Theta\left(\min_{j\in[n]} \parens{\parens{\sum_{i=1}^j \frac{1}{\tau_i}}^{-1} (m + \min\{m,n\} \log \left(\min\{m,n\}\right) + j)}\right) \\
        &=\Theta\left(\min_{j\in[n]} \parens{\parens{\sum_{i=1}^j \frac{1}{\tau_i}}^{-1} (m + j)}\right) \\
    \end{align*}
    seconds, where we use Assumption~\ref{ass:m_is_large}. With probability $1 - p,$ the algorithm calls \algname{ComputeBatchDifferenceAnySampling}, which by Theorem~\ref{thm:calculate_gradient_diff} requires
    \begin{align*}
        \Theta\left(\min_{j\in[n]} \parens{\parens{\sum_{i=1}^j \frac{1}{\tau_i}}^{-1} (S + \min\{S,n\} \log \left(\min\{S,n\}\right) + j)}\right)
    \end{align*}
    seconds. One can obtain the result by multiplying the iteration complexity \eqref{eq:page_iter_nice} by the expected time needed to collect the required number of stochastic gradients per iteration and adding the preprocessing time.
\end{proof}

\subsection{Importance sampling}
\label{sec:importance}

Here we additionally assume $L_i$-smoothness of the local objective functions $f_i$.

\begin{assumption}\label{as:worker_smooth}
    The functions $f_i$ are $L_i$-smooth. We denote $\bar{L} \eqdef \frac{1}{m} \sum_{j=1}^m L_j$ and $L_{\max} \eqdef \max_{i \in [n]} L_i.$
\end{assumption}

Importance sampling is a sampling technique that returns a multiset of indices \emph{with repetitions}. Index $j$ is included in the multiset with probability $\frac{L_j}{\sum_{i=1}^{n} L_j}.$

\begin{theorem}[\citet{tyurin2022sharper}, Section $3$]\label{thm:page_iter_import}
    Let Assumptions \ref{as:smooth_lower_bdd} and \ref{as:worker_smooth} hold. Choose a minibatch size $S\in[m]$, probability $p \in (0, 1]$ and the stepsize
    \begin{align*}
        \gamma = \frac{1}{L_- + \bar{L} \sqrt{\frac{1-p}{pS}}}.
    \end{align*}
    Then, the number of iteration needed by Algorithm~\ref{algorithm:rennala_page_gen} with \emph{importance sampling} to reach an $\varepsilon$-stationary point is
    \begin{align}\label{eq:iter_compl_import}
        K = \frac{2 \delta^0}{\varepsilon} \parens{L_- + \bar{L} \sqrt{\frac{1-p}{pS}}}.
    \end{align}
\end{theorem}

The complexity \eqref{eq:iter_compl_import} is nearly identical to \eqref{eq:iter_compl_indep} and \eqref{eq:page_iter_nice}, with the only difference being the dependence on $\bar{L}$ rather than $L_{\pm}$. Thus, all the results up to constant and logarithmic factors can be derived using the same methodology as that outlined in Section~\ref{sec:theory}, with the simple substitution of $L_{\pm}$ with $\bar{L}$.

\section{Dynamic Bounds}

As noted in Section \ref{sec:time_var_times}, the results from Section \ref{sec:proofs_rennala_page} can be easily generalized to iteration-dependent processing times.

\begin{theorem}\label{thm:page_time_compl_var}
    Consider the assumptions and the parameters from Theorem~\ref{thm:page_rate} and Assumption~\ref{ass:m_is_large}. Up to a constant factor, the time complexity of \algname{Freya PAGE} (Algorithm~\ref{algorithm:rennala_page_gen}) with iteration-dependent processing times $\{\tau_i^k\},$ which are defined in Section~\ref{sec:time_var_times}, is at most
    {\small\begin{align}\label{eq:compl_p_S_var}
    &\min_{j\in[n]} \parens{\parens{\sum_{i=1}^j \frac{1}{\tau_{\pi_{-1,i}}^{-1}}}^{-1} (m+j)}\nonumber \\
        &\qquad+ \sum_{k=0}^{\ceil{K_{\algname{PAGE}}}} \brac{p \min_{j\in[n]} \parens{\parens{\sum_{i=1}^j \frac{1}{\tau_{\pi_{k,i}}^k}}^{-1} (m+j)}
        + (1-p) \min_{j\in[n]} \parens{\parens{\sum_{i=1}^j \frac{1}{\tau_{\pi_{k,i}}^k}}^{-1} (S+j)}},
    \end{align}}
    to find an $\varepsilon$-stationary point, where $p \in (0, 1]$ and $S \in \N$ are free parameters, and $\pi_{k,\cdot}$ is a permutation such that $\tau_{\pi_{k,1}}^k \leq \dots \leq \tau_{\pi_{k,n}}^k$ for all $k \geq -1.$
\end{theorem}

\begin{remark}
    The theorem can be trivially extended to other samplings by changing $K_{\algname{PAGE}}$ to the iteration complexities from Theorems~\ref{thm:page_iter_nice} and \ref{thm:page_iter_import}.
\end{remark}

\begin{proof}
    The reasoning behind this result is exactly the same as in the proof of Theorem \ref{thm:page_time_compl}. The only difference is that in this more general setting, the expected time per iteration varies across iterations. Therefore, instead of simply multiplying, one needs to sum over the iterations to obtain the total time complexity. 

    We introduce the permutations to ensure that $\{\tau_{\pi_{k,i}}^k\}_{i=1}^n$ are sorted. When $\tau_{i}^k = \tau_{i},$ there is no need to introduce them because, throughout the paper, it is assumed $\tau_1\leq\ldots\leq\tau_n$ (see Section \ref{sec:intro}).
\end{proof}

\newpage

\section{Examples}\label{sec:examples_proofs}

Here we provide the proofs for the examples from Section \ref{sec:discussion}. We will use the notation
\begin{align*}
    t(S,j) \eqdef \parens{\sum_{i=1}^j \frac{1}{\tau_i}}^{-1} (S+j)
\end{align*}
for a fixed $S\in[m]$ and all $j\in[n]$.

\EXAMPLESAMEMAIN*
\begin{proof}
    First, when $\tau_j=\tau$ for all $j\in[n]$, then for any $S\in[m]$, $t(S,j)$ is minimized by taking $j = n$:
    \begin{align*}
        \Teq(S) &\eqdef \min_{j\in[n]} t(S,j) = \min_{j\in[n]} \parens{\parens{\sum_{i=1}^j \frac{1}{\tau_i}}^{-1} (S+j)} \\
        &= \min_{j\in[n]} \parens{\frac{\tau}{j} (S+j)}
        = \Theta\left(\tau \max\left\{\frac{S}{n}, 1\right\}\right).
    \end{align*}
    It remains to substitute this equality in \eqref{eq:time_comp}.

\end{proof}

\EXAMPLEINFFASTMAIN*
\begin{proof}
    The statement follows easily from the fact that for any $S\in[m]$ we have
    \begin{align*}
        \Teq(S) \eqdef \min_{j\in[n]} \parens{\parens{\sum_{i=1}^j \frac{1}{\tau_i}}^{-1} (S+j)}
        \leq \parens{\frac{1}{\tau_1}}^{-1} (S+j)
        = 0.
    \end{align*}
\end{proof}

\EXAMPLEINFSLOWMAIN*
\begin{proof}
    This follows from the fact that for any $S\in[m]$ and any $j\in[n]$ we have
    \begin{align*}
        t(S,j) \eqdef \parens{\sum_{i=1}^j \frac{1}{\tau_i}}^{-1} (S+j) = \infty.
    \end{align*}
\end{proof}

\EXAMPLESLOWMAIN*
\begin{proof}
    Suppose that $B \geq \frac{m+j_B}{\sum_{i=1}^{j_B} \frac{1}{\tau_i}}$ and fix any $k>j_B$. Then, since $\tau_j\geq B$ for all $j > j_{B}$, we have
    \begin{align*}
        \frac{k-j_B}{\sum_{i=j_B+1}^{k} \frac{1}{\tau_i}}
        \geq \frac{k-j_B}{\sum_{i=j_B+1}^{k} \frac{1}{B}}
        = B
        \geq \frac{m+j_B}{\sum_{i=1}^{j_B} \frac{1}{\tau_i}}
        \geq \frac{S+j_B}{\sum_{i=1}^{j_B} \frac{1}{\tau_i}}
        = t(S,j_B)
    \end{align*}
    for any $S\in[m]$. Rearranging and adding $\parens{S+j_B} \sum_{i=1}^{j_B} \frac{1}{\tau_i}$ to both sides of the inequality, we obtain
    \begin{align*}
        \parens{S+j_B} \parens{\sum_{i=1}^{j_B} \frac{1}{\tau_i}} + \parens{S+j_B} \parens{\sum_{i=j_B+1}^{k} \frac{1}{\tau_i}} \leq \parens{S + j_B} \parens{\sum_{i=1}^{j_B} \frac{1}{\tau_i}} + \parens{k - j_B} \parens{\sum_{i=1}^{j_B} \frac{1}{\tau_i}},
    \end{align*}
    meaning that
    \begin{align*}
        t(S,j_B) = \frac{S+j_B}{\sum_{i=1}^{j_B} \frac{1}{\tau_i}} \leq \frac{S+k}{\sum_{i=1}^{k} \frac{1}{\tau_i}} = t(S,k)
    \end{align*}
    for any $k>j_B$ and any $S\in[m]$. Therefore,
    \begin{align*}
        \min_{j\in[n]} \parens{\parens{\sum_{i=1}^j \frac{1}{\tau_i}}^{-1} (S+j)}
        = \min_{j\in[j_B]} \parens{\parens{\sum_{i=1}^j \frac{1}{\tau_i}}^{-1} (S+j)}
    \end{align*}
    for any $S\in[m]$, which proves the claim.
\end{proof}

\newpage

\section{A New Stochastic Gradient Method: \algnamelarge{Freya SGD}}
\label{sec:freya_sgd}

In this section, we introduce a new a non-variance reduced \algname{SGD} method that we call \algname{Freya SGD}. \algname{Freya SGD} is closely aligned with \algname{Rennala SGD} \citep{tyurin2023optimal}, but \algname{Freya SGD} does not require the $\sigma^2$--bounded variance assumption on stochastic gradients.

\begin{algorithm}[H]
    \caption{\algname{Freya SGD}}
    \label{algorithm:rennala_sgd}
    \begin{algorithmic}[1]
    \State \textbf{Parameters:} starting point $x^0 \in \R^d$, learning rate $\gamma > 0$, minibatch size $S$
    \For{$k = 0, 1, \ldots, K-1$}
        \State {\green $\frac{1}{S} \sum\limits_{i \in \cS^k} \nabla f_i(x^{k}) = \textnormal{\algname{ComputeBatch}}(S, x^k)$} \hfill (Alg.~\ref{algorithm:ga_asynch_batch})
        \State $x^{k+1} = x^k - \gamma \frac{1}{S} \sum\limits_{i \in \cS^k} \nabla f_i(x^{k})$
    \EndFor
    \end{algorithmic}
    {\small \hspace*{0.5cm} (note): $\cS^k$ is a set of size $\abs{\cS^k} = S$} of i.i.d. indices sampled from $[m]$ \emph{uniformly with replacement}
\end{algorithm}

\begin{assumption}
    \label{ass:local_lower_bound}
    For all $i \in [n],$ there exists $f^*_i$ such that $f_i(x) \geq f^*_i$ for all $x \in \R^d.$
\end{assumption}
We define
\begin{align*}
    \Delta^* \eqdef \frac{1}{n} \sum_{i=1}^{n} \left(f^* - f^*_i\right).
\end{align*}

\begin{theorem}
    \label{thm:sgd}
    Let Assumptions \ref{as:smooth_lower_bdd}, \ref{as:worker_smooth} and \ref{ass:local_lower_bound} hold. Choose minibatch size $S\in[m]$ and stepsize
    \begin{align*}
        \gamma = \min\left\{\frac{\sqrt{S}}{\sqrt{L L_{\max} K_{\textnormal{SGD}}}}, \frac{1}{L \left(1 - \frac{1}{S}\right)}, \frac{S \varepsilon}{4 L L_{\max}\Delta^*}\right\},
    \end{align*}
    where
    \begin{align*}
        K_{\textnormal{SGD}} \eqdef \frac{12 \delta^0 L}{\varepsilon} \max \left\{1 - \frac{1}{S}, \frac{12 L_{\max} \delta^0}{S \varepsilon}, \frac{4 L_{\max}\Delta^*}{S \varepsilon}\right\}.
    \end{align*}
    Then, the number of iterations needed by Algorithm~\ref{algorithm:rennala_sgd} to reach an $\varepsilon$-stationary point is $\ceil{K_{\textnormal{SGD}}}.$
\end{theorem}

\begin{proof}
    The iteration complexity can be proved using Corollary~1 and Proposition~3 (i) of \citet{khaled2020better} (with $q_i = 1 / n$).
\end{proof}

\begin{theorem}\label{thm:rennala_sgd_indep}
    Consider the assumptions and the parameters from Theorem~\ref{thm:sgd}. Up to a constant factor, the time complexity of \algname{Freya SGD} (Algorithm~\ref{algorithm:rennala_sgd}) is at most
    \begin{align*}
        T_{\textnormal{SGD}}(S, [\tau_i]_{i=1}^n)  &\eqdef \frac{\delta^0 L_-}{\varepsilon} \parens{1 - \frac{1}{S} + \frac{L_{\max}}{\varepsilon S} \parens{\delta^0 + \Delta^*}} \times \min_{j\in[n]} \parens{\parens{\sum_{i=1}^j \frac{1}{\tau_i}}^{-1} (S+j)}
    \end{align*}
    and is minimized by choosing
    \begin{align*}
        S^* = \frac{L_{\max}}{\varepsilon} \parens{\delta^0 + \Delta^*}.
    \end{align*}
    Up to a constant factor, we get
    \begin{align*}
        T_{\textnormal{SGD}}(S^*, [\tau_i]_{i=1}^n) = \frac{\delta^0 L_-}{\varepsilon} \min_{j\in[n]} \parens{\parens{\sum_{i=1}^j \frac{1}{\tau_i}}^{-1} \parens{\frac{L_{\max}}{\varepsilon} \parens{\delta^0 + \Delta^*} + j}}.
    \end{align*}
\end{theorem}

\begin{proof}
    At each iteration, the algorithm needs to collect a minibatch of stochastic gradients of size $S$. Multiplying the iteration complexity of Theorem~\ref{thm:sgd} by the time needed to gather such a minibatch (Theorem~\ref{thm:batch}), the resulting time complexity is
    \begin{align*}
        T_{\textnormal{SGD}}(S, [\tau_i]_{i=1}^n) &= \frac{\delta^0 L_-}{\varepsilon} \parens{1 - \frac{1}{S} + \frac{L_{\max}}{\varepsilon S} \parens{\delta^0 + \Delta^*}} \times \min_{j\in[n]} \parens{\parens{\sum_{i=1}^j \frac{1}{\tau_i}}^{-1} (S+j)}.
    \end{align*}
    
        We now find the optimal $S$. Assume first that $\frac{L_{\max}}{\varepsilon} \parens{\delta^0 + \Delta^*} \leq \frac{1}{2}$. Assumption~\ref{as:worker_smooth} ensures that the function $f$ is $L_{\max}$-smooth. Thus, we have $\norm{\nabla f(x^0)}^2 \leq 2L_{\max} \delta^0 \leq \varepsilon,$
        and $x^0$ is an $\varepsilon$-solution. Therefore, we can take any $S \geq 1.$

        Now, suppose that $\frac{L_{\max}}{\varepsilon} \parens{\delta^0 + \Delta^*} > \frac{1}{2}$. Then we have
        \begin{align*}
            T_{\textnormal{SGD}}(\nicefrac{L_{\max}}{\varepsilon} \parens{\delta^0 + \Delta^*}, [\tau_i]_{i=1}^n) &= \frac{\delta^0 L_-}{\varepsilon} \parens{2 - \frac{1}{S}} \min_{j\in[n]} \parens{\parens{\sum_{i=1}^j \frac{1}{\tau_i}}^{-1} \parens{\frac{L_{\max}}{\varepsilon} \parens{\delta^0 + \Delta^*} + j}} \\
            &\leq \frac{2\delta^0 L_-}{\varepsilon} \min_{j\in[n]} \parens{\parens{\sum_{i=1}^j \frac{1}{\tau_i}}^{-1} \parens{\frac{L_{\max}}{\varepsilon} \parens{\delta^0 + \Delta^*} + j}}.
        \end{align*}
        For all $S > \max\left\{\frac{L_{\max}}{\varepsilon} \parens{\delta^0 + \Delta^*}, 1\right\},$ we get $S \geq 2$ and hence
        \begin{align*}
            T_{\textnormal{SGD}}(S, [\tau_i]_{i=1}^n) &= \frac{\delta^0 L_-}{\varepsilon} \parens{1 - \frac{1}{S} + \frac{L_{\max}}{\varepsilon S} \parens{\delta^0 + \Delta^*}} \min_{j\in[n]} \parens{\parens{\sum_{i=1}^j \frac{1}{\tau_i}}^{-1} (S+j)} \\
            &\geq \frac{\delta^0 L_-}{2\varepsilon} \min_{j\in[n]} \parens{\parens{\sum_{i=1}^j \frac{1}{\tau_i}}^{-1} (S+j)} \\
            &\geq \frac{\delta^0 L_-}{2\varepsilon} \min_{j\in[n]} \parens{\parens{\sum_{i=1}^j \frac{1}{\tau_i}}^{-1} \parens{\frac{L_{\max}}{\varepsilon} \parens{\delta^0 + \Delta^*} + j}}.
        \end{align*}
        Let us now consider the case $1 < S \leq \max\left\{\frac{L_{\max}}{\varepsilon} \parens{\delta^0 + \Delta^*}, 1\right\}.$ We can additionally assume that $\frac{L_{\max}}{\varepsilon} \parens{\delta^0 + \Delta^*} > 1$ and $S \leq \frac{L_{\max}}{\varepsilon} \parens{\delta^0 + \Delta^*}$ (otherwise, the set $S$ that satisfies the condition $1 < S \leq \max\left\{\frac{L_{\max}}{\varepsilon} \parens{\delta^0 + \Delta^*}, 1\right\}$ is empty). We get
        \begin{align*}
            T_{\textnormal{SGD}}(S, [\tau_i]_{i=1}^n) &= \frac{\delta^0 L_-}{\varepsilon} \parens{1 - \frac{1}{S} + \frac{L_{\max}}{\varepsilon S} \parens{\delta^0 + \Delta^*}} \min_{j\in[n]} \parens{\parens{\sum_{i=1}^j \frac{1}{\tau_i}}^{-1} (S+j)} \\
            &\geq \frac{\delta^0 L_-}{\varepsilon} \parens{\frac{L_{\max}}{\varepsilon S} \parens{\delta^0 + \Delta^*}} \min_{j\in[n]} \parens{\parens{\sum_{i=1}^j \frac{1}{\tau_i}}^{-1} (S+j)} \\
            &= \frac{\delta^0 L_-}{\varepsilon} \min_{j\in[n]} \parens{\parens{\sum_{i=1}^j \frac{1}{\tau_i}}^{-1} \parens{\frac{L_{\max}}{\varepsilon} \parens{\delta^0 + \Delta^*} + j \frac{L_{\max}}{\varepsilon S} \parens{\delta^0 + \Delta^*}}} \\
            &\geq \frac{\delta^0 L_-}{\varepsilon} \min_{j\in[n]} \parens{\parens{\sum_{i=1}^j \frac{1}{\tau_i}}^{-1} \parens{\frac{L_{\max}}{\varepsilon} \parens{\delta^0 + \Delta^*} + j}}.
        \end{align*}
        Finally, for $S=1,$ we have
        \begin{align*}
            T_{\textnormal{SGD}}(S, [\tau_i]_{i=1}^n) &= \frac{\delta^0 L_- L_{\max}}{\varepsilon^2} \parens{\delta^0 + \Delta^*} \min_{j\in[n]} \parens{\parens{\sum_{i=1}^j \frac{1}{\tau_i}}^{-1} (1+j)} \\
            &= \frac{\delta^0 L_-}{\varepsilon} \min_{j\in[n]} \parens{\parens{\sum_{i=1}^j \frac{1}{\tau_i}}^{-1} \parens{\frac{L_{\max}}{\varepsilon} \parens{\delta^0 + \Delta^*} + \frac{L_{\max}}{\varepsilon} \parens{\delta^0 + \Delta^*} j}} \\
            &\geq \frac{\delta^0 L_-}{2 \varepsilon} \min_{j\in[n]} \parens{\parens{\sum_{i=1}^j \frac{1}{\tau_i}}^{-1} \parens{\frac{L_{\max}}{\varepsilon} \parens{\delta^0 + \Delta^*} + j}}
        \end{align*}
        because we assume $\frac{L_{\max}}{\varepsilon} \parens{\delta^0 + \Delta^*} > \frac{1}{2}.$ Therefore, an optimal choice is $S^* = \frac{L_{\max}}{\varepsilon} \parens{\delta^0 + \Delta^*}$.
\end{proof}

\newpage

\newpage
\section{Setup of the Experiments from Section~\ref{sec:exp_quad}}
\label{sec:exp:setup}
We consider the optimization problem \eqref{eq:problem} with nonconvex quadratic functions. The matrices and vectors defining the objective functions $f_i$ are generated using Algorithm~\ref{algorithm:matrix_generation} with $m = 10000,$ $d = 1000,$ $\lambda = 1\mathrm{e}{-6},$ and $s = 10$. The output is used to construct
\begin{align*}
    f_i(x) = \frac{1}{2} x^\top \mA_i x - b_i^\top x \quad \forall x \in \R^d, \, \forall i \in [m].
\end{align*}

\begin{algorithm}[H]
\caption{Quadratic optimization task generation}
\label{algorithm:matrix_generation}
\begin{algorithmic}[1]
\State \textbf{Parameters:} number of functions $m$, dimension $d$, regularizer $\lambda$, noise scale $s$
\For{$i = 1, \dots, m$}
\State Generate random noises $\nu_i^s = 1 + s \xi_i^s$ and $\nu_i^b = s \xi_i^b,$ i.i.d. $\xi_i^s, \xi_i^b \sim \cN(0, 1)$
\State Let $b_i = \frac{\nu_i^s}{4}(-1 + \nu_i^b, 0, \cdots, 0) \in \R^{d}$
\State Take the initial tridiagonal matrix
\[\mA_i = \frac{\nu_i^s}{4}\left( \begin{array}{cccc}
    2 & -1 & & 0\\
    -1 & \ddots & \ddots & \\
    & \ddots & \ddots & -1 \\
    0 & & -1 & 2 \end{array} \right) \in \R^{d \times d}\]
\EndFor
\State Take the mean of matrices $\mA = \frac{1}{m}\sum_{i=1}^m \mA_i$
\State Find the minimum eigenvalue $\lambda_{\min}(\mA)$
\For{$i = 1, \dots, m$}
\State Update matrix $\mA_i = \mA_i + (\lambda - \lambda_{\min}(\mA)) \mI$
\EndFor
\State Take starting point $x^0 = (\sqrt{d}, 0, \cdots, 0)$
\State \textbf{Output:} matrices $\mA_1, \cdots, \mA_m$, vectors $b_1, \cdots, b_m$, starting point $x^0$
\end{algorithmic}
\end{algorithm}

\newpage

\section{Lower bound}
\label{sec:lower_bound_appendix}

\subsection{Time multiple oracles protocol}
\label{sec:lower_bound_formal}

The classical lower bound frameworks \citep{nemirovskij1983problem, carmon2020lower, arjevani2022lower,nesterov2018lectures} are not convenient in the analysis of parallel algorithms since they are designed to estimate lower bounds on \emph{iteration complexities}. In order to obtain \emph{time complexity} lower bounds, we use the framework by \citet{tyurin2023optimal}. Let us briefly explain the main idea. A more detailed explanation can be found in \citep{tyurin2023optimal}[Sections 3-6].

We start by introducing an appropriate oracle for our setup:
\begin{align*}
  O_{\tau}\,:\, &\underbrace{\R_{\geq 0}}_{\textnormal{time}} \times \underbrace{\R^d}_{\textnormal{point}} \times \underbrace{(\R_{\geq 0} \times \R^d \times \{0, 1\})}_{\textnormal{input state}} \rightarrow \underbrace{(\R_{\geq 0} \times \R^d \times \{0, 1\})}_{\textnormal{output state}} \times \R^d
\end{align*}
\begin{align}
  \label{eq:oracle_inside_random}
  \begin{split}
      &\textnormal{such that } \qquad O_{\tau}(t, x, (s_t, s_x, s_q)) = \left\{
  \begin{aligned}
      &((t, x, 1), &0), \quad & s_q = 0, \\
      &((s_t, s_x, 1), &0), \quad & s_q = 1, t < s_t + \tau,\\
      &((0, 0, 0), & \nabla f_j(s_x)), \quad & s_q = 1, t \geq s_t + \tau,\\
  \end{aligned}
  \right.
  \end{split}
\end{align}
where $j \sim \textnormal{Uniform}([m]),$ i.e., $j$ is a random index sampled uniformly from the set $[m]$. We assume that all draws from $\textnormal{Uniform}([m])$ are i.i.d..

Next, we define the \emph{time multiple oracles protocol}, first introduced in \citep{tyurin2023optimal}.

\begin{protocol}[H]
    \caption{Time Multiple Oracles Protocol}
    \label{alg:time_multiple_oracle_protocol}
    \begin{algorithmic}[1]
    \State \textbf{Input: }function $f = \frac{1}{m} \sum_{i=1}^{m} f_i,$ oracles and distributions $(O_1, ..., O_n) \in \cO(f),$ algorithm $A~\in~\cA$
    \State $s^0_i = 0$ for all $i \in [n]$
    \For{$k = 0, \dots, \infty$}
    \State $({t^{k+1}}, {i^{k+1}}, x^k) = A^k(g^1, \dots, g^{k}),$ \hfill $\rhd \,{t^{k+1} \geq t^{k}}$
    \State $(s^{k+1}_{{i^{k+1}}}, g^{k+1}) = O_{{i^{k+1}}}({t^{k+1}}, x^k, s^{k}_{{i^{k+1}}})$ \hfill $\rhd\,s^{k+1}_j = s^{k}_j \quad \forall j \neq i^{k+1}$
    \EndFor
    \end{algorithmic}
\end{protocol}

Let us explain the behavior of the protocol. At each iteration, the algorithm $A$ returns three outputs, based on the available information/gradients: time $t^{k+1},$ the index of a worker $i^{k+1},$ and a new point $x^k.$ Depending on the current time $t^{k+1}$ and the state of the worker, three options are possible (see~\eqref{eq:oracle_inside_random}). If $s_q = 0,$ then the worker is idle. It then starts calculations at the point $x^k,$ changes the state $s_q$ from $0$ to $1,$ stores the point $x^k$ in $s_x$ (at which a new stochastic gradient should be calculated), and returns a zero vector. If $s_q = 0$ and $t^{k+1} < s_t + \tau,$ then the worker is still calculating a stochastic gradient. It does not change the state and returns a zero vector because the computation has not finished yet. If $s_q = 0$ and $t^{k+1} \geq s_t + \tau,$ the worker can finally return a stochastic gradient at $s_x$ because sufficient time has passed since the worker was idle ($s_q = 0$). Note that with this oracle, the algorithm will never receive the first stochastic gradient before time $\tau$ (assuming that all oracles have the same processing time $\tau;$ in general, we will assume that the processing times are different). 

In the setting considered in this work, there are $n$ oracles that can do calculations in parallel, and an algorithm orchestrates their work. Let the processing times of the oracles be equal to $\tau_1, \dots, \tau_n.$ A reasonable strategy would be to call each oracle with $t^k = 0,$ then to call the fastest worker with $t^k = \min_{i \in [n]} \tau_i$ to get the first stochastic gradients as soon as possible, then to call this worker again with $t^k = \min_{i \in [n]} \tau_i$ to request calculation of the next stochastic gradient, and so on.
One unusual thing about this protocol is that the algorithm controls the time. The oracle is designed to force the algorithm to increase the time; otherwise, the algorithm would not receive new information about the function.

Our goal will be to bound the complexity measure $\mathfrak{m}_{\textnormal{time}}\left(\cA, \cF\right)$, defined as
\begin{align}\label{eq:lower_compl_time}
    \begin{split}
    &\mathfrak{m}_{\textnormal{time}}\left(\cA, \cF\right) \eqdef \inf_{A \in \cA} \sup_{f \in \cF} \sup_{\{O_i\} \in \cO(f)} \inf\left\{t \geq 0\,\middle|\,\Exp{ \inf_{k \in S_t}\norm{\nabla f(x^k)}^2} \leq \varepsilon \right\}, 
    \\
    &S_t \eqdef \left\{k \in \N_0 \middle| t^{k} \leq t\right\},
\end{split}
\end{align}
where the sequences $t^k$ and $x^k$ are generated by Protocol~\ref{alg:time_multiple_oracle_protocol}. Hence, unlike the classical approach, where the lower bounds are obtained for the minimum number of iterations required to find an $\varepsilon$--stationary point, we seek to find the minimum \emph{time} needed to get an $\varepsilon$--stationary point.

We consider a standard for our setup class of functions \citep{fang2018spider}:

\begin{definition}[Function Class $\cF^m_{\delta^0, L_+}$]
  We say that $f \in \cF^m_{\delta^0, L_+}$ if it is $\delta^0$-bounded, i.e., $f(0) - \inf_{x \in \R^d} f(x) \leq \delta^0$, and
  $$f(x) = \frac{1}{m} \sum_{i=1}^{m} f_i(x),$$
  where the functions $f_i \,:\,\R^d \rightarrow \R$ are differentiable and satisfy
  $$\frac{1}{m} \sum_{i=1}^{m} \norm{\nabla f_i(x) - \nabla f_i(y)}^2 \leq L_+^2 \norm{x - y}^2 \quad \forall x, y \in \R^d.$$
  \label{def:func_class}
\end{definition}

Next, we define the class of algorithms we will analyze.

\begin{definition}[Algorithm Class $\cA_{\textnormal{zr}}$]
  \label{def:alg_class}
  Let us consider Protocol~\ref{alg:time_multiple_oracle_protocol}. We say that a sequence of mappings $A = \{A^k\}_{k=0}^{\infty}$ is a \emph{zero-respecting algorithm}, if
  \begin{enumerate}
    \item $A^k\,:\, \underbrace{\R^d \times \dots \times \R^d}_{k \textnormal{ times}} \rightarrow \R_{\geq 0} \times \N \times \R^d$ for all $k \geq 1$ and $A^0 \in \R_{\geq 0} \times \N \times \R^d.$ \label{prop:one}
    \item For all $k \geq 1$ and $g^1, \dots, g^k \in \R^d,$ $t^{k+1} \geq t^k,$
    where $t^{k+1}$ and $t^k$ are defined as $(t^{k+1}, \dots) = A^k(g^1, \dots, g^k)$ and $(t^k, \dots) = A^{k-1}(g^1, \dots, g^{k-1}).$ \label{prop:two}
    \item $\textnormal{supp} \left(x^k\right) \subseteq \bigcup_{j = 1}^k \textnormal{supp} \left(g^j\right)$ 
    for all $k \in \N_0,$ where $\textnormal{supp}(x) \eqdef \{i \in [d]\,|\,x_i \neq 0\}.$ \label{prop:four}
  \end{enumerate}
  We denote the set of all algorithms with these properties as $\cA_{\textnormal{zr}}.$
\end{definition}

In the above definition, property \ref{prop:one} defines the domain of the mappings $A^k$, and property \ref{prop:two} ensures that our algorithm does not ``cheat'' and does not ``travel into the past'': the time can only go forward (see \citep{tyurin2023optimal}[Section 4, Definition 4.1]). Property \ref{prop:four} is a standard assumption for zero-respecting algorithms \citep{arjevani2022lower} that is satisfied by virtually all algorithms, including \algname{Adam} \citep{kingma2014adam}, \algname{SGD}, \algname{PAGE} \citep{li2021page} and \algname{Asynchronous SGD}.

It remains to define an oracle class for our problem that employs oracles from \eqref{eq:oracle_inside_random}. We design oracles that emulate the real behavior of the workers.
\begin{definition}[Computation Oracle Class $\cO_{\tau_1, \dots, \tau_n}$]
  For any $f \in \cF^m_{\delta^0, L_+},$ the oracle class $\cO_{\tau_1, \dots, \tau_n}$ returns oracles $O_i = O_{\tau_i}$, $i \in [n]$, where the mappings $O_{\tau_i}$ are defined in \eqref{eq:oracle_inside_random}.
  \label{def:oracle_class}
\end{definition}

\subsection{The ``worst case'' function in the nonconvex world}
\label{sec:worst_case}

The analysis uses a standard function, commonly employed to derive lower bounds in the nonconvex regime. First, let us define
\begin{align*}
    \textnormal{prog}(x) \eqdef \max \{i \geq 0\,|\,x_i \neq 0\} \quad (x_0 \equiv 1).
\end{align*}
Our choice of the underlying function $F$ follows the construction introduced in \citet{carmon2020lower,arjevani2022lower}: for any $T \in \N,$ define
\begin{align}
    \label{eq:worst_case}
    F_T(x) \eqdef -\Psi(1) \Phi(x_1) + \sum_{i=2}^T \left[\Psi(-x_{i-1})\Phi(-x_i) - \Psi(x_{i-1})\Phi(x_i)\right],
\end{align}
where
\begin{align*}
    \Psi(x) = \begin{cases}
        0, & x \leq 1/2, \\
        \exp\left(1 - \frac{1}{(2x - 1)^2}\right), & x \geq 1/2,
    \end{cases}
    \quad\textnormal{and}\quad
    \Phi(x) = \sqrt{e} \int_{-\infty}^{x}e^{-\frac{1}{2}t^2}dt.
\end{align*}
Throughout the proof, we only rely on the following properties of the function:
\begin{lemma}[\cite{carmon2020lower, arjevani2022lower}]
    \label{lemma:worst_function}
    The function $F_T$ satisfies:
    \begin{enumerate}
        \item $F_T(0) - \inf_{x \in \R^T} F_T(x) \leq \Delta^0 T,$ where $\Delta^0 = 12.$
        \item The function $F_T$ is $l_1$--smooth, where $l_1 = 152.$
        \item For all $x \in \R^T,$ $\norm{\nabla F_T(x)}_{\infty} \leq \gamma_{\infty},$ where $\gamma_{\infty} = 23.$
        \item For all $x \in \R^T,$ $\textnormal{prog}(\nabla F_T(x)) \leq \textnormal{prog}(x) + 1.$
        \item For all $x \in \R^T,$ if $\textnormal{prog}(x) < T,$ then $\norm{\nabla F_T(x)} > 1.$
    \end{enumerate}
\end{lemma}

\subsection{The first lower bound}\label{sec:lower_bound1}

We are ready to present the main results of this section.

\begin{restatable}{theorem}{THEOREMLOWERBOUND}
  \label{theorem:lower_bound}
  Let us consider Protocol~\ref{alg:time_multiple_oracle_protocol}. Without loss of generality, assume that $0 < \tau_1 \leq \dots \leq \tau_n$ and take any $L_{+}, \delta^0, \varepsilon > 0,$ and $m \in \N$ such that $\varepsilon < c_1 L_{+} \delta^0$ and $\frac{\delta^0 L_+}{\varepsilon} > c_2 \sqrt{m}.$ Then, for any algorithm $A \in \cA_{\textnormal{zr}},$ there exists a function $f \in \cF^m_{\delta^0, L_+}$ and computation oracles $(O_1, \dots, O_n) \in \cO_{\tau_1, \dots, \tau_n}(f)$ such that $\Exp{\inf_{k \in S_t} \norm{\nabla f(x^k)}^2} > \varepsilon,$ 
  where $S_t \eqdef \left\{k \in \N_0 \,|\,t^k \leq t\right\}$ and 
$$t = c_3 \times \frac{\delta^0 L_+}{\sqrt{m} \varepsilon} \min_{j \in [n]} \left[\left(\sum_{i=1}^j \frac{1}{\tau_i}\right)^{-1} \left(m + j\right)\right].$$
The quantities $c_1,$$c_2,$ and $c_3$ are universal constants. The sequences $x^k$ and $t^k$ are defined in Protocol~\ref{alg:time_multiple_oracle_protocol}.
\end{restatable}

\begin{proof}
    \textbf{(Step 1: Construction of a hard problem)}\\
    We start our proof by constructing an appropriate function $f \in \cF^m_{\delta^0, L_+}.$ Let us fix any $T \geq \N$ and define $f_i \,:\,\R^T \rightarrow \R$ such that
    \begin{align*}
        f_1 \eqdef \frac{\sqrt{m} L_+ \lambda^2}{l_1} F_T\left(\frac{x}{\lambda}\right) 
    \end{align*}
    for all $x \in \R^T,$ and $f_i(x) = 0$ for all $i \in \{2, \dots, m\}$ and $x \in \R^T.$ Essentially, all information about the function $f = \frac{1}{m} \sum_{i=1}^{m} f_i$ is in the first function. Note that
    \begin{align*}
        \frac{1}{m} \sum_{i=1}^{m} \norm{\nabla f_i(x) - \nabla f_i(y)}^2 
        &= \frac{1}{m} \norm{\nabla f_1(x) - \nabla f_1(y)}^2 \\
        &= \frac{1}{m} \norm{\frac{\sqrt{m} L_+ \lambda}{l_1} \nabla F_T\left(\frac{x}{\lambda}\right) - \frac{\sqrt{m} L_+ \lambda}{l_1} \nabla F_T\left(\frac{y}{\lambda}\right)}^2 \\
        &= \frac{L_+^2 \lambda^2}{l_1^2}  \norm{\nabla F_T\left(\frac{x}{\lambda}\right) - \nabla F_T\left(\frac{y}{\lambda}\right)}^2.
    \end{align*}
    Then, using Lemma~\ref{lemma:worst_function}, we have
    \begin{align}
        \label{eq:oZptwxozdpoFk}
        \frac{1}{m} \sum_{i=1}^{m} \norm{\nabla f_i(x) - \nabla f_i(y)}^2 
        \leq L_+^2 \lambda^2  \norm{\frac{x}{\lambda} - \frac{y}{\lambda}}^2 = L_+^2 \norm{x - y}^2.
    \end{align}
    Taking
    \begin{align*}
        T = \flr{\frac{\sqrt{m} \delta^0 l_1}{L_+ \lambda^2 \Delta^0}},
    \end{align*}
    we ensure that
    \begin{align}
        \label{eq:eqhNFkSc}
        f(0) - \inf_{x \in \R^T} f(x) &= \frac{1}{m} \left(\frac{\sqrt{m} L_+ \lambda^2}{l_1} F_T\left(0\right) - \inf_{x \in \R^T} \frac{\sqrt{m} L_+ \lambda^2}{l_1} F_T\left(x\right)\right) \nonumber\\
        &= \frac{L_+ \lambda^2}{\sqrt{m} l_1} \left(F_T\left(0\right) - \inf_{x \in \R^T} F_T\left(x\right)\right) \leq \frac{L_+ \lambda^2 \Delta^0 T}{\sqrt{m} l_1} \leq \delta^0,
    \end{align}
    where in the inequalities we use Lemma~\ref{lemma:worst_function} and the choice of $T.$ Now, inequalities \eqref{eq:oZptwxozdpoFk} and \eqref{eq:eqhNFkSc} imply that $f = \frac{1}{m} \sum_{i=1}^{m} f_i \in \cF^m_{\delta^0, L_+}$, and hence, using Lemma~\ref{lemma:worst_function} again, we get
    \begin{align*}
        \inf_{k \in S^t} \norm{\nabla f(x^k)}^2 &= \inf_{k \in S^t} \norm{\frac{1}{m} \times \frac{\sqrt{m} L_+ \lambda}{l_1} \nabla F_T\left(\frac{x^k}{\lambda}\right)}^2 \\
        &= \frac{L_+^2 \lambda^2}{m l_1^2} \inf_{k \in S^t} \norm{\nabla F_T\left(\frac{x}{\lambda}\right)}^2 > \frac{L_+^2 \lambda^2}{m l_1^2} \inf_{k \in S^t} \mathbbm{1}\left[\textnormal{prog}(x^k) < T\right].
    \end{align*}
    Let us take 
    \begin{align*}
        \lambda = \frac{2 l_1 \sqrt{m}\sqrt{\varepsilon}}{L_+}.
    \end{align*}
    Then 
    \begin{align}
        \label{eq:YfxCfV}
        \inf_{k \in S^t} \norm{\nabla f(x^k)}^2 > 4 \varepsilon \inf_{k \in S^t} \mathbbm{1}\left[\textnormal{prog}(x^k) < T\right]
    \end{align}
    and
    \begin{align*}
        T = \flr{\frac{\delta^0 L_+}{4 l_1 \Delta^0 \sqrt{m} \varepsilon}}.
    \end{align*}
    The last inequality means that while $\textnormal{prog}(x^k) < T$ for all $k \in S^t,$ all gradients are large. The function $F_T$ is a \emph{zero-chain}: due to Lemma~\ref{lemma:worst_function}, we know that $\textnormal{prog}(\nabla F_T(x)) \leq \textnormal{prog}(x) + 1$ for all $x \in \R^T.$ This implies that we can discover at most one new non-zero coordinate by calculating the gradient of the function $\nabla f_1.$ Since the algorithm $A \in \cA_{\textnormal{zr}}$ is zero-respecting, by definition it cannot return a point $x^k$ with progress greater than that of the vectors returned by the oracles. In view of this, it is necessary to calculate the gradient of $f_1$ at least $T$ times to get $\textnormal{prog}(x^k) \geq T.$

    The gradient of $f_1$ can be calculated if and only if $j = 1,$ where $j \sim \textnormal{Uniform}([m])$ (see \eqref{eq:oracle_inside_random}).
    Consider worker $i$ and define $\eta_i^1$ to be the number of draws from $\textnormal{Uniform}([m])$ until the index $j = 1$ is sampled. Clearly, $\eta_i^1$ is a Geometric random variable with parameter $\Prob{j = 1} = \frac{1}{m}.$ Recall that the workers can do the computations in parallel, and by the design of the oracles, worker $i$ needs at least $\tau_i \eta_i^1$ seconds to calculate $\eta_i^1$ stochastic gradients.
    Hence, it is impossible to calculate $\nabla f_1$ before the time
    \begin{align*}
        \min_{i \in [n]} \tau_i \eta_i^1.
    \end{align*}    
    Once the algorithm calculates $\nabla f_1$ for the first time, it needs to do so at least $T - 1$ times more to achieve $\textnormal{prog}(x^k) \geq T.$ Thus, one should wait at least 
    \begin{align*}
        \sum_{k=1}^{T} \min_{i \in [n]} \tau_i \eta_i^k
    \end{align*}
    seconds, where $\eta_i^k \overset{\textnormal{i.i.d.}}{\sim} \textnormal{Geometric}(\nicefrac{1}{m})$. We can conclude that
    \begin{align}
        \label{eq:aGAVfIUGkbWG}
        \Prob{\inf_{k \in S^t} \mathbbm{1}\left[\textnormal{prog}(x^k) < T\right] = 0} \leq \Prob{\sum_{k=1}^{T} \min_{i \in [n]} \tau_i \eta_i^k \leq t}.
    \end{align} \\
    \textbf{(Step 2: The Chernoff Method)}\\
    The theorem's proof is now reduced to the analysis of the concentration of $\sum_{k=1}^{T} \min_{i \in [n]} \tau_i \eta_i^k.$
    Using the Chernoff method, for all $s >0$, we have 
    \begin{align*}
        \Prob{\sum_{k=1}^{T} \min_{i \in [n]} \tau_i \eta_i^k \leq t} &= \Prob{- s \sum_{k=1}^{T} \min_{i \in [n]} \tau_i \eta_i^k \geq - s t} \\
        &= \Prob{e^{- s \sum_{k=1}^{T} \min_{i \in [n]} \tau_i \eta_i^k} \geq e^{- s t}} \\
        &\leq e^{s t} \Exp{\exp\left(- s \sum_{k=1}^{T} \min_{i \in [n]} \tau_i \eta_i^k\right)}.
    \end{align*}
    Independence gives
    \begin{align}
        \label{eq:IuiyqWW}
        \Prob{\sum_{k=1}^{T} \min_{i \in [n]} \tau_i \eta_i^k \leq t}
        &\leq e^{s t} \prod_{k=1}^{T} \Exp{\exp\left(- s \min_{i \in [n]} \tau_i \eta_i^k\right)} \nonumber\\
        &\overset{\textnormal{i.i.d.}}{=}  e^{s t} \left(\Exp{\exp\left(- s \min_{i \in [n]} \tau_i \eta_i^1\right)}\right)^T.
    \end{align}
    Let us consider the term in the last bracket separately. For a fixed $t' > 0$, we have
    \begin{eqnarray}
        &&\hspace{-1cm}\Exp{\exp\left(- s \min_{i \in [n]} \tau_i \eta_i^1\right)} \\
        &=& \Exp{\max_{i \in [n]}\exp\left(- s \tau_i \eta_i^1\right)} \nonumber\\
        &=& \Exp{\max_{i \in [n]} \left(\mathbbm{1}\left[\tau_i \eta_i^1 \leq t'\right] \exp\left(- s \tau_i \eta_i^1\right) + \left(1 - \mathbbm{1}\left[\tau_i \eta_i^1 \leq t'\right]\right) \exp\left(- s \tau_i \eta_i^1\right)\right)} \nonumber \\
        &\leq& \Exp{\max_{i \in [n]} \left(\mathbbm{1}\left[\tau_i \eta_i^1 \leq t'\right] + \left(1 - \mathbbm{1}\left[\tau_i \eta_i^1 \leq t'\right]\right) \exp\left(- s t'\right)\right)} \nonumber \\
        &=& \exp\left(- s t'\right) + (1 - \exp\left(- s t'\right)) \Exp{\max_{i \in [n]} \left(\mathbbm{1}\left[\tau_i \eta_i^1 \leq t'\right]\right)}.
    \label{eq:xOIbDHbliybIEmlrat}
    \end{eqnarray}

    We now consider the last term. Due to the independence, we have
    \begin{align*}
        \Exp{\max_{i \in [n]} \left(\mathbbm{1}\left[\tau_i \eta_i^1 \leq t'\right]\right)} 
        = 1 - \prod_{i=1}^n \Prob{\tau_i \eta_i^1 > t'}
        = 1 - \prod_{i=1}^n \left(1 - p\right)^{\flr{\frac{t'}{\tau_i}}},
    \end{align*}
    where we use the cumulative distribution function of a geometric random variable and temporarily define $p \eqdef \frac{1}{m}.$ Using Lemma~\ref{lemma:sum_prod}, we get
    \begin{align}
        \Exp{\max_{i \in [n]} \left(\mathbbm{1}\left[\tau_i \eta_i^1 \leq t'\right]\right)} \leq p \sum_{i=1}^{n} \flr{\frac{t'}{\tau_i}}.
        \label{eq:bNNMYRoZFUY}
    \end{align}
    Let us take 
    \begin{align*}
        t' = \frac{1}{8} \times \min_{j \in [n]} \left(\sum_{i=1}^j \frac{1}{\tau_i}\right)^{-1} \left(\frac{1}{p} + j\right) = \frac{1}{8} \times \min_{j \in [n]} g(j),
    \end{align*}
    where $g(j) \eqdef \left(\sum_{i=1}^j \frac{1}{\tau_i}\right)^{-1} \left(\frac{1}{p} + j\right)$ for all $j \in [n]$ and assume that $j^*$ is the largest index such that $\min_{j \in [n]} g(j) = g(j^*).$ Then, Lemma \ref{lemma:techn1} gives
    \begin{align}
        \label{eq:STaSs}
        \tau_{j^*} \leq \min_{j \in [n]} g(j) < \tau_{j^* + 1},
    \end{align}
    where we let $\tau_{n + 1} \equiv \infty$. Therefore, $t' < \tau_{j^* + 1}$ and \eqref{eq:bNNMYRoZFUY} gives
    \begin{align*}
        \Exp{\max_{i \in [n]} \left(\mathbbm{1}\left[\tau_i \eta_i^1 \leq t'\right]\right)} 
        &\leq p \sum_{i=1}^{n} \flr{\frac{t'}{\tau_i}} = p \sum_{i=1}^{j^*} \flr{\frac{t'}{\tau_i}}.
    \end{align*}
    Using \eqref{eq:STaSs}, we get $\frac{t'}{\tau_i} = \frac{\min_{j \in [n]} g(j)}{8 \tau_i} \geq \frac{1}{8}$ for all $i \leq j^*.$
    Since $\flr{x} \leq 2 x - \frac{1}{4}$ for all $x \geq \frac{1}{8},$ we obtain
    \begin{align*}
        \Exp{\max_{i \in [n]} \left(\mathbbm{1}\left[\tau_i \eta_i^1 \leq t'\right]\right)} 
        &\leq p \sum_{i=1}^{j^*} \left(\frac{2 t'}{\tau_i} - \frac{1}{4}\right) = 2 p t' \left(\sum_{i=1}^{j^*} \frac{1}{\tau_i}\right) - \frac{p j^*}{4} \\
        &=2 p \times \frac{1}{8} \left(\sum_{i=1}^{j^*} \frac{1}{\tau_i}\right)^{-1} \left(\frac{1}{p} + j^*\right) \times \left(\sum_{i=1}^{j^*} \frac{1}{\tau_i}\right) - \frac{p j^*}{4} \\
        &=\frac{p}{4} \left(\frac{1}{p} + j^*\right) - \frac{p j^*}{4} = \frac{1}{4}.
    \end{align*}
    Substituting the last inequality to \eqref{eq:xOIbDHbliybIEmlrat} gives
    \begin{align*}
        \Exp{\exp\left(- s \min_{i \in [n]} \tau_i \eta_i^1\right)} 
        &\leq \exp\left(- s t'\right) + \frac{1}{4} (1 - \exp\left(- s t'\right)).
    \end{align*}
    We now take $s = 1 / t'$ to obtain
    \begin{align*}
        \Exp{\exp\left(- s \min_{i \in [n]} \tau_i \eta_i^1\right)} \leq e^{-1} + \frac{1}{4} (1 - e^{-1}) \leq e^{-\frac{1}{2}}.
    \end{align*}
    Substituting this inequality in \eqref{eq:aGAVfIUGkbWG} and \eqref{eq:IuiyqWW} gives
    \begin{align*}
        \Prob{\inf_{k \in S^t} \mathbbm{1}\left[\textnormal{prog}(x^k) < T\right] = 0} \leq \Prob{\sum_{k=1}^{T} \min_{i \in [n]} \tau_i \eta_i^k \leq t} \leq e^{\frac{t}{t'} - \frac{T}{2}}.
    \end{align*}
    Therefore,
    \begin{align*}
        \Prob{\inf_{k \in S^t} \mathbbm{1}\left[\textnormal{prog}(x^k) < T\right] = 0} \leq \rho
    \end{align*}
    for all 
    \begin{align*}
        t \leq \frac{1}{8} \min_{j \in [n]} \left[\left(\sum_{i=1}^j \frac{1}{\tau_i}\right)^{-1} \left(\frac{1}{p} + j\right)\right] \left(\frac{T}{2} + \log \rho\right)
    \end{align*}
    and $\rho > 0.$ Using the bound on the probability with $\rho = \frac{1}{2}$ and \eqref{eq:YfxCfV}, we finally conclude 
    \begin{align*}
        \Exp{\inf_{k \in S^t} \norm{\nabla f(x^k)}^2} > 4 \varepsilon \Prob{\inf_{k \in S^t} \mathbbm{1}\left[\textnormal{prog}(x^k) < T\right] = 1} \geq 2 \varepsilon
    \end{align*}
    for all
    \begin{align*}
        t 
        &\leq \frac{1}{8} \min_{j \in [n]} \left[\left(\sum_{i=1}^j \frac{1}{\tau_i}\right)^{-1} \left(\frac{1}{p} + j\right)\right] \left(\frac{T}{2} + \log \frac{1}{2}\right) \\
        &= \frac{1}{8} \min_{j \in [n]} \left[\left(\sum_{i=1}^j \frac{1}{\tau_i}\right)^{-1} \left(m + j\right)\right] \left(\frac{1}{2} \flr{\frac{\delta^0 L_+}{4 l_1 \Delta^0 \sqrt{m} \varepsilon}} + \log \frac{1}{2}\right).
    \end{align*}
    It remains to use the conditions $\varepsilon < c_1 L \delta^0$ and $\frac{\delta^0 L_+}{\varepsilon} > c_2 \sqrt{m}$ from the theorem to finish the proof.
\end{proof}

\subsection{The second lower bound}\label{sec:lower_bound2}

\begin{restatable}{theorem}{THEOREMLOWERBOUNDSECOND}
    \label{theorem:second_lower_bound}
    Let us consider Protocol~\ref{alg:time_multiple_oracle_protocol}. Without loss of generality, assume that $0 < \tau_1 \leq \dots \leq \tau_n$ and take any $L_{+}, \delta^0, \varepsilon > 0,$ and $m \in \N$ such that $\varepsilon < c_1 L_{+} \delta^0.$ Then, for any algorithm $A \in \cA_{\textnormal{zr}},$ there exists a function $f \in \cF^m_{\delta^0, L_+}$ and computation oracles $(O_1, \dots, O_n) \in \cO_{\tau_1, \dots, \tau_n}(f)$ such that $\Exp{\inf_{k \in S_t} \norm{\nabla f(x^k)}^2} > \varepsilon,$ 
    where $S_t \eqdef \left\{k \in \N_0 \,|\,t^k \leq t\right\}$ and 
  $$t = c_2 \times \min_{j \in [n]} \left(\sum_{i=1}^j \frac{1}{\tau_i}\right)^{-1} \left(m + j\right).$$
  The quantities $c_1$ and $c_2$ are universal constants. The sequences $x^k$ and $t^k$ are defined in Protocol~\ref{alg:time_multiple_oracle_protocol}.
\end{restatable}

\begin{proof}
    We use the same construction as in the proof of Theorem $2$ of \citet{li2021page}. Let us consider 
    \begin{align}\label{eq:worst_case_f2}
        f_i(x) \eqdef c \inp{v_i}{x} + \frac{L_{+}}{2} \norm{x}^2
    \end{align}
    for all $x \in \R^{T},$ where $c \in \R$ and $v_i \in \R^{T}$, $i \in [n]$ are parameters that we define later.
    Then 
    \begin{align*}
        \frac{1}{m} \sum_{i=1}^{m} \norm{\nabla f_i(x) - \nabla f_i(y)}^2 \leq L_{+}^2 \norm{x - y}^2
    \end{align*}
    for all $x, y \in \R^d$ and 
    \begin{align*}
        f(0) - \inf_{x \in \R^T} f(x) = - \inf_{x \in \R^T} \left[c \inp{\frac{1}{m} \sum_{i=1}^{m} v_i}{x} + \frac{L_{+}}{2} \norm{x}^2\right] 
        = \frac{c^2}{2 L_{+} m^2} \norm{\sum_{i=1}^{m} v_i}^2 = \delta^0,
    \end{align*}
    where we take 
    \begin{align*}
        c \eqdef \sqrt{\frac{2 L_{+} m^2\delta^0}{\norm{\sum_{i=1}^{m} v_i}^2}}.
    \end{align*}
    Thus, we have $f = \frac{1}{m} \sum_{i=1}^{m} f_i \in \cF^m_{\delta^0, L_+}.$ Now, let 
    \begin{align*}
        &v_1 = (\underbrace{1, \dots, 1}_{\frac{T}{m}}, 0, \dots, 0)^\top \in \R^T, \\
        &v_2 = (\underbrace{0, \dots, 0}_{\frac{T}{m}}, \underbrace{1, \dots, 1}_{\frac{T}{m}}, 0, \dots, 0)^\top \in \R^T, \\
        &v_3 = (\underbrace{0, \dots, 0}_{\frac{T}{m}}, \underbrace{0, \dots, 0}_{\frac{T}{m}}, \underbrace{1, \dots, 1}_{\frac{T}{m}}, 0, \dots, 0)^\top \in \R^T \\
        &\dots  \\
        &v_n = (0, \dots, 0, \underbrace{1, \dots, 1}_{\frac{T}{m}})^\top \in \R^T.
    \end{align*}
    and choose any $T \in \{s m \,|\, s \in \N\}$ (one can always take $T = m$).
    Then 
    \begin{align*}
        c = \sqrt{\frac{2 L_{+} m^2\delta^0}{T}}.
    \end{align*}
    Let us fix some time $t > 0$ to be determined later and recall that the workers can calculate the stochastic gradients $\nabla f_i(x) = c v_i + L_{+} x$ in parallel. Suppose that up to time $t$, fewer than $\frac{m}{2}$ stochastic gradients have been computed. Then, since the algorithm is a zero-respecting algorithm, it cannot have discovered more than $\frac{m}{2} \times \frac{T}{m} = \frac{T}{2}$ coordinates. Thus, at least $\frac{T}{2}$ coordinates are equal to $0$ and 
    \begin{align*}
        \norm{\nabla f(x^k)}^2 = \norm{\frac{c}{m} \sum_{i=1}^m v_i + L_{+} x^k}^2 \geq \frac{c^2}{m^2} \times \frac{T}{2}.
    \end{align*}
    for all $k \in \left\{k \in \N_0 \,|\,t^k \leq t\right\}$. Therefore, substituting our choice of $c$ and using the assumptions from the theorem, we get
    \begin{align}
        \label{eq:AzUpLSPcSM}
        \norm{\nabla f(x^k)}^2 \geq L_{+}\delta^0 > 2 \varepsilon.
    \end{align}

    It remains to find the time $t > 0.$ The workers work in parallel, so in $t$ seconds they calculate at most 
    \begin{align*}
        \sum_{i=1}^{n} \flr{\frac{t}{\tau_i}}
    \end{align*}
    stochastic gradients. Now, let
    \begin{align}
        \label{eq:KzAfFVwxs}
        t = \frac{1}{8} \min_{j \in [n]} \left(\sum_{i=1}^j \frac{1}{\tau_i}\right)^{-1} \left(m + j\right),
    \end{align}
    and define $g(j) \eqdef \left(\sum_{i=1}^j \frac{1}{\tau_i}\right)^{-1} \left(m + j\right)$ for all $j \in [n].$ Assume that $j^*$ is the largest index such that $\min_{j \in [n]} g(j) = g(j^*).$ Using Lemma~\ref{lemma:techn1}, we have
    \begin{align*}
        \sum_{i=1}^{n} \flr{\frac{t}{\tau_i}} = \sum_{i=1}^{j^*} \flr{\frac{t}{\tau_i}}.
    \end{align*}
    Since $\frac{t}{\tau_i} \geq \frac{1}{8}$ for all $i \leq j^*$ and $\flr{x} \leq 2 x - \frac{1}{4}$ for all $x \geq \frac{1}{8},$ we obtain
    \begin{align*}
        \sum_{i=1}^{n} \flr{\frac{t}{\tau_i}} \leq \sum_{i=1}^{j^*} \frac{2 t}{\tau_i} - \frac{j^*}{4} = \frac{1}{4} \left(\sum_{i=1}^{j^*} \frac{1}{\tau_i}\right)^{-1} \left(m + j^*\right) \sum_{i=1}^{j^*} \frac{1}{\tau_i} - \frac{j^*}{4} = \frac{m}{4}.
    \end{align*}
    Therefore, it is possible to calculate at most $\frac{m}{4}$ stochastic gradient in time \eqref{eq:KzAfFVwxs} and we can finally conclude that inequality \eqref{eq:AzUpLSPcSM} holds for any time that is less than or equal \eqref{eq:KzAfFVwxs}.
\end{proof}

\newpage

\section{Useful Identities and Inequalities}

\begin{lemma}[\citep{szlendak2021permutation}] \label{lemma:smoothness_consts}
    It holds that $L_- \leq L_+$, $L_- \leq \frac{1}{m} \sum_{i=1}^m L_i$ and $L_+^2 - L_-^2 \leq L_{\pm}^2 \leq L_+^2 \leq \frac{1}{m} \sum_{i=1}^m L_i^2$.
\end{lemma}

\begin{lemma}
\label{lemma:sum_prod}
Consider a sequence $q_1, \dots, q_n \in [0, 1]$. Then
\begin{align*}
    1 - \sum_{m=1}^n q_m \leq \prod_{m=1}^n \left(1  - q_m\right).
\end{align*}
\end{lemma}
\begin{proof}
We prove the result by induction. It is clearly true for $n = 1$: $1 - \sum_{m=1}^{1} q_m =\prod_{m=1}^1 \left(1  - q_m\right).$ Now, assume that that it holds for $n - 1$, meaning that
\begin{align*}
1 - \sum_{m=1}^{n-1} q_m \leq \prod_{m=1}^{n-1} \left(1  - q_m\right).
\end{align*}
Multiplying both sides of the inequality by $1 - q_n \in [0, 1]$ gives
\begin{align*}
\prod_{m=1}^{n} \left(1  - q_m\right) \geq \left(1 - q_n\right)\left(1 - \sum_{m=1}^{n-1} q_m\right) = 1 - \sum_{m=1}^{n-1} q_m - q_n + q_n \left(\sum_{m=1}^{n-1} q_m\right) \geq 1 - \sum_{m=1}^{n} q_m
\end{align*}
since $q_m \in [0, 1]$ for all $m \in [n].$
\end{proof}

\begin{theorem}
    \label{eq:simple_upper_bounds_for_eq_time}
    Let us consider the equilibrium time mapping from Definition~\ref{def:time_budget}. Then
    \begin{enumerate}
        \item $\Teq(S, [\tau_i]_{i=1}^{n}) \leq 2 \tau_n \max\left\{\frac{S}{n}, 1\right\}.$ \\ 
        \item $\Teq(S, [\tau_i]_{i=1}^{n}) \leq 2 \tau_1 \max\left\{S, 1\right\}$
    \end{enumerate}
    for all $S \geq 0,$ $\tau_i \in [0, \infty]$ for all $i \in [n],$ and $\tau_1 \leq \dots \leq \tau_n.$
\end{theorem}

\begin{remark}
    \label{remark:better}
    Assume that $\tau_1 = \dots = \tau_{n - 1} = \tau$ and $\tau_n \to \infty,$ then
    \begin{align*}
        \lim\limits_{\tau_n \to \infty} \Teq(S, [\tau_i]_{i=1}^{n}) = \Teq(S, [\tau_i]_{i=1}^{n - 1}) 
        = \tau\parens{\frac{S}{n - 1} + 1} \leq 2 \tau \max\left\{\frac{S}{n - 1}, 1\right\},
    \end{align*}
    \begin{align*}
        \lim\limits_{\tau_n \to \infty} 2 \tau_n \max\left\{\frac{S}{n}, 1\right\} = \infty,
    \end{align*}
    and
    \begin{align*}
        \lim\limits_{\tau_n \to \infty} 2 \tau_1 \max\left\{S, 1\right\} = 2 \tau \max\left\{S, 1\right\}.
    \end{align*}
    Thus, $\Teq(S, [\tau_i]_{i=1}^{n})$ can be arbitrarily smaller than $2 \tau_n \max\left\{\frac{S}{n}, 1\right\}$ and $2 \tau_1 \max\left\{S, 1\right\}.$ This implies that our new complexities \eqref{eq:best_time_comp} and \eqref{eq:time_comp} can be arbitrarily better than $T_{\textnormal{Soviet PAGE}}$ and $T_{\textnormal{Hero PAGE}}.$
\end{remark}

\begin{lemma}\label{lemma:techn1}
    Consider a sequence $0<\tau_1\leq\ldots\leq\tau_n$ and fix some $S>0$. For all $j \in [n]$, define $$g(j) \eqdef \left(\sum_{i=1}^j \frac{1}{\tau_i}\right)^{-1} \left(S + j\right).$$ 
    \begin{enumerate}
        \item Let $j^*_{\max}$ be the largest index such that $\min\limits_{j \in [n]} g(j) = g(j^*_{\max}).$ For $j^*_{\max} < n,$ we have
        \begin{align*}
            \min_{j \in [n]} g(j) < \tau_{(j^*_{\max} + 1)}.
        \end{align*}
        \item Let $j^*$ be any index such that $\min\limits_{j \in [n]} g(j) = g(j^*).$ For $j^* < n,$ we have
        \begin{align*}
            \min_{j \in [n]} g(j) \leq \tau_{(j^* + 1)}.
        \end{align*}
        \item Let $j^*_{\min}$ be the smallest index such that $\min\limits_{j \in [n]} g(j) = g(j^*_{\min})$. Then
        \begin{align*}
            \tau_{j^*_{\min}} < \min_{j \in [n]} g(j).
        \end{align*}
        \item Let $j^*$ be any index such that $\min\limits_{j \in [n]} g(j) = g(j^*).$ Then
        \begin{align*}
            \tau_{j^*} \leq \min_{j \in [n]} g(j).
        \end{align*}
    \end{enumerate}
\end{lemma}

\begin{proof}
    We first prove the first inequality.
    Suppose that $j^*_{\max} < n$. Then $g(j^*_{\max}) < g(j^*_{\max} + 1)$, meaning that
    \begin{align*}
        \left(\sum_{i=1}^{j^*_{\max}} \frac{1}{\tau_i}\right)^{-1} \left(S + j^*_{\max}\right) < \left(\sum_{i=1}^{j^*_{\max} + 1} \frac{1}{\tau_i}\right)^{-1} \left(S + j^*_{\max} + 1\right).
    \end{align*}
    This implies the following series of inequalities:
    \begin{align*}
        &\left(\sum_{i=1}^{j^*_{\max} + 1} \frac{1}{\tau_i}\right) \left(S + j^*_{\max}\right) < \left(\sum_{i=1}^{j^*_{\max}} \frac{1}{\tau_i}\right) \left(S + j^*_{\max} + 1\right) \\
        &\Leftrightarrow\quad \frac{1}{\tau_{(j^*_{\max} + 1)}} \left(S + j^*_{\max}\right) < \sum_{i=1}^{j^*_{\max}} \frac{1}{\tau_i} \\
        &\Leftrightarrow\quad \tau_{(j^*_{\max} + 1)} > \left(\sum_{i=1}^{j^*_{\max}} \frac{1}{\tau_i}\right)^{-1}\left(S + j^*_{\max}\right) = g(j^*_{\max}) = \min_{j \in [n]} g(j).
    \end{align*}
    The proof of the second inequality is the same, but with non-strict inequalities.
    To prove the third inequality, first suppose that $j^*_{\min} = 1$. Then $g(j^*_{\min}) = \tau_{j^*_{\min}} \left(S + 1\right) > \tau_{j^*_{\min}}$ as needed. On the other hand, $j^*_{\min} > 1$ implies
    \begin{align*}
        &g(j^*_{\min}) < g(j^*_{\min} - 1) \\
        &\Leftrightarrow \left(\sum_{i=1}^{j^*_{\min}} \frac{1}{\tau_i}\right)^{-1} \left(S + j^*_{\min}\right) < \left(\sum_{i=1}^{j^*_{\min} - 1} \frac{1}{\tau_i}\right)^{-1} \left(S + j^*_{\min} - 1\right) \\
        &\Leftrightarrow\quad \left(\sum_{i=1}^{j^*_{\min} - 1} \frac{1}{\tau_i}\right) \left(S + j^*_{\min}\right) < \left(\sum_{i=1}^{j^*_{\min}} \frac{1}{\tau_i}\right) \left(S + j^*_{\min} - 1\right) \\
        &\Leftrightarrow\quad \sum_{i=1}^{j^*_{\min}} \frac{1}{\tau_i} < \frac{1}{\tau_{j^*_{\min}}} \left(S + j^*_{\min}\right) \\
        &\Leftrightarrow\quad \tau_{j^*_{\min}} < \left(\sum_{i=1}^{j^*_{\min}} \frac{1}{\tau_i}\right)^{-1} \left(S + j^*_{\min}\right) = \min_{j \in [n]} g(j).
    \end{align*}
    The proof of the fourth inequality is the same, but with non-strict inequalities.
\end{proof}

\newpage

\section*{NeurIPS Paper Checklist}

\begin{enumerate}

\item {\bf Claims}
    \item[] Question: Do the main claims made in the abstract and introduction accurately reflect the paper's contributions and scope?
    \item[] Answer: \answerYes{} % Replace by \answerYes{}, \answerNo{}, or \answerNA{}.
    \item[] Justification: Sections~\ref{sec:theory} and \ref{sec:lower_bound_simple}, Table~\ref{tbl:main}

\item {\bf Limitations}
    \item[] Question: Does the paper discuss the limitations of the work performed by the authors?
    \item[] Answer: \answerYes{} % Replace by \answerYes{}, \answerNo{}, or \answerNA{}.
    \item[] Justification: Sections~\ref{sec:design}, \ref{sec:opt_params}, and \ref{sec:lower_bound_simple}
    \item[] Guidelines:
    \begin{itemize}
        \item The answer NA means that the paper has no limitation while the answer No means that the paper has limitations, but those are not discussed in the paper. 
        \item The authors are encouraged to create a separate ''Limitations'' section in their paper.
        \item The paper should point out any strong assumptions and how robust the results are to violations of these assumptions (e.g., independence assumptions, noiseless settings, model well-specification, asymptotic approximations only holding locally). The authors should reflect on how these assumptions might be violated in practice and what the implications would be.
        \item The authors should reflect on the scope of the claims made, e.g., if the approach was only tested on a few datasets or with a few runs. In general, empirical results often depend on implicit assumptions, which should be articulated.
        \item The authors should reflect on the factors that influence the performance of the approach. For example, a facial recognition algorithm may perform poorly when image resolution is low or images are taken in low lighting. Or a speech-to-text system might not be used reliably to provide closed captions for online lectures because it fails to handle technical jargon.
        \item The authors should discuss the computational efficiency of the proposed algorithms and how they scale with dataset size.
        \item If applicable, the authors should discuss possible limitations of their approach to address problems of privacy and fairness.
        \item While the authors might fear that complete honesty about limitations might be used by reviewers as grounds for rejection, a worse outcome might be that reviewers discover limitations that aren't acknowledged in the paper. The authors should use their best judgment and recognize that individual actions in favor of transparency play an important role in developing norms that preserve the integrity of the community. Reviewers will be specifically instructed to not penalize honesty concerning limitations.
    \end{itemize}

\item {\bf Theory Assumptions and Proofs}
    \item[] Question: For each theoretical result, does the paper provide the full set of assumptions and a complete (and correct) proof?
    \item[] Answer: \answerYes{} % Replace by \answerYes{}, \answerNo{}, or \answerNA{}.
    \item[] Justification: Sections~\ref{sec:assumptions}, \ref{sec:theory}, and the appendix

    \item {\bf Experimental Result Reproducibility}
    \item[] Question: Does the paper fully disclose all the information needed to reproduce the main experimental results of the paper to the extent that it affects the main claims and/or conclusions of the paper (regardless of whether the code and data are provided or not)?
    \item[] Answer: \answerYes{} % Replace by \answerYes{}, \answerNo{}, or \answerNA{}.
    \item[] Justification: Section~\ref{sec:experiments}

\item {\bf Open access to data and code}
    \item[] Question: Does the paper provide open access to the data and code, with sufficient instructions to faithfully reproduce the main experimental results, as described in supplemental material?
    \item[] Answer: \answerYes{} % Replace by \answerYes{}, \answerNo{}, or \answerNA{}.
    \item[] Justification: In the supplementary materials.

\item {\bf Experimental Setting/Details}
    \item[] Question: Does the paper specify all the training and test details (e.g., data splits, hyperparameters, how they were chosen, type of optimizer, etc.) necessary to understand the results?
    \item[] Answer: \answerYes{} % Replace by \answerYes{}, \answerNo{}, or \answerNA{}.
    \item[] Justification: Section~\ref{sec:experiments}

\item {\bf Experiment Statistical Significance}
    \item[] Question: Does the paper report error bars suitably and correctly defined or other appropriate information about the statistical significance of the experiments?
    \item[] Answer: \answerYes{} % Replace by \answerYes{}, \answerNo{}, or \answerNA{}.
    \item[] Justification: We provide two plots for each algorithm to ensure the soundness of our experiments. We also calculate the variance of accuracies in Table~\ref{tbl:acc}.
    \item[] Guidelines:
    \begin{itemize}
        \item The answer NA means that the paper does not include experiments.
        \item The authors should answer ''Yes'' if the results are accompanied by error bars, confidence intervals, or statistical significance tests, at least for the experiments that support the main claims of the paper.
        \item The factors of variability that the error bars are capturing should be clearly stated (for example, train/test split, initialization, random drawing of some parameter, or overall run with given experimental conditions).
        \item The method for calculating the error bars should be explained (closed form formula, call to a library function, bootstrap, etc.)
        \item The assumptions made should be given (e.g., Normally distributed errors).
        \item It should be clear whether the error bar is the standard deviation or the standard error of the mean.
        \item It is OK to report 1-sigma error bars, but one should state it. The authors should preferably report a 2-sigma error bar than state that they have a 96\% CI, if the hypothesis of Normality of errors is not verified.
        \item For asymmetric distributions, the authors should be careful not to show in tables or figures symmetric error bars that would yield results that are out of range (e.g. negative error rates).
        \item If error bars are reported in tables or plots, The authors should explain in the text how they were calculated and reference the corresponding figures or tables in the text.
    \end{itemize}

\item {\bf Experiments Compute Resources}
    \item[] Question: For each experiment, does the paper provide sufficient information on the computer resources (type of compute workers, memory, time of execution) needed to reproduce the experiments?
    \item[] Answer: \answerYes{} % Replace by \answerYes{}, \answerNo{}, or \answerNA{}.
    \item[] Justification: Section~\ref{sec:experiments}
    
\item {\bf Code Of Ethics}
    \item[] Question: Does the research conducted in the paper conform, in every respect, with the NeurIPS Code of Ethics \url{https://neurips.cc/public/EthicsGuidelines}?
    \item[] Answer: \answerYes{} % Replace by \answerYes{}, \answerNo{}, or \answerNA{}.
    \item[] Justification: We have thoroughly reviewed the code of ethics and we can confidently state that our paper is fully compliant with it.

\item {\bf Broader Impacts}
    \item[] Question: Does the paper discuss both potential positive societal impacts and negative societal impacts of the work performed?
    \item[] Answer: \answerNA{} % Replace by \answerYes{}, \answerNo{}, or \answerNA{}.
    \item[] Justification: Our paper considers a mathematical topic.
    
\item {\bf Safeguards}
    \item[] Question: Does the paper describe safeguards that have been put in place for responsible release of data or models that have a high risk for misuse (e.g., pretrained language models, image generators, or scraped datasets)?
    \item[] Answer: \answerNA{} % Replace by \answerYes{}, \answerNo{}, or \answerNA{}.
    \item[] Justification:

\item {\bf Licenses for existing assets}
    \item[] Question: Are the creators or original owners of assets (e.g., code, data, models), used in the paper, properly credited and are the license and terms of use explicitly mentioned and properly respected?
    \item[] Answer: \answerYes{} % Replace by \answerYes{}, \answerNo{}, or \answerNA{}.
    \item[] Justification: Section~\ref{sec:experiments}

\item {\bf New Assets}
    \item[] Question: Are new assets introduced in the paper well documented and is the documentation provided alongside the assets?
    \item[] Answer: \answerYes{} % Replace by \answerYes{}, \answerNo{}, or \answerNA{}.
    \item[] Justification: In the supplementary materials.

\item {\bf Crowdsourcing and Research with Human Subjects}
    \item[] Question: For crowdsourcing experiments and research with human subjects, does the paper include the full text of instructions given to participants and screenshots, if applicable, as well as details about compensation (if any)? 
    \item[] Answer: \answerNA{} % Replace by \answerYes{}, \answerNo{}, or \answerNA{}.
    \item[] Justification: 

\item {\bf Institutional Review Board (IRB) Approvals or Equivalent for Research with Human Subjects}
    \item[] Question: Does the paper describe potential risks incurred by study participants, whether such risks were disclosed to the subjects, and whether Institutional Review Board (IRB) approvals (or an equivalent approval/review based on the requirements of your country or institution) were obtained?
    \item[] Answer: \answerNA{} % Replace by \answerYes{}, \answerNo{}, or \answerNA{}.
    \item[] Justification: 

\end{enumerate}

\end{document}